\documentclass{article}
\usepackage{amsmath}
\usepackage{amssymb}
\usepackage{amsthm}
\usepackage{cite}
\usepackage{tikz}
\usepackage{stmaryrd}
\SetSymbolFont{stmry}{bold}{U}{stmry}{m}{n}
\usepackage{slashed} 
\usepackage{euscript}
\usepackage[arrow,curve,matrix,arc,2cell]{xy}
\usepackage[utf8]{inputenc}
\usepackage[unicode]{hyperref}
\usepackage{mathtools}
\usepackage{enumitem}
\DeclareFontFamily{U}{rsfs}{} 
\DeclareFontShape{U}{rsfs}{n}{it}{<->
rsfs10}{} \DeclareSymbolFont{mscr}{U}{rsfs}{n}{it}
\DeclareSymbolFontAlphabet{\scr}{mscr}
\def\mathscr{\scr}
\newcommand{\leftrarrows}{\mathrel{\raise.75ex\hbox{\oalign{%
  $\scriptstyle\leftarrow$\cr
  \vrule width0pt height.5ex$\hfil\scriptstyle\relbar$\cr}}}}
\newcommand{\lrightarrows}{\mathrel{\raise.75ex\hbox{\oalign{%
  $\scriptstyle\relbar$\hfil\cr
  $\scriptstyle\vrule width0pt height.5ex\smash\rightarrow$\cr}}}}
\newcommand{\Rrelbar}{\mathrel{\raise.75ex\hbox{\oalign{%
  $\scriptstyle\relbar$\cr
  \vrule width0pt height.5ex$\scriptstyle\relbar$}}}}
\newcommand{\longleftrightarrows}{\leftrarrows\joinrel\Rrelbar\joinrel\lrightarrows}
\makeatletter
\def\leftrightarrowsfill@{\arrowfill@\leftrarrows\Rrelbar\lrightarrows}
\newcommand{\xleftrightarrows}[2][]{\ext@arrow 3399\leftrightarrowsfill@{#1}{#2}}
\makeatother
\begin{document}
\def\e#1\e{\begin{equation}#1\end{equation}}
\def\ea#1\ea{\begin{align}#1\end{align}}
\def\eq#1{{\rm(\ref{#1})}}
\theoremstyle{plain}
\newtheorem{thm}{Theorem}[section]
\newtheorem{lem}[thm]{Lemma}
\newtheorem{prop}[thm]{Proposition}
\newtheorem{cor}[thm]{Corollary}
\newtheorem{quest}[thm]{Question}
\theoremstyle{definition}
\newtheorem{dfn}[thm]{Definition}
\newtheorem{ex}[thm]{Example}
\newtheorem{rem}[thm]{Remark}
\numberwithin{equation}{section}
\numberwithin{figure}{section}

\def\incl{\mathop{\rm incl}\nolimits}
\def\dim{\mathop{\rm dim}\nolimits}
\def\codim{\mathop{\rm codim}\nolimits}
\def\Im{\mathop{\rm Im}\nolimits}
\def\det{\mathop{\rm det}\nolimits}
\def\Ker{\mathop{\rm Ker}}
\def\Coker{\mathop{\rm Coker}}
\def\Flag{\mathop{\rm Flag}\nolimits}
\def\FlagSt{\mathop{\rm FlagSt}\nolimits}
\def\Iso{\mathop{\rm Iso}\nolimits}
\def\Aut{\mathop{\rm Aut}}
\def\End{\mathop{\rm End}\nolimits}
\def\Spec{\mathop{\rm Spec}\nolimits}
\def\Mot{\mathop{\rm Mot}\nolimits}
\def\Perf{{\rm Perf}}
\def\DSch{\mathop{\bf DSch}\nolimits}
\def\Sch{\mathop{\bf Sch}\nolimits}
\def\Art{\mathop{\bf Art}\nolimits}
\def\HSta{\mathop{\bf HSta}\nolimits}
\def\DSta{\mathop{\bf DSta}\nolimits}
\def\DArt{\mathop{\bf DArt}\nolimits}
\def\Ho{\mathop{\rm Ho}}
\def\GL{\mathop{\rm GL}}
\def\SL{\mathop{\rm SL}}
\def\SO{\mathop{\rm SO}}
\def\SU{\mathop{\rm SU}}
\def\Sp{\mathop{\rm Sp}}
\def\Spin{\mathop{\rm Spin}\nolimits}
\def\U{{\mathbin{\rm U}}}
\def\vol{\mathop{\rm vol}}
\def\inc{\mathop{\rm inc}}
\def\semi{{\rm semi}}
\def\cla{{\rm cla}}
\def\ran{{\rm an}}
\def\top{{\rm top}}
\def\ind{{\rm ind}}
\def\irr{{\rm irr}}
\def\red{{\rm red}}
\def\db{\bar\partial}
\def\rk{\mathop{\rm rk}}
\def\Crit{\mathop{\rm Crit}\nolimits}
\def\rank{\mathop{\rm rank}\nolimits}
\def\Hom{\mathop{\rm Hom}\nolimits}
\def\id{{\mathop{\rm id}\nolimits}}
\def\Id{{\mathop{\rm Id}\nolimits}}
\def\TopSta{{\mathop{\bf TopSta}\nolimits}}
\def\Map{{\mathop{\rm Map}\nolimits}}
\def\Ad{\mathop{\rm Ad}}
\def\irr{{\rm irr}}
\def\Ext{\mathop{\rm Ext}\nolimits}
\def\coh{{\rm coh}}
\def\qcoh{{\rm qcoh}}
\def\cs{{\rm cs}}
\def\Top{{\mathop{\bf Top}\nolimits}}
\def\Stab{\mathop{\rm Stab}\nolimits}
\def\ul{\underline}
\def\bs{\boldsymbol}
\def\ge{\geqslant}
\def\le{\leqslant\nobreak}
\def\boo{{\mathbin{\mathbf 1}}}
\def\cDist{\mathop{{\mathcal D}\kern -.1em\mathit{ist}}\nolimits}
\def\bcDist{\mathop{\boldsymbol{{\mathcal D}\kern -.1em\mathit{ist}}}\nolimits}
\def\cExact{\mathop{{\mathcal E}\kern -.1em\mathit{xact}}\nolimits}
\def\bcExact{\mathop{\boldsymbol{{\mathcal E}\kern -.1em\mathit{xact}}}\nolimits}
\def\O{{\mathcal O}}
\def\bA{{\mathbin{\mathbb A}}}
\def\bG{{\mathbin{\mathbb G}}}
\def\bH{{\mathbin{\mathbb H}}}
\def\bL{{\mathbin{\mathbb L}}}
\def\P{{\mathbin{\mathbb P}}}
\def\bT{{\mathbin{\mathbb T}}}
\def\K{{\mathbin{\mathbb K}}}
\def\R{{\mathbin{\mathbb R}}}
\def\Z{{\mathbin{\mathbb Z}}}
\def\Q{{\mathbin{\mathbb Q}}}
\def\N{{\mathbin{\mathbb N}}}
\def\C{{\mathbin{\mathbb C}}}
\def\CP{{\mathbin{\mathbb{CP}}}}
\def\bR{{\bs R}}
\def\bS{{\bs S}}
\def\bU{{\bs U}}
\def\bV{{\bs V}}
\def\bW{{\bs W}}
\def\bX{{\bs X}}
\def\bY{{\bs Y}}
\def\bZ{{\bs Z}}
\def\A{{\mathbin{\cal A}}}
\def\B{{\mathbin{\cal B}}}
\def\ovB{{\mathbin{\smash{\,\overline{\!\mathcal B}}}}}\def\G{{\mathbin{\cal G}}}
\def\M{{\mathbin{\cal M}}}
\def\cM{{\mathbin{\cal M}}}
\def\oM{{\mathbin{\,\,\overline{\!\!\mathcal M\!}\,}}}
\def\bcM{{\mathbin{\bs{\cal M}}}}
\def\boM{{\mathbin{\boldsymbol{\,\,\overline{\!\!\mathcal M\!}\,}}}}
\def\cB{{\mathbin{\cal B}}}
\def\cC{{\mathbin{\cal C}}}
\def\cD{{\mathbin{\cal D}}}
\def\cE{{\mathbin{\cal E}}}
\def\cF{{\mathbin{\cal F}}}
\def\cG{{\mathbin{\cal G}}}
\def\cH{{\mathbin{\cal H}}}
\def\E{{\mathbin{\cal E}}}
\def\F{{\mathbin{\cal F}}}
\def\cG{{\mathbin{\cal G}}}
\def\cH{{\mathbin{\cal H}}}
\def\cI{{\mathbin{\cal I}}}
\def\cJ{{\mathbin{\cal J}}}
\def\cK{{\mathbin{\cal K}}}
\def\cL{{\mathbin{\cal L}}}
\def\cN{{\mathbin{\cal N}\kern .04em}}
\def\cP{{\mathbin{\cal P}}}
\def\cQ{{\mathbin{\cal Q}}}
\def\cR{{\mathbin{\cal R}}}
\def\cS{{\mathbin{\cal S}}}
\def\T{{{\cal T}\kern .04em}}
\def\cW{{\mathbin{\cal W}}}
\def\cU{{\cal U}}
\def\cV{{\cal V}}
\def\cX{{\cal X}}
\def\cY{{\cal Y}}
\def\cZ{{\cal Z}}
\def\cV{{\cal V}}
\def\cW{{\cal W}}
\def\g{{\mathfrak g}}
\def\h{{\mathfrak h}}
\def\m{{\mathfrak m}}
\def\u{{\mathfrak u}}
\def\su{{\mathfrak{su}}}
\def\al{\alpha}
\def\be{\beta}
\def\ga{\gamma}
\def\de{\delta}
\def\io{\iota}
\def\ep{\epsilon}
\def\la{\lambda}
\def\ka{\kappa}
\def\th{\theta}
\def\ze{\zeta}
\def\up{\upsilon}
\def\vp{\varphi}
\def\si{\sigma}
\def\om{\omega}
\def\De{\Delta}
\def\Ka{{\rm K}}
\def\La{\Lambda}
\def\Om{\Omega}
\def\Ga{\Gamma}
\def\Si{\Sigma}
\def\Th{\Theta}
\def\Up{\Upsilon}
\def\Chi{{\rm X}}
\def\Tau{{T}}
\def\Nu{{\rm N}}
\def\pd{\partial}
\def\ts{\textstyle}
\def\st{\scriptstyle}
\def\sst{\scriptscriptstyle}
\def\w{\wedge}
\def\sm{\setminus}
\def\lt{\ltimes}
\def\bu{\bullet}
\def\sh{\sharp}
\def\di{\diamond}
\def\he{\heartsuit}
\def\op{\oplus}
\def\od{\odot}
\def\ot{\otimes}
\def\bt{\boxtimes}
\def\ov{\overline}
\def\bigop{\bigoplus}
\def\bigot{\bigotimes}
\def\iy{\infty}
\def\es{\emptyset}
\def\ra{\rightarrow}
\def\rra{\rightrightarrows}
\def\Ra{\Rightarrow}
\def\Longra{\Longrightarrow}
\def\ab{\allowbreak}
\def\longra{\longrightarrow}
\def\hookra{\hookrightarrow}
\def\dashra{\dashrightarrow}
\def\lb{\llbracket}
\def\rb{\rrbracket}
\def\ha{{\ts\frac{1}{2}}}
\def\t{\times}
\def\ci{\circ}
\def\ti{\tilde}
\def\d{{\rm d}}
\def\md#1{\vert #1 \vert}
\def\ms#1{\vert #1 \vert^2}
\def\bmd#1{\big\vert #1 \big\vert}
\def\bms#1{\big\vert #1 \big\vert^2}
\def\an#1{\langle #1 \rangle}
\def\ban#1{\bigl\langle #1 \bigr\rangle}
\def\o{\operatorname{o}}
\def\O{\mathcal{O}}
\title{Orientation data for moduli spaces of coherent sheaves over Calabi--Yau $3$-folds}
\author{Dominic Joyce and Markus Upmeier}
\date{}
\maketitle

\begin{quotation}
`If orientation data does not exist, then String Theory is dead.'

--- Maxim Kontsevich, at a conference in Budapest, May 2012.	
\end{quotation}

\medskip

\begin{abstract}
Let $X$ be a compact Calabi--Yau 3-fold, and write $\M,\oM$ for the moduli stacks of objects in $\coh(X),D^b\coh(X)$. There are natural line bundles $K_\M\ra\M$, $K_\oM\ra\oM$, analogues of canonical bundles. {\it Orientation data\/} on $\M,\oM$ is an isomorphism class of square root line bundles $K_\M^{1/2},K_\oM^{1/2}$, satisfying a compatibility condition on the stack of short exact sequences. It was introduced by Kontsevich and Soibelman \cite[\S 5]{KoSo1} in their theory of motivic Donaldson--Thomas invariants, and is also important in categorifying Donaldson--Thomas theory using perverse sheaves.

We show that natural orientation data can be constructed for all compact Calabi--Yau 3-folds $X$, and also for compactly-supported coherent sheaves and perfect complexes on noncompact Calabi--Yau 3-folds $X$ that admit a spin smooth projective compactification $X\hookra Y$. This proves a long-standing conjecture in Donaldson--Thomas theory.

These are special cases of a more general result. Let $X$ be a spin smooth projective 3-fold. Using the spin structure we construct line bundles $K_\M\ra\M$, $K_\oM\ra\oM$. We define {\it spin structures\/} on $\M,\oM$ to be isomorphism classes of square roots $K_\M^{1/2},K_\oM^{1/2}$. We prove that natural spin structures exist on $\M,\oM$. They are equivalent to orientation data when $X$ is a Calabi--Yau 3-fold with the trivial spin structure.

We prove this using our previous paper \cite{JoUp2}, which constructs `spin structures' (square roots of a certain complex line bundle $K_P^{E_\bu}\ra\B_P$) on differential-geometric moduli stacks $\B_P$ of connections on a principal $\U(m)$-bundle $P\ra X$ over a compact spin 6-manifold~$X$.
\end{abstract}

\setcounter{tocdepth}{2}
\tableofcontents

\section{Introduction}
\label{od1}

Let $X$ be a (compact) {\it Calabi--Yau\/ $3$-fold}, that is, a smooth projective $\C$-scheme of dimension 3 with trivial canonical bundle $K_X\cong\O_X$. Write $\M$ for the moduli stack of objects in the category of coherent sheaves $\coh(X)$, an Artin $\C$-stack, and $\oM$ for the moduli stack of objects in the derived category $D^b\coh(X)$, a higher $\C$-stack. Then $\M\subset\oM$ is an open substack.

The {\it Donaldson--Thomas invariants\/} $DT^\al(\tau)$ of a Calabi--Yau 3-fold $X$ are integers or rational numbers `counting' open substacks $\M^\al_{\rm st}(\tau)\subset\M^\al_{\rm ss}(\tau)\subset\M$ of $\tau$-(semi)stable coherent sheaves on $X$ with Chern character $\al$, for $\tau$ a suitable stability condition on $\coh(X)$. They are unchanged under deformations of $X$. They were proposed by Donaldson and Thomas \cite{DoTh}, and defined by Thomas \cite{Thom1} when $\M^\al_{\rm st}(\tau)=\M^\al_{\rm ss}(\tau)$, and by Joyce and Song \cite{JoSo} in the general case. They are important in String Theory as `numbers of BPS states'.

Donaldson--Thomas theory may be generalized in two directions, discussed in \S\ref{od43}. Firstly, one can refine $DT^\al(\tau)$ to `motivic Donaldson--Thomas invariants' $DT^\al_{\rm mot}(\tau)$ in a ring $\Mot$ rather than in $\Z$ or $\Q$, as in Kontsevich and Soibelman \cite{KoSo1,KoSo2}. And secondly, as in \cite{BBDJS,BBBJ} one can construct a perverse sheaf $P^\bu_{\M^\al_{\rm ss}(\tau)}$ on $\M^\al_{\rm ss}(\tau)$, with $DT^\al(\tau)=\sum_{i\in\Z}(-1)^i\dim\bH^i(P^\bu_{\M^\al_{\rm ss}(\tau)})$. The hypercohomology $\bH^*(P^\bu_{\M^\al_{\rm ss}(\tau)})$ is understood in String Theory as the `vector space of BPS states'.

Both of these generalizations require an extra structure on $\M$ or $\oM$ called {\it orientation data}, introduced by Kontsevich and Soibelman \cite[\S 5]{KoSo1}. There are natural line bundles $K_\M\ra\M$,  $K_\oM\ra\oM$, regarded as `canonical bundles' of $\M,\oM$, and orientation data is a choice of isomorphism class of square root line bundle $K_\M^{1/2}$ or $K_\oM^{1/2}$. For the $DT^\al_{\rm mot}(\tau)$ to satisfy multiplicative identities, these $K_\M^{1/2}$ or $K_\oM^{1/2}$ must satisfy compatibilities on the stacks $\cExact$ of short exact sequences in $\coh(X)$, or $\cDist$ of distinguished triangles in~$D^b\coh(X)$.

The goal of this paper is to prove, in Theorem \ref{od4thm1} below, that there are natural choices of orientation data on $\M$ and $\oM$ for all compact Calabi--Yau 3-folds $X$. We also give, in Theorem \ref{od4thm2}, a sufficient condition for the existence of natural orientation data for compactly-supported coherent sheaves and perfect complexes on noncompact (quasi-projective) Calabi--Yau 3-folds~$X$. 

This proves a long-standing conjecture in Donaldson--Thomas theory. So far as the authors are aware, orientation data was not known to exist for any compact Calabi--Yau 3-fold until now.

If $Y$ is a complex manifold then {\it spin structures\/} on $Y$ in the usual sense of differential geometry are in natural 1-1 correspondence with square roots $K_Y^{1/2}$ of the canonical bundle $K_Y$ of $Y$, where the corresponding spin bundle is $\slashed{S}=K_Y^{1/2}\ot\La^{0,*}T^*Y$. Because of this, the authors feel that `spin structure' on $\M,\oM$ would be a better term than `orientation data', but the latter is already established in the Donaldson--Thomas theory literature.

To prove Theorems \ref{od4thm1} and \ref{od4thm2}, we first study a more general problem. Let $X$ be a smooth projective 3-fold with a choice of spin structure $K_X^{1/2}$. Then we can again define natural line bundles $K_\M\ra\M$, $K_\oM\ra\oM$ using $K_X^{1/2}$. A {\it spin structure\/} on $\M$ or $\oM$ is a choice of isomorphism class of square roots $K_\M^{1/2}$ or $K_\oM^{1/2}$. We call a spin structure {\it compatible with direct sums\/} if it satisfies a condition involving direct sums in $\coh(X)$ or~$D^b\coh(X)$.

If $X$ is a Calabi--Yau 3-fold then it has a natural spin structure. The line bundles $K_\M,K_\oM$ reduce to the previous ones, and the compatibility conditions over exact sequences and direct sums are equivalent. Thus `orientation data' agrees with `spin structures compatible with direct sums' for Calabi--Yau 3-folds.

We show in Theorem \ref{od3thm2} that if $X$ is any spin smooth projective 3-fold, then there are natural choices of spin structures on $\M,\oM$ compatible with direct sums. When $X$ is a compact Calabi--Yau 3-fold this implies Theorem \ref{od4thm1}. If $X$ is a noncompact Calabi--Yau 3-fold with an open spin inclusion $X\hookra Y$ into a spin smooth projective 3-fold $Y$, Theorem \ref{od3thm2} for $Y$ implies Theorem \ref{od4thm2} for~$X$.

Theorem \ref{od3thm2} is proved using the main result \cite[Th.~5.12]{JoUp2} of a previous paper by the authors, stated as Theorem \ref{od2thm1} below. In \cite{JoUp2}, given a compact manifold $X$ and a complex elliptic operator $E_\bu$ on $X$, for any principal $\U(m)$-bundle $P\ra X$ we define the moduli stack $\B_P$ of connections $\nabla_P$ on $P$, as a topological stack, and we construct a natural topological complex line bundle $K_P^{E_\bu}\ra\B_P$. A {\it spin structure\/} on $\B_P$ is an isomorphism class of square roots~$(K_P^{E_\bu})^{1/2}$. 

In \cite[Th.~5.12]{JoUp2} we show that if $X$ is a compact spin 6-manifold and $E_\bu$ is the positive Dirac operator on $X$, then there are natural spin structures on $\B_P$ for all $\U(m)$-bundles $P\ra X$, all $m\ge 0$, with a compatibility under direct sums $P_1\op P_2$. This is a differential-geometric version of orientation data. 

To prove Theorem \ref{od3thm2}, we form the open substack $\M_{\rm vect}\subset\M$ of algebraic vector bundles on the spin smooth projective 3-fold $X$. Roughly speaking, there is a natural map $(\M_{\rm vect})^\top\ra\coprod_{\text{iso. classes $[P]$}}\B_P$. Using results of Cao, Gross and Joyce \cite{CGJ}, we show we can pull back the spin structures on $\B_P$ for all $P$ to a spin structure on $\M_{\rm vect}$, and then extend this to spin structures on~$\M,\oM$.

In \S\ref{od33} we also introduce a notion of {\it strong spin structure\/} on $\M,\oM$, and show in Theorem \ref{od3thm3} that strong spin structures on $\oM$ compatible with direct sums are controlled by classes in the group cohomology $H^*\bigl(K_0^\semi(X),\Z_2\bigr)$ of the semi-topological K-theory group $K_0^\semi(X)$. For Calabi--Yau 3-folds, such strong spin structures are equivalent to {\it strong orientation data}, which is important in the categorification of Donaldson--Thomas theory using perverse sheaves.

We begin in \S\ref{od2} by summarizing the results of \cite{JoUp2} on spin structures on connection moduli spaces $\B_P$. Section \ref{od3} studies spin structures on sheaf moduli spaces $\M,\oM$ for a spin smooth projective $m$-fold $X$ with $m$ odd. Section \ref{od4} restricts to $X$ Calabi--Yau, and relates spin structures and orientation data on $\M,\oM$. The proofs of Theorem \ref{od3thm2} and \ref{od3thm3} are deferred to~\S\ref{od5}--\S\ref{od6}.
\medskip

\noindent{\it Acknowledgements.} This research was partly funded by a Simons Collaboration Grant on `Special Holonomy in Geometry, Analysis and Physics'. The second author was partly funded by DFG grant UP 85/2-1 of the DFG priority program SPP 2026 `Geometry at Infinity'. The authors would like to thank Arkadij Bojko, Jacob Gross, Zheng Hua, Richard Thomas, and Yukinobu Toda for helpful conversations. 

\section{Spin structures on connection moduli spaces}
\label{od2}

We now explain some material from the authors' previous papers \cite{JTU,JoUp2} on connection moduli spaces $\B_P$ and `spin structures' upon them. Our first definition comes from \cite[Def.~1.1]{JTU} and~\cite[Def.~2.1]{JoUp2}.

\begin{dfn}
\label{od2def1}
Let $X$ be a compact manifold, and $P\ra X$ be a principal $\U(m)$-bundle for $m\ge 0$. (Our previous papers \cite{JTU,JoUp1,CGJ} discussed principal $G$-bundles $P\ra X$ for general Lie groups $G$, but in this paper we restrict to $G=\U(m)$.) We write $\Ad(P)\ra X$ for the vector bundle with fibre $\u(m)$ defined by $\Ad(P)=(P\t\u(m))/\U(m)$, where $\U(m)$ acts on $P$ by the principal bundle action, and on the Lie algebra $\u(m)$ of $\U(m)$ by the adjoint action.

Write $\A_P$ for the set of connections $\nabla_P$ on the principal bundle $P\ra X$. This is a real affine space modelled on the infinite-dimensional vector space $\Ga^\iy(\Ad(P)\ot T^*X)$, and we make $\A_P$ into a topological space using the $C^\iy$ topology on $\Ga^\iy(\Ad(P)\ot T^*X)$. Here if $E\ra X$ is a vector bundle then $\Ga^\iy(E)$ denotes the vector space of smooth sections of $E$. Note that $\A_P$ is contractible. 

Write $\G_P=\Aut(P)$ for the infinite-dimensional Lie group of $\U(m)$-equi\-var\-iant diffeomorphisms $\ga:P\ra P$ with $\pi\ci\ga=\pi$. Then $\G_P$ acts continuously on $\A_P$ by gauge transformations. Write $\B_P=[\A_P/\G_P]$ for the moduli space of gauge equivalence classes of connections on $P$, considered as a {\it topological stack\/} in the sense of Metzler \cite{Metz} and Noohi~\cite{Nooh1,Nooh2}.

Here $P\ra X$ has an {\it associated complex vector bundle\/} $F\ra X$ with fibre $\C^m$, given by $F=(P\t\C^m)/\U(m)$. There is a Hermitian metric $h_F$ on the fibres, induced by the Hermitian metric $h_{\C^m}$ on $\C^m$, and $P$ is the bundle of $\U(m)$-frames of $(F,h_F)$. There is a natural 1-1 correspondence between principal bundle connections $\nabla_P$ on $P$, and vector bundle connections $\nabla_F$ on $F$ preserving $h_F$. Thus we may also regard $\A_P,\B_P$ as moduli spaces of connections on~$F$.
\end{dfn}

We define direct sums $P_1\op P_2$ and morphisms $\Phi_{P_1,P_2}:\B_{P_1}\t\B_{P_2}\ra\B_{P_1\op P_2}$, following \cite[Ex.~2.11]{JTU} and~\cite[Ex.~2.9]{JoUp2}.

\begin{dfn}
\label{od2def2}
Let $X$ be a compact manifold, and $P_1\ra X$, $P_2\ra X$ be principal $\U(m_1)$- and $\U(m_2)$-bundles for $m_1,m_2\ge 0$. There is an inclusion $\U(m_1)\t\U(m_2)\hookra\U(m_1+m_2)$ mapping $(A,B)\mapsto\bigl(\begin{smallmatrix} A & 0 \\ 0 & B \end{smallmatrix}\bigr)$ for $A\in\U(m_1)$ and $B\in\U(m_2)$. Define a principal $\U(m_1+m_2)$-bundle $P_1\op P_2\ra X$ by 
\e
P_1\op P_2=(P_1\t_XP_2\t\U(m_1+m_2))/(\U(m_1)\t\U(m_2)),
\label{od2eq1}
\e
where $\U(m_1)\t\U(m_2)$ acts on $P_1\t_XP_2$ via the $\U(m_1)$- and $\U(m_2)$-actions on $P_1,P_2$, and on $\U(m_1+m_2)$ by the inclusion $\U(m_1)\t\U(m_2)\hookra\U(m_1+m_2)$ and the right action of $\U(m_1+m_2)$ on itself. We use the notation $P_1\op P_2$ as if $F_1\ra X$, $F_2\ra X$ are the associated complex vector bundles of $P_1,P_2$, then $F_1\op F_2$ is the associated complex vector bundle of $P_1\op P_2$.

Define a continuous map $\hat\Phi_{P_1,P_2}:\A_{P_1}\t\A_{P_2}\longra\A_{P_1\op P_2}$ by $(\nabla_{P_1},\nabla_{P_2})\mapsto\nabla_{P_1\op P_2}$, where $\nabla_{P_1\op P_2}$ is the connection induced on $P_1\op P_2$ by $\nabla_{P_1},\nabla_{P_2}$ using \eq{od2eq1}. Then $\hat\Phi_{P_1,P_2}$ is equivariant under $\G_{P_1}\t\G_{P_2}$ and the natural morphism $\G_{P_1}\t\G_{P_2}\ra\G_{P_1\op P_2}$. Thus $\hat\Phi_{P_1,P_2}$ descends to a morphism of topological stacks $\Phi_{P_1,P_2}:\B_{P_1}\t\B_{P_2}\longra\B_{P_1\op P_2}$ mapping $([\nabla_{P_1}],[\nabla_{P_2}])\mapsto[\nabla_{P_1\op P_2}]$ on points.
\end{dfn}

The next two definitions come from~\cite[Def.s 3.1 \& 3.2]{JoUp2}.

\begin{dfn}
\label{od2def3}
Let $X$ be a compact manifold. Suppose we are given complex vector bundles $E_0,E_1\ra X$, of the same rank $r$, and a complex linear elliptic partial differential operator $D:\Ga^\iy(E_0)\ra\Ga^\iy(E_1)$, of degree $d$. As a shorthand we write $E_\bu=(E_0,E_1,D)$. With respect to connections $\nabla_{E_0}$ on $E_0\ot\bigot^iT^*X$ for $0\le i<d$, when $e\in\Ga^\iy(E_0)$ we may write
\begin{equation*}
D(e)=\ts\sum_{i=0}^d a_i\cdot \nabla_{E_0}^ie,
\end{equation*}
where $a_i\in \Ga^\iy(E_0^*\ot E_1\ot S^iTX)$ for $i=0,\ldots,d$. The condition that $D$ is {\it elliptic\/} is that $a_d\vert_x\cdot\ot^d\xi:E_0\vert_x\ra E_1\vert_x$ is an isomorphism for all $x\in X$ and $0\ne\xi\in T_x^*X$, and the {\it symbol\/} $\si(D)$ of $D$ is defined using~$a_d$.

Now suppose we are given Hermitian metrics $h_{E_0},h_{E_1}$ (that is, Euclidean metrics on $E_0,E_1$ compatible with the complex structures) on the fibres of $E_0,E_1$, and a volume form $\d V$ on $X$. Then there is a unique {\it adjoint operator\/} $D^*:\Ga^\iy(E_1)\ra\Ga^\iy(E_0)$, which is also a complex linear elliptic partial differential operator of degree $d$, satisfying for all $e_0\in\Ga^\iy(E_0)$, $e_1\in\Ga^\iy(E_1)$
\begin{equation*}
\int_X h_{E_1}(D(e_0),e_1)\d V=\int_X h_{E_0}(e_0,D^*(e_1))\d V.
\end{equation*}
It is complex anti-linear in $D$, as $h_{E_0},h_{E_1}$ are Hermitian.

Write $\bar E_0,\bar E_1$ for the complex conjugate vector bundles of $E_0,E_1$ (the same real vector bundles, but the complex structures change sign), and $\bar D:\Ga^\iy(\bar E_0)\ra\Ga^\iy(\bar E_1)$ for the complex conjugate operator (as real vector spaces and operators $\Ga^\iy(\bar E_j)=\Ga^\iy(E_j)$ and $\bar D=D$). We call $D$ {\it antilinear self-adjoint\/} if $E_0=\bar E_1$, and $h_{E_0}=h_{E_1}$, and $D^*=\bar D$. For example, if $(X,g)$ is a spin Riemannian manifold of dimension $8n+6$ then the positive Dirac operator $\slashed{D}_+:\Ga^\iy(S_+)\ra \Ga^\iy(S_-)$ is antilinear self-adjoint.
\end{dfn}

\begin{dfn} 
\label{od2def4}
Suppose $X,\U(m),P,\A_P,\B_P$ are as Definition \ref{od2def1}, and $E_\bu$ is a complex elliptic operator on $X$ as in Definition \ref{od2def3}. Let $\nabla_P\in\A_P$. Then $\nabla_P$ induces a connection $\nabla_{\Ad(P)}$ on the real vector bundle $\Ad(P)\ra X$. Thus we may form the twisted complex elliptic operator
\e
\begin{split}
D^{\nabla_{\Ad(P)}}&:\Ga^\iy(\Ad(P)\ot_\R E_0)\longra\Ga^\iy(\Ad(P)\ot_\R E_1),\\
D^{\nabla_{\Ad(P)}}&:f\longmapsto \ts\sum_{i=0}^d (\id_{\Ad(P)}\ot a_i)\cdot \nabla_{\Ad(P)\ot E_0}^if,
\end{split}
\label{od2eq2}
\e
where $\nabla_{\Ad(P)\ot E_0}$ are the connections on $\Ad(P)\ot_\R E_0\ot_\R\bigot^iT^*X$ for $0\le i<d$ induced by $\nabla_{\Ad(P)}$ and~$\nabla_{E_0}$.

Since $D^{\nabla_{\Ad(P)}}$ is a complex linear elliptic operator on a compact manifold $X$, it has finite-dimensional kernel $\Ker(D^{\nabla_{\Ad(P)}})$ and cokernel $\Coker(D^{\nabla_{\Ad(P)}})$. The {\it determinant\/} $\det_\C(D^{\nabla_{\Ad(P)}})$ is the 1-dimensional complex vector space
\begin{equation*}
\det_\C(D^{\nabla_{\Ad(P)}})=\det_\C\Ker(D^{\nabla_{\Ad(P)}})\ot_\C\bigl(\det_\C\Coker(D^{\nabla_{\Ad(P)}})\bigr)^*,
\end{equation*}
where if $V$ is a finite-dimensional complex vector space then $\det_\C V=\La_\C^{\dim_\C V}V$.

These operators $D^{\nabla_{\Ad(P)}}$ vary continuously with $\nabla_P\in\A_P$, so they form a family of elliptic operators over the base topological space $\A_P$. Thus as in Atiyah and Singer \cite{AtSi} there is a natural complex line bundle $\hat K{}^{E_\bu}_P\ra\A_P$ with fibre $\hat K{}^{E_\bu}_P\vert_{\nabla_P}=\det_\C(D^{\nabla_{\Ad(P)}})$ at each $\nabla_P\in\A_P$. It is equivariant under the action of $\G_P$ on $\A_P$, and so pushes down to a complex line bundle $K^{E_\bu}_P\ra\B_P$ on the topological stack $\B_P$. We call $K^{E_\bu}_P$ the {\it determinant line bundle\/} of $\B_P$.

A {\it spin structure\/} on $\B_P$ is an isomorphism class $\bigl[(K^{E_\bu}_P)^{1/2}\bigr]$ of square root line bundles $(K^{E_\bu}_P)^{1/2}$ for $K^{E_\bu}_P\ra\B_P$. That is, a spin structure is an equivalence class of pairs $(J,\jmath)$, where $J\ra\B_P$ is a topological complex line bundle on $\B_P$, and $\jmath:J^{\ot^2}\ra K^{E_\bu}_P$ is an isomorphism, and pairs $(J,\jmath),(J',\jmath')$ are equivalent if there exists an isomorphism $\io:J\ra J'$ with~$\jmath=\jmath'\ci(\io\ot\io):J^{\ot^2}\ra K^{E_\bu}_P$.
\end{dfn}

\begin{rem}
\label{od2rem1}
If $Y$ is a complex manifold then spin structures on $Y$ in the usual sense of differential geometry, up to isomorphism, are in natural 1-1 correspondence with isomorphism classes $[K_Y^{1/2}]$ of square roots $K_Y^{1/2}$ of the canonical bundle $K_Y$ of $Y$, where the corresponding spin bundle is~$\slashed{S}=K_Y^{1/2}\ot\La^{0,*}T^*Y$.

We think of $\B_P$ as like an infinite-dimensional complex manifold, and $K^{E_\bu}_P$ as like its canonical bundle. This is why we call $\bigl[(K^{E_\bu}_P)^{1/2}\bigr]$ a `spin structure'. A similar analogy justifies the naming of the `orientations' on $\B_P$, defined using real elliptic operators $E_\bu$ on $X$, studied in~\cite{CGJ,JTU,JoUp1}.
\end{rem}

The next definition summarizes parts of~\cite[\S 4.1--\S 4.5]{JoUp2}.

\begin{dfn}
\label{od2def5}
Let $X$ be a compact manifold, and $P_1\ra X$, $P_2\ra X$ be principal $\U(m_1)$- and $\U(m_2)$-bundles for $m_1,m_2\ge 0$, with associated complex vector bundles $F_1\ra X$, $F_2\ra X$. Write $\bar F_1$ for the complex conjugate vector bundle of $F_1$. Then $\bar F_1\ot_\C F_2$ is a complex vector bundle on $X$, with fibre $\C^{m_1m_2}$, which we may write as
\begin{equation*}
\bar F_1\ot_\C F_2=(P_1\ot_XP_2\t\C^{m_1m_2})/(\U(m_1)\t\U(m_2)).
\end{equation*}

Let $E_\bu$ be a complex elliptic operator on $X$ as in Definition \ref{od2def3}. Let $(\nabla_{P_1},\nabla_{P_2})\in\A_{P_1}\t\A_{P_2}$. Then $(\nabla_{P_1},\nabla_{P_2})$ induces a connection $\nabla_{\bar F_1\ot F_2}$ on the complex vector bundle $\bar F_1\ot_\C F_2\ra X$. Thus as for \eq{od2eq2} we may form the twisted complex elliptic operator
\e
\begin{split}
D^{\nabla_{\bar F_1\ot F_2}}&:\Ga^\iy(\bar F_1\ot_\C F_2\ot_\C E_0)\longra\Ga^\iy(\bar F_1\ot_\C F_2\ot_\C E_1),\\
D^{\nabla_{\bar F_1\ot F_2}}&:f\longmapsto \ts\sum_{i=0}^d (\id_{\bar F_1\ot F_2}\ot a_i)\cdot \nabla_{\bar F_1\ot F_2\ot E_0}^if.
\end{split}
\label{od2eq3}
\e
As for $\hat K{}^{E_\bu}_P,K^{E_\bu}_P$ in Definition \ref{od2def4}, we define complex line bundles $\hat L{}^{E_\bu}_{P_1,P_2}\ra\A_{P_1}\t\A_{P_2}$ and $L^{E_\bu}_{P_1,P_2}\ra\B_{P_1}\t\B_{P_2}$ to be the determinant line bundle of the family of complex elliptic operators \eq{od2eq3} on $\A_{P_1}\t\A_{P_2}$, and its descent to $\B_{P_1}\t\B_{P_2}$. We construct some isomorphisms of line bundles involving~$L^{E_\bu}_{P_1,P_2}$:
\begin{itemize}
\setlength{\itemsep}{0pt}
\setlength{\parsep}{0pt}
\item[(a)] Let $P\ra X$ be a principal $\U(m)$-bundle, with associated complex vector bundle $F\ra X$. Take $P_1=P_2=P$, and consider the diagonal morphism $\De_{\B_P}:\B_P\ra\B_P\t\B_P$. The pullback $\De_{\B_P}^*(L^{E_\bu}_{P,P})$ is the determinant line bundle of the family of complex elliptic operators $D^{\nabla_{\bar F\ot F}}$. Since $\Ad(P)\ot_\R\C\cong F^*\ot_\C F\cong \bar F\ot_\C F$ we have natural isomorphisms
\begin{equation*}
\Ad(P)\ot_\R E_i\cong \bar F\ot_\C F\ot_\C E_i,\quad i=0,1.
\end{equation*}
Comparing \eq{od2eq2} and \eq{od2eq3}, we see that $D^{\nabla_{\Ad(P)}}\cong D^{\nabla_{\bar F\ot F}}$. Thus we have
\e
K^{E_\bu}_P\cong \De_{\B_P}^*(L^{E_\bu}_{P,P}).
\label{od2eq4}
\e
\item[(b)] Write $\Si_{\B_{P_1},\B_{P_2}}:\B_{P_1}\t\B_{P_2}\ra \B_{P_2}\t\B_{P_1}$ for the isomorphism swapping the factors. Exchanging $P_1,P_2$, we have a line bundle $L^{E_\bu}_{P_2,P_1}\ra\B_{P_2}\t\B_{P_1}$ from twisting $E_\bu$ by connections on $\bar F_2\ot_\C F_1\ra X$, so $\Si_{\B_{P_1},\B_{P_2}}^*(L^{E_\bu}_{P_2,P_1})$ is a complex line bundle on~$\B_{P_1}\t\B_{P_2}$.
\item[(c)] Suppose that $E_\bu$ is antilinear self-adjoint, as in Definition \ref{od2def3}, so that $D^*=\bar D$. Then we have an isomorphism of complex elliptic operators
\e
(D^{\nabla_{\bar F_1\ot F_2}})^*\cong \overline{D^{\nabla_{\bar F_2\ot F_1}}}.
\label{od2eq5}
\e
Taking determinants of \eq{od2eq5} gives an isomorphism on $\B_{P_1}\t\B_{P_2}$
\e
(L^{E_\bu}_{P_1,P_2})^*\cong\overline{\Si_{\B_{P_1},\B_{P_2}}^*(L^{E_\bu}_{P_2,P_1})}.
\label{od2eq6}
\e
But we can define a Hermitian metric on $L^{E_\bu}_{P_1,P_2}$ using the metrics in the problem, so that $\ov{L^{E_\bu}_{P_1,P_2}}\cong (L^{E_\bu}_{P_1,P_2})^*$. Combining this with \eq{od2eq6} gives
\e
L^{E_\bu}_{P_1,P_2}\cong\Si_{\B_{P_1},\B_{P_2}}^*(L^{E_\bu}_{P_2,P_1}).
\label{od2eq7}
\e
\item[(d)] For $i=0,1$ we have isomorphisms of complex vector bundles on $X$:
\e
\begin{split}
\Ad(&P_1\op P_2)\ot_\R E_i\cong (\bar F_1\op\bar F_2)\ot_\C (F_1\op F_2)\ot_\C E_i\\
&\cong (\bar F_1\ot_\C F_1\ot_\C E_i)\op(\bar F_1\ot_\C F_2\ot_\C E_i)\\
&\qquad \op
(\bar F_2\ot_\C F_1\ot_\C E_i)\op
(\bar F_2\ot_\C F_2\ot_\C E_i)\\
&\cong (\Ad(P_1)\ot_\R E_i)\op(\bar F_1\ot_\C F_2\ot_\C E_i)\\
&\qquad \op
(\bar F_2\ot_\C F_1\ot_\C E_i)\op
(\Ad(P_2)\ot_\R E_i).
\end{split}
\label{od2eq8}
\e
Given connections $\nabla_{P_1},\nabla_{P_2}$ on $P_1,P_2$ inducing $\nabla_{P_1\op P_2}$ on $P_1\op P_2$ as in Definition \ref{od2def2}, equation \eq{od2eq8} induces an isomorphism of elliptic operators
\e
D^{\nabla_{\Ad(P_1\op P_2)}}\cong D^{\nabla_{\Ad(P_1)}}\op D^{\nabla_{\bar F_1\ot F_2}}\op D^{\nabla_{\bar F_2\ot F_1}}\op D^{\nabla_{\Ad(P_2)}}.
\label{od2eq9}
\e
Taking determinants gives an isomorphism
\begin{align*}
\det_\C\bigl(D^{\nabla_{\Ad(P_1\op P_2)}}\bigr)&\cong \det_\C\bigl(D^{\nabla_{\Ad(P_1)}}\bigr)\ot \det_\C\bigl(D^{\nabla_{\bar F_1\ot F_2}}\bigr)\\
&\qquad \ot \det_\C\bigl(D^{\nabla_{\bar F_2\ot F_1}}\bigr)\ot \det_\C\bigl(D^{\nabla_{\Ad(P_2)}}\bigr).
\end{align*}
This is the fibre at $([\nabla_{P_1}],[\nabla_{P_2}])$ of an isomorphism on $\B_{P_1}\t\B_{P_2}$:
\e
\begin{split}
&\Phi_{P_1,P_2}^*(K_{P_1\op P_2}^{E_\bu})\cong \\
&\pi_{\B_{P_1}}^*(K_{P_1}^{E_\bu})\ot L^{E_\bu}_{P_1,P_2}\ot\Si_{\B_{P_1},\B_{P_2}}^*(L^{E_\bu}_{P_2,P_1})\ot\pi_{\B_{P_2}}^*(K_{P_2}^{E_\bu}).
\end{split}
\label{od2eq10}
\e
If also $E_\bu$ is antilinear self-adjoint, \eq{od2eq7} and \eq{od2eq10} give an isomorphism
\end{itemize}
\e
\begin{split}
\phi^{E_\bu}_{P_1,P_2}:\pi_{\B_{P_1}}^*(K_{P_1}^{E_\bu})\ot\pi_{\B_{P_2}}^*(K_{P_2}^{E_\bu})\ot (L^{E_\bu}_{P_1,P_2})^{\ot^2} 
\longra\Phi_{P_1,P_2}^*(K_{P_1\op P_2}^{E_\bu})&.
\end{split}
\label{od2eq11}
\e
\end{dfn}

The next theorem \cite[Th.s 5.12 \& 5.1(b)]{JoUp2} will be central to our paper.

\begin{thm}
\label{od2thm1}
Suppose\/ $(X,g)$ is a compact, oriented, spin Riemannian\/ $6$-manifold, and take\/ $E_\bu$ to be the positive Dirac operator\/ $\slashed{D}_+:\Ga^\iy(S_+)\ra\Ga^\iy(S_-),$ an antilinear self-adjoint complex linear elliptic operator. 

Then we can construct \begin{bfseries}canonical choices\end{bfseries} of spin structures\/ $\bigl[(K_P^{E_\bu})^{1/2}\bigr]$ on\/ $\B_P$ for all principal\/ $\U(m)$-bundles\/ $P\ra X,$ for all\/ $m\ge 0$. Furthermore:
\begin{itemize}
\setlength{\itemsep}{0pt}
\setlength{\parsep}{0pt}
\item[{\bf(a)}] Suppose\/ $P\ra X,$ $P'\ra X$ are principal\/ $\U(m)$-bundles and\/ $\rho:P\ra P'$ is an isomorphism. This induces isomorphisms\/ $\B_P\cong\B_{P'}$ and\/ $K_P^{E_\bu}\cong K_{P'}^{E_\bu},$ both of which are independent of the choice of\/ $\rho$. These identify the canonical spin structures\/~$\bigl[(K_P^{E_\bu})^{1/2}\bigr]\cong\bigl[(K_{P'}^{E_\bu})^{1/2}\bigr]$.
\item[{\bf(b)}] These canonical spin structures are \begin{bfseries}compatible with direct sums\end{bfseries}, in the sense that if\/ $P_1\ra X,$ $P_2\ra X$ are principal\/ $\U(m_1)$- and\/ $\U(m_2)$-bundles, so that\/ $P_1\op P_2\ra X$ is a principal\/ $\U(m_1+m_2)$-bundle, and\/ $(K_{P_1}^{E_\bu})^{1/2},\ab(K_{P_2}^{E_\bu})^{1/2},\ab(K_{P_1\op P_2}^{E_\bu})^{1/2}$ are representatives for the spin structures on\/ $\B_{P_1},\B_{P_2},\B_{P_1\op P_2},$ then there exists an isomorphism
\e
\begin{split}
\xi^{E_\bu}_{P_1,P_2}:
\pi_{\B_{P_1}}^*\bigl((K_{P_1}^{E_\bu})^{1/2}\bigr)&\ot\pi_{\B_{P_2}}^*\bigl((K_{P_2}^{E_\bu})^{1/2}\bigr)\ot L^{E_\bu}_{P_1,P_2}\\
&\qquad\longra\Phi_{P_1,P_2}^*\bigl((K_{P_1\op P_2}^{E_\bu})^{1/2}\bigr)
\end{split}
\label{od2eq12}
\e
on\/ $\B_{P_1}\t\B_{P_2}$ with\/ $\xi^{E_\bu}_{P_1,P_2}\ot\xi^{E_\bu}_{P_1,P_2}=\phi^{E_\bu}_{P_1,P_2}$ in equation \eq{od2eq11}. 
\end{itemize}
\end{thm}

The proof of Theorem \ref{od2thm1} in \cite{JoUp2} is complicated. We show spin structures on $\B_P$ for $\U(m)$-bundles $P\ra X$ can be related to orientations on $\B_Q$ for $\U(m)$-bundles $Q\ra X\t\cS^1$, by mapping $\B_Q$ to the loop space of $\B_P$. Then we use the construction in \cite{JoUp1} of canonical orientations on $\B_Q$ for $\U(m)$-bundles $P\ra Y$ over compact spin 7-manifolds $Y$ to construct canonical spin structures on~$\B_P$.

\section{Spin structures on algebraic moduli spaces}
\label{od3}

We now develop an analogue of the material of \S\ref{od2} for moduli stacks $\M,\oM$ of (complexes of) coherent sheaves on a smooth projective $\C$-scheme $X$, rather than moduli stacks $\B_P$ of connections on a principal $\U(m)$-bundle $P\ra X$. We will use parallel notation: $\Phi,K_\M,L_\M,\De_\M,\Si_\M,\phi_\M,\xi_\M$ in \S\ref{od31}--\S\ref{od32} below are the analogues of $\Phi_{P_1,P_2}, K^{E_\bu}_P,L^{E_\bu}_{P_1,P_2},\De_{\B_P},\Si_{\B_{P_1},\B_{P_2}},\phi^{E_\bu}_{P_1,P_2},\xi^{E_\bu}_{P_1,P_2}$ in~\S\ref{od2}.

\subsection{Background in algebraic geometry}
\label{od31}

This paper uses a lot of advanced technology --- stacks and higher stacks, derived categories, and so on --- which would take many pages to explain. So we will just give references, and hope that readers already have the necessary background. The next remark reviews the background material we will need. 

\begin{rem}
\label{od3rem1}
{\bf(a)} We work throughout over the field of complex numbers $\C$. For foundations of Algebraic Geometry, including $\C$-schemes $X$, see Hartshorne \cite{Hart}. Write $\Sch_\C$ for the category of $\C$-schemes.
\smallskip

\noindent{\bf(b)} Let $X$ be a smooth projective $\C$-scheme. Write $\coh(X)$ and $\qcoh(X)$ for the abelian categories of coherent and quasicoherent sheaves on $X$, as in Hartshorne \cite[\S II.5]{Hart} and Huybrechts and Lehn \cite{HuLe}.

Write $D^b\coh(X)$ for the bounded derived category of coherent sheaves on $X$, and $D\qcoh(X)$ for the unbounded derived category of quasicoherent sheaves on $X$, so that $D^b\coh(X)\subset D\qcoh(X)$. See Gelfand and Manin \cite{GeMa} for the theory of derived categories, and Huybrechts \cite{Huyb} for derived categories $D^b\coh(X)$.

When we use functors on derived categories, such as $f^*:D^b\coh(Y)\ra D^b\coh(X)$ and $f_*:D^b\coh(X)\ra D^b\coh(Y)$ for suitable morphisms $f:X\ra Y$ in $\Sch_\C$, we always mean derived functors, as in Huybrechts~\cite{Huyb}.
\smallskip

\noindent{\bf(c)} Write $\Perf(X)\subset D^b\coh(X)$ for the full triangulated subcategory of {\it perfect complexes\/} $\cE^\bu$, which are locally quasi-isomorphic to a bounded complex of vector bundles. We have $\Perf(X)=D^b\coh(X)$ when $X$ is smooth. 

A perfect complex $\cE^\bu$ has a {\it dual\/} perfect complex $(\cE^\bu)^\vee$. 

A perfect complex $\cE^\bu$ has a {\it determinant line bundle\/} $\det\cE^\bu$, a line bundle on $X$. If $\ra\cE^\bu\ra\cF^\bu\ra\cG^\bu\,{\buildrel[1]\over\longra}$ is a distinguished triangle in $\Perf(X)$ then there is a natural isomorphism $\det\cF^\bu\cong\det\cE^\bu\ot\det\cG^\bu$. See Knudsen and Mumford \cite{KnMu} and Quillen \cite{Quil} for the theory of determinant line bundles.
\smallskip

\noindent{\bf(d)} We will be interested in {\it Artin $\C$-stacks}, as in  G\'omez \cite{Gome}, Laumon and Moret-Bailly \cite{LaMo}, Olsson \cite{Olss}, and de Jong \cite{Jong}. All Artin stacks in this paper will be {\it locally of finite type}.

Classical $\C$-schemes and algebraic $\C$-spaces may be written as
functors
\e
S:\{\text{commutative $\C$-algebras}\}\longra\{\text{sets}\}.
\label{od3eq1}
\e
Extending this, classical Artin $\C$-stacks may be defined as functors
\e
S:\{\text{commutative $\C$-algebras}\}\longra\{\text{groupoids}\},
\label{od3eq2}
\e
satisfying many conditions. Artin $\C$-stacks form a 2-category $\Art_\C$.

The categories $\coh(X),\qcoh(X),D^b\coh(X),D\qcoh(X),\Perf(X)$ also make sense for Artin $\C$-stacks, with the same behaviour as in~{\bf(b)\rm,\bf(c)}.
\smallskip

\noindent{\bf(e)} Let $X$ be a smooth projective $\C$-scheme. We often write $\M$ for the moduli stack of objects in $\coh(X)$. It is an Artin $\C$-stack. As a functor \eq{od3eq2},
\begin{align*}
\M:A\longmapsto \bigl\{&\text{groupoid of coherent sheaves $\cE\ra X\t\Spec A$,}\\
&\text{flat over $\Spec A$,   isomorphisms of such $\cE$}\bigr\}.	
\end{align*}
The $\C$-points $[F]$ of $\M$ are isomorphism classes of $F\in\coh(X)$. There is a {\it universal coherent sheaf\/} $\cU$ in $\Perf(X\t\M)$, with $\cU\vert_{X\t\{[F]\}}\cong F$.

There is a natural morphism $\Phi:\M\t\M\ra\M$ in $\Art_\C$ mapping $\Phi:([F_1],[F_2])\ra F_1\op F_2$ on the level of $\C$-points. There is a natural isomorphism
\e
(\id_X\t\Phi)^*(\cU)\cong \pi_{12}^*(\cU)\op \pi_{13}^*(\cU)	
\label{od3eq3}
\e
in $\Perf(X\t\cM\t\cM)$, where $\pi_{ij}$ maps $X\t\cM\t\cM$ to the product of its $i^{\rm th}$ and $j^{\rm th}$ factors.
\smallskip

\noindent{\bf(f)} We shall also use the theory of {\it higher stacks}. Higher $\C$-stacks are explained in Simpson \cite{Simp2} and To\"en and Vezzosi \cite{Toen1,ToVe1,ToVe2}. They form an $\iy$-category $\HSta_\C$. Extending \eq{od3eq2}, one may regard higher $\C$-stacks as $\iy$-functors
\e
S:\{\text{commutative $\C$-algebras}\}\longra\{\text{$\iy$-groupoids}\}
\label{od3eq4}
\e
satisfying many conditions, where a model for $\iy$-groupoids is the $\iy$-category of Kan simplicial sets. Artin $\C$-stacks embed in higher $\C$-stacks $\Art_\C\hookra\HSta_\C$, so we can regard Artin $\C$-stacks as special examples of higher $\C$-stacks.	
\smallskip

\noindent{\bf(g)} Let $S$ be a higher $\C$-stack. Then To\"en \cite[\S 3.1.7]{Toen1} defines a triangulated category $L_\qcoh(S)$ of modules on $S$, which agrees with $D\qcoh(S)$ if $S$ is an Artin $\C$-stack. There is a full subcategory of perfect complexes $\Perf(S)\subset L_\qcoh(S)$. An object $\cE^\bu$ in $\Perf(S)$ has a dual $(\cE^\bu)^\vee$ and a determinant line bundle $\det\cE^\bu$, with the same behaviour as in {\bf(b)\rm,\bf(c)}. When working with higher and derived stacks, one should use $L_\qcoh(S)$ instead of~$D\qcoh(S)$.
\smallskip

\noindent{\bf(h)} Let $X$ be a smooth projective $\C$-scheme. Extending {\bf(f)}, we often write $\oM$ for the moduli stack of objects in $D^b\coh(X)$, which exists by To\"en--Vaqui\'e \cite{ToVa}. It is a higher $\C$-stack, with $\C$-points the isomorphism classes $[F^\bu]$ of objects $F^\bu$ in $D^b\coh(X)$. Embedding $\coh(X)\hookra D^b\coh(X)$ in the usual way gives an inclusion $\M\hookra\oM$ as an open $\C$-substack. If $F^\bu\in D^b\coh(X)$ with $\Ext^{<0}(F^\bu,F^\bu)\ne 0$ then $\oM$ is not an Artin $\C$-stack near $[F^\bu]$. This is why we need higher $\C$-stacks.

There is a {\it universal complex\/} $\cU^\bu$ in $\Perf(X\t\oM)\subset L_\qcoh(X\t\oM)$, with $\cU^\bu\vert_{X\t\{[F^\bu]\}}\cong F^\bu$ for each $F^\bu$ in $D^b\coh(X)$. This corresponds to a morphism $u:X\t\oM\ra\Perf_\C,$ where $\Perf_\C$ is a higher stack which classifies perfect complexes, given by $\Perf_\C=t_0(\bs\Perf_\C)$ for $\bs\Perf_\C$ the derived stack from To\"en and Vezzosi \cite[Def.~1.3.7.5]{ToVe2}. Then $\Perf_\C$ is just $\oM$ for $X=\Spec\C$ the point. In fact $u$ realizes $\oM$ as the mapping stack~$\oM=\Map_{\HSta_\C}(X,\Perf_\C)$.

There is a natural morphism $\bar\Phi:\oM\t\oM\ra\oM$ in $\HSta_\C$ mapping $\bar\Phi:([F_1^\bu],[F_2^\bu])\ra F_1^\bu\op F_2^\bu$ on the level of $\C$-points. There is a natural isomorphism
\begin{equation*}
(\id_X\t\Phi)^*(\cU^\bu)\cong \pi_{12}^*(\cU^\bu)\op \pi_{13}^*(\cU^\bu)	
\end{equation*}
in $\Perf(X\t\oM\t\oM)$, where $\pi_{ij}$ maps $X\t\oM\t\oM$ to its $i^{\rm th}$ and $j^{\rm th}$ factors.
\end{rem}

We will use the material in Remark \ref{od3rem1} freely from now on. 

\subsection{Spin structures in algebraic geometry}
\label{od32}

The next (rather long) definition sets up the situation we will study.

\begin{dfn}
\label{od3def1}
Let $X$ be a smooth projective $\C$-scheme, of complex dimension $m$. We call $X$ a {\it smooth projective $m$-fold}. Write $X^\ran$ for the underlying complex $m$-manifold of $X$, which is also a smooth $2m$-manifold. As a topological space, $X^\ran$ is the set of $\C$-points of $X$ with the complex analytic~topology.

An ({\it algebraic\/}) {\it spin structure\/} on $X$ is a choice of square root $K_X^{1/2}$ of the canonical bundle $K_X=\La^m_\C T^*X$ of $X$. Explicitly, a spin structure is a pair $(J,\jmath)$ of a line bundle $J\ra X$ and an isomorphism $\jmath:J\ot J\ra K_X$. If $X$ has a spin structure we call it a {\it spin smooth projective $m$-fold}.

It is well known that such algebraic spin structures $(J,\jmath)$ correspond to spin structures on the smooth manifold $X^\ran$ in the usual sense of differential geometry, where the corresponding spin bundle is $\slashed{S}=J\ot\La^{0,*}T^*X$.

Fix $X,m$ and $(J,\jmath)$ as above. As in Remark \ref{od3rem1}(e), write $\M$ for the moduli stack of objects in $\coh(X)$, as an Artin $\C$-stack with $\C$-points $[F]$ for $F$ in $\coh(X)$, and $\Phi:\M\t\M\ra\M$ for the direct sum morphism mapping $\Phi:([F_1],[F_2])\ra F_1\op F_2$ on $\C$-points, and $\cU\in\Perf(X\t\M)$ for the universal sheaf. As in Remark \ref{od3rem1}(h), write $\oM$ for the moduli stack of objects in $D^b\coh(X)$, as a higher $\C$-stack with $\C$-points $[F^\bu]$ for $F^\bu\in D^b\coh(X)$, and $\bar\Phi:\oM\t\oM\ra\oM$ for the direct sum morphism mapping $\bar\Phi:([F_1^\bu],[F_2^\bu])\ra F_1^\bu\op F_2^\bu$ on $\C$-points, and $\cU^\bu\in\Perf(X\t\oM)$ for the universal complex.

Given a product of stacks $S_1\t\cdots\t S_n$, we will write $\pi_i:S_1\t\cdots\t S_n\ra S_i$ for the projection to the $i^{\rm th}$ factor, and $\pi_{ij}:S_1\t\cdots\t S_n\ra S_i\t S_j$ for the projection to the product of the $i^{\rm th}$ and $j^{\rm th}$ factors, and so on.

Define complexes $\cC^\bu\in\Perf(\M)$ and $\cD^\bu\in\Perf(\M\t\M)$ by 
\ea
\cC^\bu&=(\pi_2)_*\bigl(\pi_1^*(J)\ot\cU^\vee\ot\cU\bigr),
\label{od3eq5}\\
\cD^\bu&=(\pi_{23})_*\bigl(\pi_1^*(J)\ot\pi_{12}^*(\cU^\vee)\ot\pi_{13}^*(\cU)\bigr),
\label{od3eq6}
\ea
where in \eq{od3eq5} we work on $X\t\M$ with projections $\pi_1:X\t\M\ra X$, and so on, and in \eq{od3eq6} we work on $X\t\M\t\M$. Here the pushforwards $(\pi_2)_*,(\pi_{23})_*$ are well defined as $\pi_2,\pi_{23}$ have fibre $X$, and so are representable, smooth, and proper. Write $\De_\M:\M\ra\M\t\M$ for the diagonal morphism. Then using $\id_X\t\De_\M:X\t\M\ra X\t\M\t\M$ we see there is a natural isomorphism
\e
\cC^\bu\cong\De_\M^*(\cD^\bu).
\label{od3eq7}
\e

Define line bundles $K_\M$ on $\M$ and $L_\M$ on $\M\t\M$ by $K_\M=\det\cC^\bu$ and $L_\M=\det\cD^\bu$. We think of $K_\M$ as a kind of `canonical bundle' of $\M$, for reasons explained in \S\ref{od43}. Equation \eq{od3eq7} yields an isomorphism, analogous to~\eq{od2eq4}:
\e
K_\M\cong\De_\M^*(L_\M).
\label{od3eq8}
\e
The cohomology groups of $\cC^\bu,\cD^\bu$ at $\C$-points $[F],([F_1],[F_2])$ are
\e
\begin{split}
h^i(\cC^\bu\vert_{[F]})&\cong \Ext^i(F,F\ot J),\\
h^i(\cD^\bu\vert_{([F_1],[F_2])})&\cong \Ext^i(F_1,F_2\ot J).
\end{split}
\label{od3eq9}
\e
Thus the determinant line bundles $L_\M,K_\M$ have fibres
\begin{align*}
K_\M\vert_{[F]}&\cong \bigot\nolimits_{i=0}^m\bigl(\La^{\rm top}_\C \Ext^i(F,F\ot J)\bigr)^{(-1)^i},\\
L_\M\vert_{([F_1],[F_2])}&\cong \bigot\nolimits_{i=0}^m\bigl(\La^{\rm top}_\C \Ext^i(F_1,F_2\ot J)\bigr)^{(-1)^i}.
\end{align*}

As $X$ is a smooth projective $\C$-scheme of dimension $m$ with canonical bundle $K_X$, by Huybrechts \cite[Th.~3.34]{Huyb} (see also Brav and Dyckerhoff \cite[\S 2.1, \S 5.2]{BrDy} in a more derived/stacky context), {\it Grothendieck--Verdier duality\/} for the pushforward $\pi_{23}:X\t\M\t\M\ra\M\t\M$ implies that if $\cG^\bu\in\Perf(X\t\M\t\M)$ then there is a natural isomorphism in $\Perf(\M\t\M)$:
\e
\bigl((\pi_{23})_*(\cG^\bu)\bigr)^\vee\cong(\pi_{23})_*	\bigl((\cG^\bu)^\vee \ot \pi_1^*(K_X)[m]\bigr).
\label{od3eq10}
\e
Write $\Si_\M:\M\t\M\ra\M\t\M$ for the involution swapping the factors. Then
\e
\begin{split}
(\cD^\bu)^\vee&=\bigl((\pi_{23})_*\bigl(\pi_1^*(J)\ot\pi_{12}^*(\cU^\vee)\ot\pi_{13}^*(\cU)\bigr)\bigr)^\vee\\
&\cong (\pi_{23})_*\bigl(\pi_1^*(J^*)\ot\pi_{12}^*(\cU)\ot\pi_{13}^*(\cU^\vee)\ot \pi_1^*(K_X)[m]\bigr)\\
&\cong (\pi_{23})_*\bigl(\pi_1^*(J)\ot\pi_{13}^*(\cU^\vee)\ot\pi_{12}^*(\cU)\bigr)[m]\\
&\cong \Si_\M^*\bigl((\pi_{23})_*\bigl(\pi_1^*(J)\ot\pi_{12}^*(\cU^\vee)\ot\pi_{13}^*(\cU)\bigr)\bigr)[m]=\Si_\M^*(\cD^\bu)[m],
\end{split}
\label{od3eq11}
\e
using \eq{od3eq6} in the first and fifth steps, \eq{od3eq10} in the second, $J\ot J\cong K_X$ in the third, and swapping round second and third factors in $X\t\M\t\M$ using $\Si_\M$ in the fourth. Taking determinant line bundles in \eq{od3eq11} gives an isomorphism
\e
L_\M^{(-1)^{m+1}}\cong \Si_\M^*(L_\M),
\label{od3eq12}
\e
which when $m$ is odd is an analogue of \eq{od2eq7}. Applying $\De_\M^*$ and using \eq{od3eq7} and $\Si_\M\ci\De_\M=\De_\M$ yields an isomorphism
\e
K_\M^{(-1)^{m+1}}\cong K_\M.
\label{od3eq13}
\e

Under the morphism $\Phi\t\Phi:\M\t\M\t\M\t\M\ra\M\t\M$ we have
\ea
(\Phi&\t\Phi)^*(\cD^\bu)=(\Phi\t\Phi)^*\bigl((\pi_{23})_*\bigl(\pi_1^*(J)\ot\pi_{12}^*(\cU^\vee)\ot\pi_{13}^*(\cU)\bigr)\bigr)
\nonumber\\
&\cong(\pi_{2345})_*\bigl(\pi_1^*(J)\ot\pi_{123}^*((\id_X\t\Phi)^*(\cU)^\vee)\ot\pi_{145}^*((\id_X\t\Phi)^*(\cU))\bigr)
\nonumber\\
&\cong(\pi_{2345})_*\bigl(\pi_1^*(J)\ot(\pi_{12}^*(\cU)^\vee\op \pi_{13}^*(\cU)^\vee)\ot(\pi_{14}^*(\cU)\op \pi_{15}^*(\cU))\bigr)
\nonumber\\
&\cong(\pi_{2345})_*\bigl((\pi_1^*(J)\ot\pi_{12}^*(\cU)^\vee\ot\pi_{14}^*(\cU))\op
(\pi_1^*(J)\ot\pi_{12}^*(\cU)^\vee\ot\pi_{15}^*(\cU))
\nonumber\\
&\qquad\op(\pi_1^*(J)\ot\pi_{13}^*(\cU)^\vee\ot\pi_{14}^*(\cU))\op
(\pi_1^*(J)\ot\pi_{13}^*(\cU)^\vee\ot\pi_{15}^*(\cU))\bigr)
\nonumber\\
&\cong(\pi_{13})^*\bigl((\pi_{23})_*\bigl(\pi_1^*(J)\ot\pi_{12}^*(\cU^\vee)\ot\pi_{13}^*(\cU)\bigr)\bigr)\nonumber\\
&\quad\op(\pi_{14})^*\bigl((\pi_{23})_*\bigl(\pi_1^*(J)\ot\pi_{12}^*(\cU^\vee)\ot\pi_{13}^*(\cU)\bigr)\bigr)
\nonumber\\
&\quad\op(\pi_{23})^*\bigl((\pi_{23})_*\bigl(\pi_1^*(J)\ot\pi_{12}^*(\cU^\vee)\ot\pi_{13}^*(\cU)\bigr)\bigr)\nonumber\\
&\quad\op(\pi_{24})^*\bigl((\pi_{23})_*\bigl(\pi_1^*(J)\ot\pi_{12}^*(\cU^\vee)\ot\pi_{13}^*(\cU)\bigr)\bigr)
\nonumber\\
&\cong\pi_{13}^*(\cD^\bu)\op \pi_{14}^*(\cD^\bu)\op\pi_{23}^*(\cD^\bu)\op\pi_{24}^*(\cD^\bu).
\label{od3eq14}
\ea
Here the terms inside the pushforwards $(\pi_{\cdots})_*(\cdots)$ live on $X\t\M\t\M$ in the second and sixth steps, and on $X\t\M\t\M\t\M\t\M$ in the third--fifth. We use \eq{od3eq6} in the first and sixth step, commute pullbacks and pushforwards in the second and fifth, and use \eq{od3eq3} in the third.

Taking determinant line bundles in \eq{od3eq14} (see Remark \ref{od3rem2}(b) on this) gives an isomorphism of line bundles on $\M\t\M\t\M\t\M$:
\e
(\Phi\t\Phi)^*(L_\M)\cong \pi_{13}^*(L_\M)\ot \pi_{14}^*(L_\M)\ot\pi_{23}^*(L_\M)\ot\pi_{24}^*(L_\M).
\label{od3eq15}
\e
The diagonal morphism $\De_{\M\t\M}:\M\t\M\ra\M\t\M\t\M\t\M$ satisfies
\begin{gather*}
(\Phi\t\Phi)\ci\De_{\M\t\M}=\De_\M\ci\Phi, \qquad \pi_{13}\ci\De_{\M\t\M}=\De_\M\ci\pi_1, \\
\pi_{14}\ci\De_{\M\t\M}=\id, \quad \pi_{23}\ci\De_{\M\t\M}=\Si_\M, \quad \pi_{24}\ci\De_{\M\t\M}=\De_\M\ci\pi_2.
\end{gather*}
Applying $\De_{\M\t\M}^*$ to \eq{od3eq15} and using these and equations \eq{od3eq8} and \eq{od3eq12} gives an isomorphism of line bundles on~$\M\t\M$:
\e
\phi_\M:\pi_1^*(K_\M)\ot\pi_2^*(K_\M)\ot L_\M\ot L_\M^{(-1)^{m+1}}\longra\Phi^*(K_\M),
\label{od3eq16}
\e
which when $m$ is odd is analogous to~\eq{od2eq11}.

By a similar but simpler argument to \eq{od3eq14}--\eq{od3eq15}, under the morphisms $\Phi\t\id_\M,\id_\M\t\Phi:\M\t\M\t\M\ra\M\t\M$ we have natural isomorphisms
\ea
\chi_\M:\pi_{13}^*(L_\M)\ot\pi_{23}^*(L_\M)\longra(\Phi\t\id_\M)^*(L_\M),
\label{od3eq17}\\
\psi_\M:\pi_{12}^*(L_\M)\ot\pi_{13}^*(L_\M)\longra(\id_\M\t\Phi)^*(L_\M).
\label{od3eq18}
\ea
Then we can show that the following diagram on $\M\t\M\t\M$ commutes:
\e
\begin{gathered}
\xymatrix@!0@C=92pt@R=54pt{
*+[r]{\begin{subarray}{l} \ts \pi_1^*(K_\M)\!\ot\!\pi_2^*(K_\M)\!\ot\!\pi_3^*(K_\M)\ot \\
\ts \pi_{12}^*(L_\M^{\ot^2})\!\ot\!  \pi_{13}^*(L_\M^{\ot^2})\!\ot\!\pi_{23}^*(L_\M^{\ot^2})\end{subarray}}
\ar[rrr]_(0.59){\begin{subarray}{l} \quad\pi_{12}^*(\phi_\M)\ot \\  \id_{\pi_3^*(K_\M)}\ot\chi_\M^{\ot^2}\end{subarray}} \ar@<-10pt>[d]^{\id_{\pi_1^*(K_\M)}\ot \pi_{23}^*(\phi_\M)\ot \psi_\M^{\ot^2}}
&&& *+[l]{\begin{subarray}{l} \ts (\Phi\!\t\!\id_\M)^* \bigl(\pi_1^*(K_\M) \\ \ts \quad\;\>{}\ot\pi_2^*(K_\M)\!\ot\! L_\M^{\ot^2}\bigr)\end{subarray}} \ar@<10pt>[d]_{(\Phi\t\id_\M)^*(\phi_\M)}
\\
*+[r]{\begin{subarray}{l} \ts (\id_\M\!\t\!\Phi)^*\bigl(\pi_1^*(K_\M) \\ \ts{}\ot\pi_2^*(K_\M)\!\ot\! L_\M^{\ot^2}\bigr)\end{subarray}} \ar[rr]^(0.63){(\id_\M\t\Phi)^*(\phi_\M)} && {\begin{subarray}{l} \ts (\id_\M\!\t\!\Phi)^*\\
\ts{}\ci\Phi^*(K_\M)\end{subarray}} \ar@{=}[r] &  *+[l]{\begin{subarray}{l} \ts (\Phi\!\t\!\id_\M)^*\\
\ts{}\ci\Phi^*(K_\M).\!\end{subarray}}
}\!\!\!
\end{gathered}
\label{od3eq19}
\e
This is a lift to line bundles of the commutative diagram in $\Ho(\Art_\C)$:
\e
\begin{gathered}
\xymatrix@!0@C=280pt@R=30pt{
*+[r]{\M\t\M\t\M} \ar[r]_{\Phi\t\id_\M} \ar[d]^{\id_\M\t\Phi} & *+[l]{\M\t\M} \ar[d]_\Phi \\
*+[r]{\M\t\M} \ar[r]^\Phi & *+[l]{\M,\!}}
\end{gathered}
\label{od3eq20}
\e
since direct sum in $\coh(X)$ is associative.

We can also generalize all the above to the moduli stack $\oM$ of objects in $D^b\coh(X)$ in the obvious way. We replace $\M,\cU,\Phi$ by $\oM,\cU^\bu,\bar\Phi$ as in Remark \ref{od3rem1}(h), and we write $\bar\cC{}^\bu,\bar\cD{}^\bu,K_\oM,L_\oM,\Si_\oM,\phi_\oM,\chi_\oM,\psi_\oM$ for the analogues of $\cC^\bu,\cD^\bu,\ldots,\chi_\M,\psi_\M$. Then the analogues of \eq{od3eq5}--\eq{od3eq20} hold. Note that $\M\subset\oM$ is an open substack, and $K_\oM\vert_\M=K_\M$, $L_\oM\vert_{\M\t\M}=L_\M$, and equations \eq{od3eq5}--\eq{od3eq20} for $\oM$ restrict on $\M,\M\t\M,\ldots$ to \eq{od3eq5}--\eq{od3eq20} for~$\M$.
\end{dfn}

\begin{rem}
\label{od3rem2}
{\bf(a)} The line bundles $K_\M,L_\M$ in Definition \ref{od3def1} behave differently when $m$ is even, and when $m$ is odd. When $m$ is even, equation \eq{od3eq13} implies that $K_\M\cong P_\M\ot_{\Z_2}\O_\M$ for a principal $\Z_2$-bundle $P_\M\ra\M$, and we think of sections of $P_\M$ as `orientations' on $\M$. Also \eq{od3eq16} reduces to $\Phi^*(K_\M)\cong\pi_1^*(K_\M)\ot\pi_2^*(K_\M)$, so we can work just with $K_\M$. The case $m$ even was studied in Cao, Gross and Joyce \cite{CGJ},  as part of our series on orientations~\cite{CGJ,JTU,JoUp1}.

When $m$ is odd, \eq{od3eq13} is the identity and is boring, and \eq{od3eq16} becomes
\e
\phi_\M:\pi_1^*(K_\M)\ot\pi_2^*(K_\M)\ot L_\M^{\ot^2}\longra\Phi^*(K_\M).
\label{od3eq21}
\e
In this case, as in Definition \ref{od3def2} below, instead of orientations, it is interesting to study square roots $K_\M^{1/2}$, thought of as `spin structures' on~$\M$.
\smallskip

\noindent{\bf(b)} To deduce \eq{od3eq15} from \eq{od3eq14}, we use the fact that if $\cE^\bu,\cF^\bu$ are perfect complexes on $S$ then there is an isomorphism
\e
\det(\cE^\bu\op\cF^\bu)\cong (\det\cE^\bu)\ot(\det\cF^\bu).
\label{od3eq22}
\e
As discussed in \cite[\S 3.1.1]{Upme}, the isomorphism \eq{od3eq22} depends on an orientation convention, and swapping $\cE^\bu$ and $\cF^\bu$ changes its sign by $(-1)^{\rank\cE^\bu\rank\cF^\bu}$. Thus the isomorphisms \eq{od3eq15}--\eq{od3eq16} depend up to sign on the order we chose to write the four factors in the last line of~\eq{od3eq14}.
\end{rem}

We can now define an algebro-geometric analogue of spin structures on differential-geometric moduli spaces in \cite{JoUp2}.

\begin{dfn}
\label{od3def2}
Let $X$ be a spin smooth projective $m$-fold for $m$ odd, and use the notation of Definition \ref{od3def1}. We define a {\it spin structure on\/} $\M$ (or on $\oM$) to be an isomorphism class $[K_\M^{1/2}]$ of square root line bundles $K_\M^{1/2}$ for the line bundle $K_\M$ on $\M$ (or an isomorphism class $[K_\oM^{1/2}]$ of square roots $K_\oM^{1/2}$ on $\oM$). 

Explicitly, a spin structure on $\M$ is an isomorphism class $[J_\M,\jmath]$ of pairs $(J_\M,\jmath)$, where $J_\M\ra\M$ is a line bundle and $\jmath:J_\M\ot J_\M\ra K_\M$ is an isomorphism of line bundles, and two pairs $(J_\M,\jmath),(J'_\M,\jmath')$ are isomorphic if there exists an isomorphism $\io:J_\M\ra J'_\M$ with~$\jmath=\jmath'\ci(\io\ot\io)$.

As $\M\subset\oM$ is an open substack and $K_\oM\vert_\M=K_\M$, a spin structure on $\oM$ restricts to a spin structure on $\M$.

We call a spin structure $[K_\M^{1/2}]$ on $\M$ {\it compatible with direct sums\/} if choosing any representative $K_\M^{1/2}$ for $[K_\M^{1/2}]$, as for \eq{od2eq12} there exists an isomorphism
\e
\xi_\M:\pi_1^*(K_\M^{1/2})\ot\pi_2^*(K_\M^{1/2})\ot L_\M\longra\Phi^*(K_\M^{1/2})
\label{od3eq23}
\e
on $\M\t\M$ with $\xi_\M\ot\xi_\M=\phi_\M$ in equation \eq{od3eq21}. We call a spin structure $[K_\oM^{1/2}]$ on $\oM$ {\it compatible with direct sums\/} if the analogue $\xi_\oM$ exists on~$\oM\t\oM$.
\end{dfn}

The next theorem is a topological existence result for spin structures. It is based on Nekrasov and Okounkov \cite[\S 6]{NeOk}, who were interested in orientation data for Calabi--Yau 3-folds as in~\S\ref{od43}.

\begin{thm}[Extension of Nekrasov--Okounkov {\cite[\S 6]{NeOk}}]
\label{od3thm1}
In the situation of Definition\/ {\rm\ref{od3def1}} with\/ $m$ odd, the topological complex line bundle\/ $K_{\oM{}^\top} \ra \oM{}^\top$ admits a square root. In particular, combined with Proposition\/ {\rm\ref{od5prop5}} below, this implies the existence of spin structures\/ $K_\M^{1/2},K_\oM^{1/2}$ on\/ $\M$ and\/~$\oM$.

In greater generality, let\/ $B$ be a topological space with the homotopy type of a CW complex,\/ ${\De\colon B \ra B\t B}$ be the diagonal,\/ ${\Si\colon B\t B \ra B\t B}$ the braid, and\/ ${u \in H^2(B\t B, \Z)}$. If\/ ${\Si^* u = u}$, then\/ ${\De^*u \in H^2(B,\Z)}$ is divisible by two.
\end{thm}

\begin{proof}
The first Chern class determines an isomorphism between $H^2(B,\Z)$ and the set of isomorphism classes of complex line bundles on $B$ with the tensor product as group operation. Equation \eq{od3eq8} implies that $K_\oM \cong\De_\oM^*(L_\oM)$, where $\De_\oM:\oM\ra \oM\t\oM$ is the diagonal. Also \eq{od3eq12} and $m$ odd yield $\Si^*(L_\oM)\cong L_\oM$, so that $u=c_1(L_{\oM{}^\top})$ satisfies the assumption $\Si^*u=u$. Similarly for $\M$, as $K_\oM\vert_\M=K_\M$, $L_\oM\vert_{\M\t\M}=L_\M$.

From the Bockstein exact sequence $H^2(B,\Z) \xrightarrow{2\cdot} H^2(B,\Z) \ra H^2(B,\Z_2)$, see Dold \cite[\S 7.15]{Dold}, and from $H^2(B,\Z_2)\cong \Hom_{\Z_2}(H_2(B,\Z_2), \Z_2)$ by universal coefficients, we see that we may reduce to the case of a finite connected CW complex $B$. We may therefore suppose that\/ ${H_p(B,\Z)}$ is finitely generated for all\/ ${p\geq 0}$.  As $H^0(B,\Z)$ and $H^1(B,\Z) \cong \Hom(H_1(B),\Z)$ are finitely generated
torsion-free, they are free and the K\"unneth formula reduces to an isomorphism
\begin{equation*}
H^2(B,\Z) \oplus \left[H^1(B,\Z)\ot H^1(B,\Z)\right] \oplus H^2(B,\Z)\longrightarrow
H^2(B\t B,\Z).
\end{equation*}
This means that we may decompose the class $u$ as
\e
\label{od3eq24}
u=\pi_1^*a + \pi_1^*b\cup \pi_2^*c + \pi_2^*d
\e
for $a, d \in H^2(B,\Z)$ and $b,c \in H^1(B,\Z)$. By naturality of `$\cup$'
and ${\Si\circ\pi_1 = \pi_2}$,
\e
\label{od3eq25}
\Si^*u=\pi_2^*a + \pi_2^*b \cup \pi_1^*c + \pi_1^*d.
\e
Pulling back along $\incl(x) = (x,x_0)$ for fixed $x_0 \in B$ we find that $\Si^*u=u$
implies $a=d$.
Hence $\De^*(\pi_1^* a + \pi_2^*d) = a+d=2a$. Choose a basis
${e_1, \ldots, e_n}$ of ${H^1(B,\Z)}$ and expand ${b=\sum_{i=1}^n b^ie_i}$ and
${c=\sum_{j=1}^n c^je_j}$ for $b^i, c^j \in \Z$. Then the equality of \eqref{od3eq24} and
\eqref{od3eq25} means ${b^ic^j = -c^ib^j}$ for all $1\leq i,j\leq n$. Therefore
\begin{equation*}
\De^*u = 2a + \sum_{i,j=1}^n b^ic^j(e_i \cup e_j)
= 2a + \sum_{1\leq i<j\leq n} 2b^i c^j(e_i \cup e_j).\qedhere
\end{equation*}
\end{proof}

The following theorem, one of our main results, will be proved in \S\ref{od5}, using Theorem \ref{od2thm1} and results of Cao, Gross and Joyce~\cite{CGJ}.

\begin{thm}
\label{od3thm2}
Let\/ $X$ be a spin smooth projective\/ $3$-fold. Use the notation\/ $\M,\oM,K_\M,\ldots$ of Definitions\/ {\rm\ref{od3def1}} and\/ {\rm\ref{od3def2}}. Then we can construct canonical spin structures\/ $[K_\M^{1/2}]$ on\/ $\M$ and\/ $[K_\oM^{1/2}]$ on\/ $\oM,$ both compatible with direct sums, with\/~$[K_\oM^{1/2}]\vert_\M=[K_\M^{1/2}]$.
\end{thm}

A very rough idea of the proof is that we first pull back the spin structures on $\B_P$ for all $P\ra X$ in Theorem \ref{od2thm1} to spin structures on the open substack $\M_{\rm vect}\subset\M\subset\oM$ of algebraic vector bundles in $\M$, by mapping a rank $m$ algebraic vector bundle $F\ra X$ to its underlying holomorphic vector bundle $F^\ran\ra X^\ran$, for $F^\ran$ the associated vector bundle of a $\U(m)$-bundle~$P^\ran\ra X^\ran$.

Then we show that spin structures on $\M_{\rm vect}$ compatible with direct sums extend uniquely to spin structures on $\oM$ compatible with direct sums. This holds because, in a homotopy-theoretic sense, we can think of $(\M_{\rm vect},\Phi_{\rm vect})$ as a commutative monoid in stacks, and $(\oM,\bar\Phi)$ as its abelian group completion, and we use a universal property of group completions.

\subsection{Strong spin structures}
\label{od33}

Here is a categorification of the notion of spin structure in Definition \ref{od3def2}.

\begin{dfn}
\label{od3def3}
Let $X$ be a spin smooth projective $m$-fold for $m$ odd, and use the notation of Definitions \ref{od3def1} and \ref{od3def2}. We define a {\it strong spin structure on\/} $\M$ to be a square root $K_\M^{1/2}$ for $K_\M$ on~$\M$. 

We define a {\it strong spin structure on\/ $\M$ compatible with direct sums\/} to be a pair $(K_\M^{1/2},\xi_\M)$ of a square root $K_\M^{1/2}$ and an isomorphism $\xi_\M$ in \eq{od3eq23} with $\xi_\M\ot\xi_\M=\phi_\M$, such that the following diagram on $\M\t\M\t\M$ commutes:
\e
\begin{gathered}
\xymatrix@!0@C=92pt@R=54pt{
*+[r]{\begin{subarray}{l} \ts \pi_1^*(K_\M^{1/2})\!\ot\!\pi_2^*(K_\M^{1/2})\!\ot\!\pi_3^*(K_\M^{1/2})\ot \\
\ts \pi_{12}^*(L_\M)\!\ot\!  \pi_{13}^*(L_\M)\!\ot\!\pi_{23}^*(L_\M)\end{subarray}}
\ar[rrr]_(0.59){\begin{subarray}{l} \quad\pi_{12}^*(\xi_\M)\ot \\  \id_{\pi_3^*(K_\M^{1/2})}\ot\chi_\M\end{subarray}} \ar@<-10pt>[d]^{\id_{\pi_1^*(K_\M^{1/2})}\ot \pi_{23}^*(\xi_\M)\ot \psi_\M}
&&& *+[l]{\begin{subarray}{l} \ts (\Phi\!\t\!\id_\M)^* \bigl(\pi_1^*(K_\M^{1/2}) \\ \ts \quad\;\>{}\ot\pi_2^*(K_\M^{1/2})\!\ot\! L_\M\bigr)\end{subarray}} \ar@<10pt>[d]_{(\Phi\t\id_\M)^*(\xi_\M)}
\\
*+[r]{\begin{subarray}{l} \ts (\id_\M\!\t\!\Phi)^*\bigl(\pi_1^*(K_\M^{1/2}) \\ \ts{}\ot\pi_2^*(K_\M^{1/2})\!\ot\! L_\M\bigr)\end{subarray}} \ar[rr]^(0.63){(\id_\M\t\Phi)^*(\xi_\M)} && {\begin{subarray}{l} \ts (\id_\M\!\t\!\Phi)^*\\
\ts{}\ci\Phi^*(K_\M^{1/2})\end{subarray}} \ar@{=}[r] &  *+[l]{\begin{subarray}{l} \ts (\Phi\!\t\!\id_\M)^*\\
\ts{}\ci\Phi^*(K_\M^{1/2}).\!\end{subarray}}
}\!\!\!
\end{gathered}
\label{od3eq27}
\e
Equation \eq{od3eq27} is a `square root' of \eq{od3eq19}, and a lift to line bundles of~\eq{od3eq20}.

If $(K_\M^{1/2},\xi_\M),(\dot K_\M^{1/2},\dot\xi_\M)$ are strong spin structures on $\M$ compatible with direct sums, an {\it isomorphism\/} $\ka:(K_\M^{1/2},\xi_\M)\ra(\dot K_\M^{1/2},\dot\xi_\M)$ is an isomorphism of line bundles $\ka:K_\M^{1/2}\ra \dot K_\M^{1/2}$ such that the following commute:
\e
\begin{gathered}
\xymatrix@C=46pt@R=15pt{
*+[r]{K_\M^{1/2}\ot K_\M^{1/2}} \ar[d]^{\ka\ot\ka} \ar[r]_(0.68)\jmath & *+[l]{K_\M} \ar@{=}[d] 
\\ 
*+[r]{\dot K_\M^{1/2}\ot\dot K_\M^{1/2}} \ar[r]^(0.68)j & *+[l]{K_\M,\!} 
}\;\>
\xymatrix@C=88pt@R=15pt{ 
*+[r]{\pi_1^*(K_\M^{1/2})\ot\pi_2^*(K_\M^{1/2})\ot L_\M} \ar[d]^{\pi_1^*(\ka)\ot\pi_2^*(\ka)\ot \id_{L_\M}} \ar[r]_(0.72){\xi_\M} & *+[l]{\Phi^*(K_\M^{1/2})} \ar@{=}[d] 
\\ 
*+[r]{\pi_1^*(\dot K_\M^{1/2})\ot\pi_2^*(\dot K_\M^{1/2})\ot L_\M}  \ar[r]^(0.72){\dot\xi_\M} & *+[l]{\Phi^*(K_\M^{1/2}).\!} 
}\!\!\!
\end{gathered}
\label{od3eq28}
\e

We make the analogous definitions of {\it strong spin structure on\/} $\oM$ ({\it compatible with direct sums\/}) by replacing $\M,\Phi,K_\M,\ldots,\xi_\M$ by $\oM,\bar\Phi,K_\oM,\ldots,\xi_\oM$.
\end{dfn}

\begin{rem}
\label{od3rem3}
We can also define `strong spin structure compatible with direct sums' in the differential-geometric setting of \S\ref{od2}. It is probably best to start by choosing one principal $\U(m)$-bundle $P\ra X$ representing each isomorphism class $[P]$ of $\U(m)$-bundles, and then choosing square roots $(K^{E_\bu}_P)^{1/2}$ for each such representative $P$, and morphisms $\xi^{E_\bu}_{P_1,P_2}$ in \eq{od2eq12} for each pair $P_1,P_2$, replacing $P_1\op P_2$ by the chosen representative in its isomorphism class $[P_1\op P_2]$. In the analogue of Theorem \ref{od3thm3} below, which needs a different proof, one should replace $K_0^\semi(X)$ by complex K-theory $K^0(X)$ in the differential-geometric case.	
\end{rem}

We will explain in Remark \ref{od4rem5}(d) that equation \eq{od3eq27} will be needed to ensure Cohomological Hall Algebras of Calabi--Yau 3-folds are associative.

In Theorem \ref{od3thm3} we give criteria for when a spin structure $[K_\oM^{1/2}]$ on $\oM$ compatible with direct sums can be lifted to a strong spin structure $(K_\oM^{1/2},\xi_\oM)$ compatible with direct sums. We will need the following definition:

\begin{dfn}
\label{od3def4}
Let $X$ be a smooth projective $\C$-scheme. We will explain a definition of the {\it semi-topological K-theory group\/} $K_0^\semi(X)$ of Friedlander and Walker \cite{FrWa}. This is also known as the {\it holomorphic K theory group\/} $K^0_{\rm hol}(X)$, as in Lawson et al.\ \cite{LLM} and Cohen and Lima-Filho \cite{CoLF}.

As in Remark \ref{od3rem1}(h) write $\oM$ for the moduli stack of objects in $D^b\coh(X)$, as a higher $\C$-stack. Define $K_0^\semi(X)$ to be the set $\pi_0(\oM)$ of connected components of $\oM$. For each $F^\bu\in D^b\coh(X)$, write $\lb F^\bu\rb\in K_0^\semi(X)$ for the connected component of $\pi_0(\oM)$ containing the $\C$-point $[F^\bu]$. Then $K_0^\semi(X)=\bigl\{\lb F^\bu\rb:F^\bu\in D^b\coh(X)\bigr\}$. We make $K_0^\semi(X)$ into a commutative ring by 
\begin{gather*}
\lb F^\bu\rb+\lb G^\bu\rb=\lb F^\bu\op G^\bu\rb, \qquad \lb F^\bu\rb\cdot\lb G^\bu\rb=\lb F^\bu\ot G^\bu\rb, \\
-\lb F^\bu\rb=\lb F^\bu[1]\rb, \qquad 0=\lb 0\rb \qquad\text{and}\qquad 1=\lb\O_X\rb.
\end{gather*}
\end{dfn}

The next theorem will be proved in~\S\ref{od6}.

\begin{thm}
\label{od3thm3}
Let\/ $X$ be a spin smooth projective\/ $m$-fold with\/ $m$ odd, and\/ $[K_\oM^{1/2}]$ be a spin structure on\/ $\oM$ compatible with direct sums, as in Definition\/ {\rm\ref{od3def2}}. Write\/ $H^*\bigl(K_0^\semi(X),\Z_2\bigr)$ for the group cohomology of\/ $K_0^\semi(X)$ over $\Z_2,$ as in Brown {\rm\cite{Brow},} regarding\/ $K_0^\semi(X)$ as an abelian group under addition. Then:
\begin{itemize}
\setlength{\itemsep}{0pt}
\setlength{\parsep}{0pt}
\item[{\bf(a)}] We may define an obstruction class\/ $\Om\bigl([K_\oM^{1/2}]\bigr)$ in\/ $H^3\bigl(K_0^\semi(X),\Z_2\bigr),$ such that\/ $[K_\oM^{1/2}]$ may be extended to a strong spin structure\/ $(K_\oM^{1/2},\xi_\oM)$ compatible with direct sums, as in Definition\/ {\rm\ref{od3def3},} if and only if\/~$\Om\bigl([K_\oM^{1/2}]\bigr)=0$.
\item[{\bf(b)}] When\/ $\Om\bigl([K_\oM^{1/2}]\bigr)=0,$ the isomorphism classes of extensions\/ $(K_\oM^{1/2},\xi_\oM)$ in {\bf(a)} are a torsor over\/ $H^2\bigl(K_0^\semi(X),\Z_2\bigr)$.
\item[{\bf(c)}] Extensions\/ $(K_\oM^{1/2},\xi_\oM)$ have automorphism group\/~$H^1\bigl(K_0^\semi(X),\Z_2\bigr)$.
\end{itemize}
\end{thm}

Unfortunately requiring $H^3\bigl(K_0^\semi(X),\Z_2\bigr)=0$ is a very restrictive condition on $K_0^\semi(X)$, so the theorem does not provide useful conditions for $(K_\oM^{1/2},\xi_\oM)$ to exist. For example, $H^3(\Z^r,\Z_2)=\Z_2^{\binom{r}{3}}$, so if $K_0^\semi(X)\cong\Z^r\op\text{torsion}$ for $r\ge 3$ then $H^3\bigl(K_0^\semi(X),\Z_2\bigr)\ne 0$. The next question seems interesting:

\begin{quest}
\label{od3quest1} For a spin smooth projective\/ $3$-fold\/ $X,$ can\/ $[K_\M^{1/2}],[K_\oM^{1/2}]$ in Theorem\/ {\rm\ref{od3thm2}} be lifted to strong spin structures\/ $(K_\M^{1/2},\xi_\M),(K_\oM^{1/2},\xi_\oM)$ compatible with direct sums? Can this be done naturally up to \textup(canonical\textup) isomorphism?
\end{quest}

One can also ask the analogous question in the differential-geometric setting, as in Remark~\ref{od3rem3}.

\section{Calabi--Yau manifolds and orientation data}
\label{od4}

A very brief summary of this section is as follows:
\begin{itemize}
\setlength{\itemsep}{0pt}
\setlength{\parsep}{0pt}
\item[(a)] Kontsevich and Soibelman \cite[\S 5]{KoSo1} introduced a notion of {\it orientation data\/} for a Calabi--Yau $m$-fold $X$ when $m$ is odd. \item[(b)] Their orientation data is slightly weaker than our notion of `spin structure $[K_\M^{1/2}]$ on $\M$ compatible with direct sums' in Definition \ref{od3def2}, in the case when $K_\M\cong\O_X$ (as $X$ is Calabi--Yau) and $J\cong\O_X$.
\item[(c)] Thus, Theorem \ref{od3thm2} constructs canonical orientation data for every Calabi--Yau 3-fold, proving a well-known conjecture in Donaldson--Thomas theory.
\item[(d)] Orientation data is essential for various generalizations of Donaldson--Thomas theory of Calabi--Yau 3-folds. For example, (c) implies motivic Donaldson--Thomas invariants of Calabi--Yau 3-folds are well defined.
\item[(e)] {\it Strong orientation data}, meaning `strong spin structures $(K_\M^{1/2},\xi_\M)$ on $\M$ compatible with direct sums' from \S\ref{od33} in the Calabi--Yau 3-fold case, will also be important for some future applications.
\end{itemize}

Section \ref{od41} introduces orientation data, and proves our main result, Theorem \ref{od4thm1}, on existence of canonical orientation data for compact Calabi--Yau 3-folds. Section \ref{od42} extends this to noncompact Calabi--Yau 3-folds. Section \ref{od43} discusses the r\^ole of (strong) orientation data in several current and future generalizations of Donaldson--Thomas theory for Calabi--Yau 3-folds.

\subsection{\texorpdfstring{Orientation data for Calabi--Yau $m$-folds, $m$ odd}{Orientation data for Calabi--Yau m-folds, m odd}}
\label{od41}

\begin{dfn}
\label{od4def1}
A {\it Calabi--Yau\/ $m$-fold\/} is a smooth projective $\C$-scheme $X$ of complex dimension $m$ with a section $\th\in H^0(K_X)$ of the canonical bundle $K_X\ra X$ inducing an isomorphism~$\th:\O_X\ra K_X$.

Unless we explicitly say otherwise, Calabi--Yau $m$-folds $X$ are assumed to be projective, so their complex manifolds $X^\ran$ are compact. But in \S\ref{od42} we will discuss {\it noncompact Calabi--Yau\/ $m$-folds}, or {\it local Calabi--Yau\/ $m$-folds}, in which $X$ is quasi-projective rather than projective.

A Calabi--Yau $m$-fold has a natural spin structure $(\O_X,\jmath)$ in Definition \ref{od3def1}, in which $J=\O_X$ and $\jmath$ is the composition $\O_X^{\ot^2}\,{\buildrel\text{multiply}\over\longra}\,\O_X\,{\buildrel\th\over\longra}\,K_X$. We call this the {\it trivial spin structure}. Then in Definition \ref{od3def1}, the terms in $J$ in \eq{od3eq5}--\eq{od3eq14}  can all be omitted, as tensoring by $\O_X$ or $\pi_1^*(\O_X)$ is the identity. In \eq{od3eq9} we have $h^i(\cD^\bu\vert_{([F_1],[F_2])})\cong \Ext^i(F_1,F_2)$, and $\cD^\bu$ is the {\it Ext complex\/} on~$\M\t\M$.
\end{dfn}

We define orientation data, following Kontsevich and Soibelman~\cite[\S 5]{KoSo1}.

\begin{dfn}
\label{od4def2}
Let $X$ be a Calabi--Yau $m$-fold with $m$ odd, let $\M,\cU$ be as in Remark \ref{od3rem1}(e), and use the notation $\cD^\bu,K_\M,L_\M,\ldots$ of Definition \ref{od3def1}. Write $\cExact$ for the moduli stack of exact sequences $F_\bu=\bigl(0\ra F_1\ra F_2\ra F_3\ra 0\bigr)$ in $\coh(X)$, as an Artin $\C$-stack. Write $\pi_i:\cExact\ra\M$ for the morphism in $\Ho(\Art_\C)$ mapping $F_\bu$ to $F_i$ for $i=1,2,3$. Then we have a tautological exact sequence on $X\t\cExact$:
\e
\xymatrix@C=15pt{ 0 \ar[r] & (\id_X\t\pi_1)^*(\cU) \ar[r] & (\id_X\t\pi_2)^*(\cU) \ar[r] & (\id_X\t\pi_3)^*(\cU) \ar[r] & 0. }
\label{od4eq1}
\e
Taking $\Hom$ from this exact sequence to itself gives a diagram of perfect complexes on $X\t\cExact$, with rows and columns distinguished triangles:
\ea
\nonumber\\[-23pt]
\xymatrix@C=16pt@R=15pt{ &  \ar[d] &   \ar[d] &   \ar[d] \\
\ar[r] & {\begin{subarray}{l}\ts (\id_X\t\pi_3)^*(\cU)^\vee \\ \ts \ot(\id_X\t\pi_1)^*(\cU)\end{subarray}} \ar[r] \ar[d] & {\begin{subarray}{l}\ts (\id_X\t\pi_3)^*(\cU)^\vee \\ \ts \ot(\id_X\t\pi_2)^*(\cU)\end{subarray}} \ar[r] \ar[d] & {\begin{subarray}{l}\ts (\id_X\t\pi_3)^*(\cU)^\vee \\ \ts \ot(\id_X\t\pi_3)^*(\cU)\end{subarray}} \ar[d] \ar[r]^(0.8){[1]} &  \\
 \ar[r] & {\begin{subarray}{l}\ts (\id_X\t\pi_2)^*(\cU)^\vee \\ \ts \ot(\id_X\t\pi_1)^*(\cU)\end{subarray}} \ar[r] \ar[d] & {\begin{subarray}{l}\ts (\id_X\t\pi_2)^*(\cU)^\vee \\ \ts \ot(\id_X\t\pi_2)^*(\cU)\end{subarray}} \ar[r] \ar[d] & {\begin{subarray}{l}\ts (\id_X\t\pi_2)^*(\cU)^\vee \\ \ts \ot(\id_X\t\pi_3)^*(\cU)\end{subarray}} \ar[d] \ar[r]^(0.8){[1]} &  \\
\ar[r] & {\begin{subarray}{l}\ts (\id_X\t\pi_1)^*(\cU)^\vee \\ \ts \ot(\id_X\t\pi_1)^*(\cU)\end{subarray}} \ar[r] \ar[d]^(0.6){[1]} & {\begin{subarray}{l}\ts (\id_X\t\pi_1)^*(\cU)^\vee \\ \ts \ot(\id_X\t\pi_2)^*(\cU)\end{subarray}} \ar[r] \ar[d]^(0.6){[1]} & {\begin{subarray}{l}\ts (\id_X\t\pi_1)^*(\cU)^\vee \\ \ts \ot(\id_X\t\pi_3)^*(\cU)\end{subarray}} \ar[d]^(0.6){[1]} \ar[r]^(0.8){[1]} &  \\
&&& }
\nonumber\\[-23pt]
\label{od4eq2}
\ea

Pushing \eq{od4eq2} down along the projection $X\t\cExact\ra\cExact$ and using \eq{od3eq6} with $J=\O_X$ gives a diagram of perfect complexes on $\cExact$, with rows and columns distinguished triangles:
\ea
\nonumber\\[-23pt]
\xymatrix@C=16pt@R=15pt{ &  \ar[d] &   \ar[d] &   \ar[d] \\
\ar[r] & (\pi_3\t\pi_1)^*(\cD^\bu) \ar[r] \ar[d] & (\pi_3\t\pi_2)^*(\cD^\bu) \ar[r] \ar[d] & (\pi_3\t\pi_3)^*(\cD^\bu) \ar[d] \ar[r]^(0.8){[1]} &  \\
 \ar[r] & (\pi_2\t\pi_1)^*(\cD^\bu) \ar[r] \ar[d] & (\pi_2\t\pi_2)^*(\cD^\bu) \ar[r] \ar[d] & (\pi_2\t\pi_3)^*(\cD^\bu) \ar[d] \ar[r]^(0.8){[1]} &  \\
\ar[r] & (\pi_1\t\pi_1)^*(\cD^\bu) \ar[r] \ar[d]^(0.6){[1]} & (\pi_1\t\pi_2)^*(\cD^\bu) \ar[r] \ar[d]^(0.6){[1]} & (\pi_1\t\pi_3)^*(\cD^\bu) \ar[d]^(0.6){[1]} \ar[r]^(0.8){[1]} &  \\
&&& }
\nonumber\\[-23pt]
\label{od4eq3}
\ea
Generalizing Remark \ref{od3rem2}(b), if $\ra\cE^\bu\ra\cF^\bu\ra\cG^\bu\,{\buildrel[1]\over\longra}$ is a distinguished triangle then $\det\cF^\bu\cong\det\cE^\bu\ot\det\cG^\bu$. Thus equation \eq{od4eq3} gives an isomorphism of line bundles on $\cExact$:
\e
\begin{split}
(\pi_2\t\pi_2)^*(L_\M)\cong (\pi_1\t\pi_1)^*(L_\M)\ot(\pi_3\t\pi_3)^*(L_\M)&\\
{}\ot(\pi_1\t\pi_3)^*(L_\M)\ot(\pi_3\t\pi_1)^*(L_\M)&.
\end{split}
\label{od4eq4}
\e

Now $(\pi_i\t\pi_i)^*(L_\M)=\pi_i^*\ci\De_\M^*(L_\M)\cong\pi_i^*(K_\M)$ by \eq{od3eq8}. Also $(\pi_3\t\pi_1)^*(L_\M)=(\pi_1\t\pi_3)^*\ci\Si_\M^*(L_\M)\cong (\pi_1\t\pi_3)^*(L_\M)$ by \eq{od3eq12} and $m$ odd. Combining these with \eq{od4eq4} gives an isomorphism on~$\cExact$:
\e
\ka_\M:\pi_1^*(K_\M)\ot\pi_3^*(K_\M)
\ot(\pi_1\t\pi_3)^*(L_\M)^{\ot^2}\longra\pi_2^*(K_\M).
\label{od4eq5}
\e

We define {\it orientation data for\/} $\M$ to be an isomorphism class $[K_\M^{1/2}]$ of square root line bundles $K_\M^{1/2}$ for $K_\M$ (that is, a spin structure in Definition \ref{od3def2}) satisfying the additional condition \cite[Def.~15]{KoSo2} that choosing a representative $K_\M^{1/2}$ for the isomorphism class $[K_\M^{1/2}]$, there exists an isomorphism
\e
\la_\M:\pi_1^*(K_\M^{1/2})\ot\pi_3^*(K_\M^{1/2})\ot(\pi_1\t\pi_3)^*(L_\M)\longra\pi_2^*(K_\M^{1/2})
\label{od4eq6}
\e
on $\cExact$ with $\la_\M\ot\la_\M=\ka_\M$. Note that \eq{od4eq6} is analogous to~\eq{od3eq23}.

This is not quite the same as Kontsevich and Soibelman's definition of orientation data \cite[\S 5]{KoSo1}. As they are aiming at motivic invariants, they allow their square roots $K_\M^{1/2}$ to be {\it constructible line bundles}. Basically this means that they take a locally finite stratification $\M=\coprod_{i\in I}\M_i$ with $\M_i\subset\M$ a locally closed $\C$-substack,  and choose square roots $K_\M\vert_{\M_i}^{1/2}$ for each $i\in I$, but these need not glue continuously on the transitions between $\M_i$ and~$\M_j$. 

We define {\it constructible orientation data\/} $[K_\M^{1/2}]$ to be as above, but taking $K_\M^{1/2}$ to be a constructible line bundle rather than an ordinary line bundle, so that $K_\M^{1/2}\ot K_\M^{1/2}\cong K_\M$ and \eq{od4eq6} hold in constructible isomorphisms of constructible line bundles. Then Kontsevich and Soibelman's orientation data \cite[\S 5]{KoSo1} is our constructible orientation data. 

There is an obvious map from orientation data $[K_\M^{1/2}]$ to constructible orientation data, by regarding the line bundle $K_\M^{1/2}$ as a constructible line bundle.

We can also extend all the above to define ({\it constructible\/}) {\it orientation data\/} on moduli spaces $\oM$ of objects in $D^b\coh(X)$ from Remark \ref{od3rem1}(h). We replace $\M,\cU,\cD^\bu,K_\M,\ldots$ by $\oM,\cU^\bu,\bar\cD^\bu,K_\oM,\ldots,$ and we replace $\cExact$ by the moduli stack $\cDist$ of distinguished triangles $F_\bu=\bigl(\ra F_1^\bu\ra F_2^\bu\ra F_3^\bu\,{\buildrel[1]\over\longra}\bigr)$ in $D^b\coh(X)$, as a higher $\C$-stack. Then $\cExact\subset\cDist$ is an open substack. We replace $\pi_i$ by the projection $\pi_i:\cDist\ra\oM$ mapping $F_\bu$ to $F_i^\bu$ for $i=1,2,3$. Equation \eq{od4eq1} becomes a distinguished triangle on $X\t\cDist$ rather than an exact sequence, but the rest is essentially the same.
\end{dfn}

In the next proposition, the point is that the conditions on $[K_\M^{1/2}]$ of existence of $\xi_\M$ in \eq{od3eq23} involving direct sums in $\coh(X)$, and of existence of $\la_\M$ in \eq{od4eq6} involving short exact sequences in $\coh(X)$, are equivalent.

\begin{prop}
\label{od4prop1}
Let\/ $X$ be a Calabi--Yau\/ $m$-fold for\/ $m$ odd, with the trivial spin structure\/ $(\O_X,\jmath),$ and let\/ $\M,\oM$ be as in Remark\/ {\rm\ref{od3rem1}(e),(h)}. Then spin structures on\/ $\M$ \textup(or\/ $\oM$\textup) compatible with direct sums in Definition\/ {\rm\ref{od3def2}} are the same as orientation data on\/ $\M$ \textup(or\/ $\oM$\textup) in Definition\/~{\rm\ref{od4def2}}.	
\end{prop}

\begin{proof} Use the notation of Definitions \ref{od3def1}, \ref{od3def2} and \ref{od4def2}. Write $\Xi:\M\t\M\ra\cExact$ for the natural morphism acting on $\C$-points by
\begin{equation*}
([F_1],[F_2])\longmapsto \bigl(0\ra F_1\ra F_1\op F_2\ra F_2\ra 0\bigr).
\end{equation*}
Then $\pi_1\ci\Xi=\pi_1$, $\pi_2\ci\Xi=\Phi$, and $\pi_3\ci\Xi=\pi_2$. Using this we can show that the pullback of \eq{od3eq14} by $\De_{\M\t\M}:\M\t\M\ra\M\t\M\t\M\t\M$ is essentially the same as the pullback of \eq{od4eq3} by $\Xi:\M\t\M\ra\cExact$. So taking determinant line bundles shows that the pullback of \eq{od3eq15}  by $\De_{\M\t\M}$ is \eq{od4eq4}. So from the definitions of $\phi_\M$ in \eq{od3eq16} and $\ka_\M$ in \eq{od4eq5} we see that~$\phi_\M=\Xi^*(\ka_\M)$.

It follows that if $\la_\M$ in \eq{od4eq6} satisfies $\la_\M\ot\la_\M=\ka_\M$, then $\xi_\M=\Xi^*(\la_\M)$ in \eq{od3eq23} satisfies $\xi_\M\ot\xi_\M=\phi_\M$. This proves that if $[K_\M^{1/2}]$ is orientation data for $\M$, then $[K_\M^{1/2}]$ is a spin structure on $\M$ compatible with direct sums.

For the converse, observe that $\pi_1\t\pi_3:\cExact\ra\M\t\M$ is left inverse to $\Xi:\M\t\M\ra\cExact$. The fibre of $\pi_1\t\pi_3$ over $([F_1],[F_3])$ is $[\Ext^1(F_3,F_1)/\Hom(F_3,F_1)]$, which is contractible. One can define a morphism $G:\cExact\t\bA^1\ra\cExact$ with $G\vert_{\cExact\t\{0\}}=\Xi\ci(\pi_1\t\pi_3)$ and $G\vert_{\cExact\t\{1\}}=\id_{\cExact}$, such that if the exact sequence $F_\bu=\bigl(0\ra F_1\ra F_2\ra F_3\ra 0\bigr)$ corresponds to $\al\in\Ext^1(F_3,F_1)$, then $G\bigl([F_\bu],t\bigr)$ corresponds to $t\al\in\Ext^1(F_3,F_1)$. Thus $\pi_1\t\pi_3$ and $\Xi$ are $\bA^1$-homotopy equivalences.

Recall here that in general a morphism $f\colon Y\to Z$ is called an $\bA^1$-{\it homotopy equivalence\/} if there exist morphisms $g\colon Z\to Y,$ $ H\colon Y\t\bA^1\to Y,$ and $J\colon Z\t\bA^1\to Z$ with $H\circ(\id_Y\t i_0)=g\circ f,$ $H\circ(\id_Y\t i_1)=\id_Y,$ $J\circ(\id_Z\t i_0)=f\circ g,$ $J\circ(\id_Z\t i_1)=\id_Z,$ for the inclusion $i_k$ of a point at $k\in\bA^1.$

Now if $L\ra Z$ is a line bundle, then pullback $f^*$ induces an equivalence of categories from the groupoid $S_L$ of square roots $L^{1/2}$ on $Z$ to the groupoid $S_{f^*(L)}$ of square roots $(f^*(L))^{1/2}$ on $Y,$ as follows. From the commutative diagrams
\[
\xymatrix{
	&	*+[r]{S_{(fg)^*(L)}}\\
S_L\ar[rd]_{\id_{S_L}}\ar[ru]^{g^*\circ f^*}\ar[r]^{J^*}	&	*+[r]{S_{J^*(L)}}\ar[u]_{(\id_Z\t i_0)^*}\ar[d]^{(\id_Z\t i_1)^*}\\
	&	*+[r]{S_L}
}
\xymatrix{
	&	*+[r]{S_{(gf)^*(f^*(L))}}\\
S_{f^*(L)}\ar[rd]_{\id_{S_{f^*(L)}}}\ar[ru]^{f^*\circ g^*}\ar[r]^{H^*}	&	*+[r]{S_{H^*(f^*(L))}}\ar[u]_{(\id_Y\t i_0)^*}\ar[d]^{(\id_Y\t i_1)^*}\\
	&	*+[r]{S_{f^*(L)}}
}
\]
it suffices to prove that each of the left vertical functors is an equivalence. For example, consider $\id_Z\t i_0\colon Z\to Z\t\bA^1$ and a complex line bundle $K\to Z\t\bA^1.$ The key point is that, up to isomorphism, any two square roots of $K$ differ by a principal $\Z_2$-bundle over $Z\t\bA^1$ and similarly for $(\id_Z\t i_0)^*K\to Z.$ Pullback along $\id_Z\t i_0$ identifies the groups of isomorphism classes of principal $\Z_2$-bundles over $Z$ and $Z\t\bA^1,$ so $\pi_0(\id_Z\t i_0)^*\colon \pi_0(S_K)\to \pi_0(S_{(\id_Z\t i_0)^*K})$ is an equivariant map of torsors, which are automatically bijective. Similarly, $\pi_1(\id_Z\t i_0)^*$ on automorphism groups is equivariant over the groups of $\Z_2$-valued maps.

Therefore the map $\Xi^*:\la_\M\mapsto\xi_\M$ above induces a 1-1 correspondence between $\la_\M$ in \eq{od4eq6} with $\la_\M\ot\la_\M=\ka_\M$ and $\xi_\M$ in \eq{od3eq23} with $\xi_\M\ot\xi_\M=\phi_\M$. So orientation data on $\M$ and spin structures on $\M$ compatible with direct sums are the same. The analogous argument works for~$\oM$.	
\end{proof}

Combining Theorem \ref{od3thm2} and Proposition \ref{od4prop1} yields one of our main results, which as in \S\ref{od43} solves a long-standing problem in Donaldson--Thomas theory:

\begin{thm}
\label{od4thm1}
Let\/ $X$ be a Calabi--Yau\/ $3$-fold, and use the notation\/ $\M,\ab\oM,\ab K_\M,\ab\ldots$ of Definitions\/ {\rm\ref{od3def1}} and\/ {\rm\ref{od4def2}}. Then we can construct canonical orientation data\/ $[K_\M^{1/2}]$ on\/ $\M$ and\/ $[K_\oM^{1/2}]$ on\/ $\oM,$ with\/~$[K_\oM^{1/2}]\vert_\M=[K_\M^{1/2}]$.
\end{thm}

\begin{rem}
\label{od4rem1}
{\bf(a)} So far as the authors know, Theorem \ref{od4thm1} is the first complete proof of existence of orientation data for {\it any example\/} of a (projective) Calabi--Yau 3-fold. We construct {\it canonical\/} orientation data for {\it all\/} Calabi--Yau 3-folds.

\smallskip 

\noindent{\bf(b)} The authors prefer the term `spin structure', as in \S\ref{od2}--\S\ref{od3}, to `orientation data'. But the term `orientation data' is already established in the literature.
\smallskip 

\noindent{\bf(c)} We review the literature on orientation data. The argument of Nekrasov and Okounkov \cite[\S 6]{NeOk} explained in Theorem \ref{od3thm1} is sufficient to establish the existence of square roots $K_\M^{1/2}$ in some situations, for instance if we consider a stable fine moduli scheme $\M^\al_{\rm st}(\tau)$ rather than a moduli stack. But it does not seem to help with either making a canonical choice of $[K_\M^{1/2}]$, or proving compatibility with direct sums or exact sequences as in \eq{od3eq23} or~\eq{od4eq6}.

Davison \cite{Davi1} constructs orientation data for moduli stacks of objects in certain classes of 3-Calabi--Yau categories coming from representation theory, including derived categories of quivers with superpotential. See \cite{Davi2} for a survey.

Maulik and Toda \cite{MaTo} discuss `orientation data' (really meaning square roots $K_S^{1/2}$) as part of a programme to define Gopakumar--Vafa invariants of Calabi--Yau 3-folds using perverse sheaf methods (see \S\ref{od43}). They are mainly interested in constructing square roots $K_S^{1/2}$ on certain moduli schemes $S$, and impose the condition that $K_S^{1/2}$ should be trivial on the fibres of some fibrations $\pi:S\ra T$. In a sequel \cite{Toda}, Toda constructs square roots $K_\M\vert_\cN^{1/2}$ for certain substacks $\cN\subset\M$ of sheaves of dimension 0 and 1 on a Calabi--Yau 3-fold.

Thomas \cite{Thom2} defines and studies Vafa--Witten invariants of projective surfaces $S$ via Donaldson--Thomas theory of the moduli stack $\M$ of compactly-supported coherent sheaves on the noncompact Calabi--Yau 3-fold $X=K_S$, as in \S\ref{od42}. (More precisely, Thomas uses a modification of $\M$, parametrizing sheaves with fixed determinants over $S$, and with centre of mass zero in each $K_S$ fibre.) The action of $T=\C^*$ rotating the fibres of $K_S$ induces a $T$-action on $\M$, so we can consider the fixed substack $\M^T$. The natural symmetric obstruction theory $\phi:\cE^\bu\ra\bL_\M$ is $T$-equivariant, where the symmetry $\om:(\cE^\bu)^\vee\,{\buildrel\cong\over\longra}\, \cE^\bu[-1]$ has $T$-weight 1. Under these conditions, using that $\om$ has odd $T$-weight, he shows \cite[Prop.~2.6]{Thom2} that the restriction of $K_\M=\det\cE^\bu$ to $\M^T$ has a natural square root $K_\M\vert_{\M^T}^{1/2}$, that is, canonical $T$-localized orientation data exists.

Shi \cite{Shi} constructs orientation data for the moduli stack of compactly supported coherent sheaves on the noncompact Calabi--Yau 3-fold $K_{\CP^2}$, as in~\S\ref{od42}. 

\end{rem}

Proposition \ref{od4prop1} also justifies the following definition:

\begin{dfn}
\label{od4def3}
Let $X$ be a Calabi--Yau $m$-fold for $m$ odd, with the trivial spin structure $(\O_X,\jmath)$, and let $\M,\oM$ be as in Remark \ref{od3rem1}(e),(h). We define {\it strong orientation data\/ $(K_\M^{1/2},\xi_\M)$ on\/} $\M$ (or {\it strong orientation data\/ $(K_\oM^{1/2},\xi_\oM)$ on\/} $\oM$) to be a strong spin structure on $\M$ (or $\oM$) compatible with direct sums, in the sense of Definition~\ref{od3def3}.
\end{dfn}

\subsection{\texorpdfstring{Extension to noncompact Calabi--Yau $m$-folds}{Extension to noncompact Calabi--Yau m-folds}}
\label{od42}

As in Definition \ref{od4def1}, unless we explicitly say otherwise we suppose Calabi--Yau $m$-folds $X$ are projective, so their underlying complex manifolds $X^\ran$ are compact. We now explain how to generalize \S\ref{od41} to the noncompact case.

\begin{dfn}
\label{od4def4}
A {\it noncompact Calabi--Yau\/ $m$-fold}, or {\it local Calabi--Yau\/ $m$-fold}, is a smooth quasiprojective $\C$-scheme $X$ of complex dimension $m$, which is not projective (so not proper), with a section $\th\in H^0(K_X)$ of the canonical bundle $K_X$ inducing an isomorphism $\th:\O_X\ra K_X$. Examples include canonical bundles $X=K_S$ for a smooth projective $\C$-scheme $S$ of dimension~$m-1$.

For a noncompact Calabi--Yau $m$-fold $X$, instead of $\coh(X)$ we consider the full subcategory $\coh_\cs(X)\subset\coh(X)$ of {\it compactly-supported coherent sheaves\/} $F$ (i.e.\ $\mathop{\rm supp}F$ is proper), and instead of $D^b\coh(X)$ we consider the subcategory $\Perf_\cs(X)\subset \Perf(X)$ of perfect complexes $F^\bu$ whose cohomology sheaves $h^i(F^\bu)$ for $i\in\Z$ are compactly-supported. Then $\coh_\cs(X)\subset\Perf_\cs(X)$. 

We write $\M$ for the moduli stack of objects in $\coh_\cs(X)$, as an Artin $\C$-stack, and $\oM$ for the moduli stack of objects in $\Perf_\cs(X)$, as a higher $\C$-stack.

We must work with $\coh_\cs(X)$ rather than $\coh(X)$, and with $\Perf_\cs(X)$ rather than $\Perf(X)=D^b\coh(X)$ or $D^b\coh_\cs(X)$, to make the theory well behaved. For example, non-compactly supported coherent sheaves $E$ on $X$, such as $E=\O_X$, may have infinite-dimensional automorphism groups $\Aut(E)$, so the moduli stack of objects in $\coh(X)$ does not exist as an Artin $\C$-stack.

The material of \S\ref{od3}--\S\ref{od41} extends to quasi-projective $X$, replacing $\coh(X),\ab D^b\coh(X)$ by $\coh_\cs(X),\Perf_\cs(X)$, in a straightforward way. In particular, the projections $(\pi_2)_*,(\pi_{23})_*$ in \eq{od3eq5}--\eq{od3eq6} needed $X$ proper to be well defined. But as we are working with compactly-supported sheaves, $\mathop{\rm supp}\cU\subset X\t\M$ is proper over $\M$, and this is sufficient to define $(\pi_2)_*,(\pi_{23})_*$ when $X$ is not proper.

Thus Definition \ref{od4def2} extends to give notions of ({\it constructible\/}) {\it orientation data\/} on $\M$ and $\oM$ when $X$ is a noncompact Calabi--Yau $m$-fold for $m$ odd.
\end{dfn}

Any smooth quasi-projective $m$-fold $X$ may be embedded as an open $\C$-subscheme $X\subset Y$ of a smooth projective $m$-fold $Y$. Then compactly-supported coherent sheaves or perfect complexes on $X$ extend uniquely to coherent sheaves or perfect complexes on $Y$ which are zero on a neighbourhood of $Y\sm X$. Hence we have inclusions $\M_X\hookra\M_Y$, $\oM_X\hookra\oM_Y$ of the moduli stacks of objects in $\coh_\cs(X)\hookra\coh(Y)$ and~$\Perf_\cs(X)\hookra\Perf(Y)=D^b\coh(Y)$.

A spin structure $(J_Y,\jmath_Y)$ on $Y$ may be restricted to a spin structure $(J_X,\jmath_X)$ on $X$. Then a spin structure $[K_{\M_Y}^{1/2}]$ on $\M_Y$ compatible with direct sums for $(J_Y,\jmath_Y)$ restricts to a spin structure $[K_{\M_X}^{1/2}]=[K_{\M_Y}^{1/2}\vert_{\M_X}]$ on $\M_X\subset\M_Y$ compatible with direct sums for $(J_X,\jmath_X)$, and similarly for $\oM_X,\oM_Y$. Hence Proposition \ref{od4prop1} implies:

\begin{prop}
\label{od4prop2}
Let\/ $X$ be a noncompact Calabi--Yau\/ $m$-fold. Suppose we can embed\/ $X\subset Y$ as an open subscheme of a smooth projective\/ $m$-fold\/ $Y,$ and that\/ $Y$ has a spin structure\/ $(J_Y,\jmath_Y)$ which restricts to the trivial spin structure\/ $(\O_X,\jmath)$ on\/ $X$. Then any spin structure\/ $[K_{\M_Y}^{1/2}]$ on\/ $\M_Y$ \textup(or\/ $[K_{\oM_Y}^{1/2}]$ on\/ $\oM_Y$\textup) compatible with direct sums, as in Definition\/ {\rm\ref{od3def2},} restricts to orientation data\/ $[K_{\M_Y}^{1/2}\vert_{\M_X}]$ on\/ $\M_X$ \textup(or\/ $[K_{\oM_Y}^{1/2}\vert_{\oM_X}]$ on\/ $\oM_X$\textup), as in Definition\/ {\rm\ref{od4def4}}.
\end{prop}

Combining Theorem \ref{od3thm2} and Proposition \ref{od4prop2} yields:

\begin{thm}
\label{od4thm2}
Let\/ $X$ be a noncompact Calabi--Yau\/ $3$-fold. Suppose we can embed\/ $X\subset Y$ as an open subscheme of a smooth projective\/ $3$-fold\/ $Y,$ and that\/ $Y$ has a spin structure\/ $(J_Y,\jmath_Y)$ which restricts to the trivial spin structure\/ $(\O_X,\jmath)$ on\/ $X.$ Then we can construct canonical orientation data\/ $[K_\M^{1/2}]$ on\/ $\M$ and\/ $[K_\oM^{1/2}]$ on\/ $\oM,$ with\/~$[K_\oM^{1/2}]\vert_\M=[K_\M^{1/2}]$.
\end{thm}

\begin{rem}
\label{od4rem2}
Arkadij Bojko, a PhD student supervised by the first author, is working on an alternative construction of the orientation data $[K_\M^{1/2}],[K_\oM^{1/2}]$ in Theorem \ref{od4thm2} which is independent of choice of spin compactification $Y,(J_Y,\jmath_Y)$ of $X$, and also works even if $X$ does not admit a spin compactification.
\end{rem}

\begin{ex}
\label{od4ex1}
Let $S$ be a smooth projective complex surface. Define $X$ to be the canonical bundle $K_S$, considered as a noncompact Calabi--Yau 3-fold. Define $Y={\mathbb P}(K_S\op\O_S)$, a $\CP^1$-bundle over $S$, which is a smooth projective 3-fold. Define a smooth divisor $D\subset Y$ by $D={\mathbb P}(0\op\O_S)\subset{\mathbb P}(K_S\op\O_S)$. Then there is a natural identification $X\cong Y\sm D$ by mapping $(s,\ka)$ in $X=K_S$ to $(s,[\ka,1])$ in $Y$, for $s\in S$ and~$\ka\in K_S\vert_s$.

Write $\O_Y(D)$ for the usual line bundle on $Y$ with section $s_D$ vanishing transversely on the divisor $D\subset Y$. Then by a standard computation one can show that there is a natural isomorphism $K_Y\cong \O_Y(D)^{-2}$. Hence we get a spin structure $(J,\jmath)$ on $Y$ with $J=\O_Y(D)^{-1}$ and $\jmath$ the isomorphism $\O_Y(D)^{-2}\ra K_Y$. On restricting to $Y\sm D\cong X$, we have a section $s_D\vert_{Y\sm D}^{-1}$ of $J\vert_{Y\sm D}$ giving an isomorphism $J\vert_{Y\sm D}\cong\O_{Y\sm D}$, which identifies $(J,\jmath)\vert_{Y\sm D}$ with the trivial spin structure on $X$. Thus Theorem \ref{od4thm2} constructs canonical orientation data on $\M_X$ and~$\oM_X$.	
\end{ex}

\begin{rem}
\label{od4rem3}
In \cite[\S 2.2]{Shi}, Shi outlines a construction of canonical orientation data $[K_\M^{1/2}]$ for $\coh_\cs(K_{\CP^2})$, which she attributes to unpublished work of Yukinobu Toda. Toda's construction also works for $\coh_\cs(K_S)$ for any smooth projective surface $S$. We expect this gives the same answer as Example~\ref{od4ex1}.
\end{rem}

\subsection{Donaldson--Thomas theory and its generalizations}
\label{od43}

Finally we discuss Donaldson--Thomas theory of Calabi--Yau 3-folds, and various generalizations of it in which orientation data is important, so our Theorems \ref{od4thm1} and \ref{od4thm2} make a new contribution. These generalizations are best explained using the Derived Algebraic Geometry of To\"en and Vezzosi \cite{Toen1,Toen2,ToVe1,ToVe2}, and the shifted symplectic geometry of Pantev, To\"en, Vaqui\'e and Vezzosi \cite{PTVV}. The next remark provides some very brief orientation on these.

\begin{rem}
\label{od4rem4}
{\bf(a)} Derived Algebraic Geometry studies {\it derived\/ $\C$-stacks\/} $\bS$. They form an $\iy$-category $\DSta_\C$. We use bold characters $\bS,\bcM,\bs f:\bS\ra\bs T,\ldots$ to indicate derived stacks and their morphisms. Generalizing \eq{od3eq1}, \eq{od3eq2} and \eq{od3eq4}, one may regard derived $\C$-stacks as $\iy$-functors
\begin{equation*}
\bS:\{\text{derived commutative $\C$-algebras}\}\longra\{\text{$\iy$-groupoids}\}
\end{equation*}
satisfying many conditions. Here two (essentially equivalent) possible models for derived commutative $\C$-algebras are simplicial $\C$-algebras, and commutative differential graded $\C$-algebras (cdgas) in nonpositive degrees.
\smallskip

\noindent{\bf(b)} There is a full and faithful {\it inclusion functor\/} $\io:\HSta_\C\hookra\DSta_\C$. We use this to identify $\HSta_\C$ as a full $\iy$-subcategory of $\DSta_\C$. There is a {\it classical truncation functor\/} $t_0:\DSta_\C\ra\HSta_\C$. We write the classical truncations of $\bS,\bcM,\ldots$ as $S,\cM,\ldots.$ For any $\bS$ in $\DSta_\C$ with classical truncation $S=t_0(\bS)$ there is a natural inclusion morphism $i_\bS:S=\io(S)\hookra\bS$. We can regard $\bS$ as a kind of `infinitesimal formal thickening' of $S$ in the `derived' directions.
\smallskip

\noindent{\bf(c)} We call a derived stack $\bS$ a {\it derived scheme}, or a {\it derived Artin stack}, if $S=t_0(\bS)$ lies in  $\Sch_\C\subset\HSta_\C$ or $\Art_\C\subset\HSta_\C$. We write $\DSch_\C\subset\DArt_\C\subset\DSta_\C$ for the full $\iy$-subcategories of these.
\smallskip

\noindent{\bf(d)} Let $\bS$ be a derived $\C$-stack. As in Remark \ref{od3rem1}(g), To\"en \cite[\S 3.1.7, \S 4.2.4]{Toen1} defines a triangulated category $L_\qcoh(\bS)$ of modules which should be used instead of $D\qcoh(\bS)$. It has a triangulated subcategory $\Perf(\bS)$ of {\it perfect complexes}. An object $\cE^\bu$ in $\Perf(\bS)$ has dual $(\cE^\bu)^\vee$ and determinant line bundle~$\det(\bL_\bS)$.
\smallskip

\noindent{\bf(e)} Let $\bS$ be a locally finitely presented derived $\C$-stack. Then To\"en and Vezzosi \cite[\S 4.2.5]{Toen1}, \cite[\S 1.4]{ToVe1} (see also Lurie \cite[\S 3.2]{Luri}) define the {\it cotangent complex\/}
$\bL_\bS$ in $L_\qcoh(\bS)$. It is perfect, as $\bS$ is locally finitely presented, so it has a determinant line bundle $\det\bL_\bS$ on $\bS$, which we call the {\it canonical bundle\/} of $\bS$ and write $K_\bS$. Its dual $\bT_\bS=\bL_\bS^\vee$ is called the {\it tangent complex}.
\smallskip

\noindent{\bf(f)} Let $X$ be a smooth projective $\C$-scheme. We often write $\bcM$ and $\boM$ for the derived moduli stacks of objects in $\coh(X)$ and $D^b\coh(X)$, respectively. These exist as locally finitely presented derived $\C$-stacks by To\"en--Vaqui\'e \cite{ToVa}, and $\bcM\subset\boM$ is an open substack. They have classical truncations $\M=t_0(\bcM)$ and $\oM=t_0(\boM)$, for $\M,\oM$ as in Remark~\ref{od3rem1}(e),(h).
\smallskip

\noindent{\bf(g)} Let $X$ be a Calabi--Yau $m$-fold with the trivial spin structure $(\O_X,\jmath)$. Then {\bf(f)} gives a derived stack $\bcM$ with perfect cotangent complex $\bL_\bcM$. We pull this back to a perfect complex $\bL_{i_\bcM}^*(\bL_\bcM)$ on $\cM=t_0(\bcM)$. Then using \eq{od3eq5} with $J=\O_X$ we can show that $\bL_{i_\bcM}^*(\bL_\bcM)\cong (\cC^\bu)^\vee[-1]$. So taking determinant line bundles gives an isomorphism, for $K_\M$ as in Definition~\ref{od3def1}:
\e
\bL_{i_\bcM}^*\bigl(\det(\bL_\bcM)\bigr)\cong K_\M.
\label{od4eq7}
\e
This justifies calling $K_\M$ the `canonical bundle' of $\M$ or $\bcM$.
\smallskip

\noindent{\bf(h)} It is a general principle that in passing from classical to derived geometry, one should replace vector bundles (such as the (co)tangent bundles $TS,T^*S$ of a manifold or smooth $\C$-scheme $S$) by perfect complexes (such as the (co)tangent complexes $\bT_\bS,\bL_\bS$ of a locally finitely presented derived $\C$-stack $\bS$). It is also a general principle, known as Kontsevich's `hidden smoothness' philosophy, that (nice) derived stacks behave a lot like manifolds or smooth schemes.

Following these principles, Pantev, To\"en, Vaqui\'e, and Vezzosi \cite{PTVV} introduced a derived version of symplectic geometry. Given a locally finitely presented derived $\C$-stack $\bS$, they defined notion of $k$-{\it shifted\/ $p$-form\/} $\om_0$ on $\bS$ (basically an element of $\bH^k(\La^p\bL_\bS)$) and $k$-{\it shifted closed\/ $p$-form\/} $\om=(\om_0,\om_1,\ldots)$ on $\bS$ (more complicated, but with a map $(\om_0,\om_1,\ldots)\mapsto\om_0$ to $k$-shifted $p$-forms). They define a $k$-{\it shifted symplectic structure\/} $\om$ on $\bS$ to be a $k$-shifted closed 2-form $\om$ on $\bS$ whose associated 2-form $\om_0$ induces an equivalence $\om_0\cdot:\bT_\bS\ra\bL_\bS[k]$.
They define $k$-{\it shifted Lagrangians\/} $\bs i:\bs L\ra\bS$ in $k$-shifted symplectic~$(\bS,\om)$.
\smallskip

\noindent{\bf(i)} Pantev et al.\ \cite[\S 2.1]{PTVV} show that if $X$ is a Calabi--Yau $m$-fold then $\bcM,\boM$ from {\bf(f)} have natural $(2-m)$-shifted symplectic structures $\om$. We can think of $\om$ as a geometric incarnation of Serre duality $\Ext^i(F,F)\cong \Ext^{m-i}(F,F)^*$ on the Calabi--Yau $m$-fold $X$. The following morphism is $(2-m)$-shifted Lagrangian:
\e
\bs\pi_1\t\bs\pi_2\t\bs\pi_3:\bcExact\longra(\bcM\t\bcM\t\bcM,\om\boxplus -\om\boxplus\om),
\label{od4eq8}
\e
where generalizing $\cExact$ in Definition \ref{od4def2}, we write $\bcExact$ for the derived moduli stack of exact sequences $F_\bu=\bigl(0\ra F_1\ra F_2\ra F_3\ra 0\bigr)$ in~$\coh(X)$.
\smallskip

\noindent{\bf(j)}  Following \cite{BBBJ,BBJ,BBDJS,BJM,Joyc}, several interesting geometric constructions start with $-1$-shifted symplectic derived $\C$-stacks $(\bS,\om)$. This holds because as in Ben-Bassat--Bussi--Brav--Joyce \cite{BBBJ,BBJ}, such $(\bS,\om)$ are (\'etale- or smooth-) locally modelled on the critical locus $\Crit(f:U\ra{\mathbb A}^1)$ of a regular function $f$ on a smooth $\C$-scheme $U$. Thus constructions for critical loci, such as perverse sheaves of vanishing cycles,  extend to $-1$-shifted symplectic derived $\C$-stacks.	
\end{rem}

The next theorem summarizes results of Ben-Bassat, Bussi, Brav, Dupont, Joyce, Meinhardt, and Szendr\"oi \cite{BBBJ,BBJ,BBDJS,BJM,Joyc}. They were inspired by ideas of Behrend \cite{Behr} and Kontsevich and Soibelman \cite{KoSo1,KoSo2,KoSo3} that predated the shifted symplectic geometry of Pantev--To\"en--Vaqui\'e--Vezzosi~\cite{PTVV}.

\begin{thm}
\label{od4thm3}
{\bf(a) \cite{BBBJ,BJM}} Suppose\/ $(\bS,\om)$ is a finite type\/ $-1$-shifted symplectic derived\/ $\C$-scheme or\/ $\C$-stack, and we are given an isomorphism class\/ $[K_\bS^{1/2}]$ of square roots\/ $K_\bS^{1/2}$ \textup(essentially, a \begin{bfseries}spin structure\end{bfseries} as in {\rm \S\ref{od32}}\textup). Then we can construct a natural motive\/ $MF_{\bS,\om}$ in a ring of motives\/ $\Mot_S$ on\/ $S=t_0(\bS).$ Here if\/ $(\bS,\om)$ is locally modelled on\/ $\Crit(f:U\ra{\mathbb A}^1)$ then\/ $MF_{\bS,\om}$ is locally modelled on the motivic vanishing cycle\/ $MF_{U,f}^{\rm mot,\phi}$ of $f$ from\/~{\rm\cite{DeLo}}.
\smallskip

\noindent{\bf(b) \cite{BBBJ,BBDJS}} Suppose\/ $(\bS,\om)$ is a\/ $-1$-shifted symplectic derived\/ $\C$-scheme or\/ $\C$-stack, and we are given a square root\/ $K_\bS^{1/2}$ \textup(essentially, a \begin{bfseries}strong spin structure\end{bfseries} as in {\rm \S\ref{od33}}\textup). Then we can construct a natural perverse sheaf\/ $P_{\bS,\om}^\bu$ on\/ $S=t_0(\bS)$. Here if\/ $(\bS,\om)$ is locally modelled on\/ $\Crit(f:U\ra{\mathbb A}^1)$ then\/ $P_{\bS,\om}^\bu$ is locally modelled on the perverse sheaf of vanishing cycles\/ $PV_{U,f}^\bu$ from\/~{\rm\cite[\S 2.4]{BBDJS}}.	
\end{thm}

\begin{rem}
\label{od4rem5}
{\bf(On generalizations of Donaldson--Thomas theory of Calabi--Yau 3-folds.)} The programme of \cite{BBBJ,BBJ,BBDJS,BJM,Joyc} was aimed in part at extending Donaldson--Thomas theory, and we now explain this.
\smallskip

\noindent{\bf(a) Classical Donaldson--Thomas invariants.} The {\it Donaldson--Thomas invariants\/} $DT^\al(\tau)$ of a Calabi--Yau 3-fold $X$ are integers or rational numbers `counting' moduli spaces $\M^\al_{\rm ss}(\tau)$ (as a scheme or Artin stack) of $\tau$-semistable coherent sheaves on $X$ with Chern character $\al$, for $\tau$ a suitable stability condition (e.g.\ Gieseker stability). They are unchanged under deformations of $X$. They were proposed by Donaldson and Thomas \cite{DoTh}, and defined by Thomas \cite{Thom1} in 1998 when $\M^\al_{\rm st}(\tau)=\M^\al_{\rm ss}(\tau)$, that is, when $\M^\al_{\rm ss}(\tau)$ contains no strictly semistables, and by Joyce and Song \cite{JoSo} in 2008 in the general case.

The $DT^\al(\tau)$ satisfy additive and multiplicative identities, including a wall-crossing formula \cite[Th.~5.18]{JoSo} under change of stability condition~$\tau$.

\smallskip

\noindent{\bf(b) Extension to motivic invariants.} In a seminal 2005 paper, Behrend \cite{Behr} showed that when $\M^\al_{\rm st}(\tau)=\M^\al_{\rm ss}(\tau)$ we may write $DT^\al(\tau)$ as a weighted Euler characteristic $\chi(\M^\al_{\rm ss}(\tau),\nu)$, where $\nu:\M^\al_{\rm ss}(\tau)\ra\Z$ is a constructible function, the {\it Behrend function}. This showed $DT^\al(\tau)$ has a motivic nature.

Here a {\it motivic invariant\/} is a function $\mu:\{$finite type $\C$-schemes$\}\ra\Mot$ taking values in a commutative ring $\Mot$, satisfying $\mu(Y)=\mu(X)+\mu(Y\sm X)$ for $X\subset Y$ closed and $\mu(Y\t Z)=\mu(Y)\mu(Z)$. Examples include the Euler characteristic $\chi$ (with $\Mot=\Z$), and virtual Poincar\'e and Hodge polynomials. 

In 2008, Kontsevich and Soibelman \cite{KoSo1} refined Donaldson--Thomas invariants $DT^\al(\tau)$, taking values in $\Z$ or $\Q$, to {\it motivic Donaldson--Thomas invariants\/} $DT^\al_{\rm mot}(\tau)$ taking values in a commutative ring $\Mot$. In Behrend's formula $DT^\al(\tau)=\chi(\M^\al_{\rm ss}(\tau),\nu)$, they wanted to replace the Euler characteristic $\chi$ by other motivic invariants. They invented the notion of {\it orientation data\/} on $\coh(X)$ \cite[\S 5]{KoSo1}, which they needed to define~$DT^\al_{\rm mot}(\tau)$.

Kontsevich and Soibelman's theory \cite{KoSo1} depended on a number of conjectures and sketch proofs. One of these conjectures --- the existence of orientation data --- we proved in Theorem \ref{od4thm1}. After Pantev--To\"en--Vaqui\'e--Vezzosi \cite{PTVV} in 2011, the theory could be rewritten using $-1$-shifted symplectic structures. Theorem \ref{od4thm3}(a) gives a new definition of $DT^\al_{\rm mot}(\tau)$, resolving some issues in~\cite{KoSo1}. 

Just to define the $DT^\al_{\rm mot}(\tau)$, one needs an isomorphism class $[K_\M^{1/2}]$ of (constructible) square roots $K_\M^{1/2}$ of $K_\M$, by \eq{od4eq7} and Theorem \ref{od4thm3}(a). To make the $DT^\al_{\rm mot}(\tau)$ satisfy multiplicative identities in $\Mot$, including the wall-crossing formula \cite[\S 2.3]{KoSo1} under change of stability condition, Kontsevich and Soibelman also needed an isomorphism $\la_\M$ in \eq{od4eq6} satisfying~$\la_\M\ot\la_\M=\ka_\M$.

Meinhardt \cite{Mein} gives an introduction to motivic Donaldson--Thomas theory.

\smallskip

\noindent{\bf(c) Categorification using perverse sheaves.} Behrend's 2005 paper \cite{Behr} also motivated a second generalization of Donaldson--Thomas theory. He observed that following a heuristic argument of Thomas \cite{Thom1}, one expects Calabi--Yau 3-fold moduli schemes $\M^\al_{\rm ss}(\tau)$ to be locally modelled on critical loci $\Crit (f:U\ra{\mathbb A}^1)$, and if $\M^\al_{\rm ss}(\tau)$ is of this local form then $\nu$ is the pointwise Euler characteristic $\chi_{PV_{U,f}^\bu}$ of the perverse sheaf of vanishing cycles~$PV_{U,f}^\bu$.

This suggested the conjecture that there should exist a perverse sheaf $P^\bu_{\M^\al_{\rm ss}(\tau)}$ on $\M^\al_{\rm ss}(\tau)$, locally modelled on $PV_{U,f}^\bu$ when $\M^\al_{\rm ss}(\tau)$ is locally modelled on $\Crit f$, with $\chi_{P^\bu_{\M^\al_{\rm ss}(\tau)}}=\nu$. This would imply that when $\M^\al_{\rm st}(\tau)=\M^\al_{\rm ss}(\tau)$
\begin{equation*}
DT^\al(\tau)=\chi(\M^\al_{\rm ss}(\tau),\nu)=\sum_{i\in\Z}(-1)^i\dim\bH^i(P^\bu_{\M^\al_{\rm ss}(\tau)}).	
\end{equation*}
Thus, the hypercohomology $\bH^*(P^\bu_{\M^\al_{\rm ss}(\tau)})$, thought of as some kind of cohomology of $\M^\al_{\rm ss}(\tau)$, would be a graded vector space categorifying the Donaldson--Thomas invariant $DT^\al(\tau)$, giving a new interpretation of $DT^\al(\tau)$ as a dimension. This conjecture is proved in Theorem \ref{od4thm3}(b), but requires a square root $K_{\M^\al_{\rm ss}(\tau)}^{1/2}$ on $\M^\al_{\rm ss}(\tau)$ to define the perverse sheaf~$P^\bu_{\M^\al_{\rm ss}(\tau)}$.
\smallskip

\noindent{\bf(d) Cohomological Hall algebras.} In 2010, Kontsevich and Soibelman \cite{KoSo3} made another important contribution to Donaldson--Thomas theory, by introducing {\it Cohomological Hall Algebras\/} ({\it CoHAs\/}). They did this only for representations of quivers with superpotential, a simple class of 3-Calabi--Yau categories, but it is clear that one should try to generalize it to Calabi--Yau 3-folds.
 
Expressed using shifted symplectic geometry, rather than Kontsevich and Soibelman's original picture, the idea is this. Let $(\bcM,\om)$ be the $-1$-shifted derived moduli stack of objects in $\coh(X)$ from Remark \ref{od4rem4}(f),(i). Suppose we are given a square root $K_\cM^{1/2}$ on $\M$. Then \eq{od4eq7} and Theorem \ref{od4thm3}(b) give a perverse sheaf $P_\bcM^\bu$ on $\cM$, with hypercohomology $\bH^*(P_\bcM^\bu)$. The goal is to define a multiplication $\star$ on $\bH^*(P_\bcM^\bu)$, making it into an associative algebra. In String Theory, this should be interpreted as an `algebra of BPS states'.

The rough idea for doing this, for $\cExact,\pi_i$ as in Definition \ref{od4def2}, is to construct a morphism in $D^b_c(\cExact)$, with an associativity property: 
\e
\mu:\pi_1^!(P_\bcM^\bu)\ot\pi_3^!(P_\bcM^\bu)\longra \pi_2^!(P_\bcM^\bu)[\text{local integer shift}],
\label{od4eq9}
\e
and then for $\al,\be\in\bH^*(P_\bcM^\bu)$ (ignoring certain properness issues) to define
\begin{equation*}
\al\star\be= (\pi_2)_!\ci\mu(\pi_1^!(\al)\ot \pi_3^!(\be)).
\end{equation*}

To construct the morphism $\mu$ in \eq{od4eq9}, we need two things. Firstly, we need to prove a conjecture by the first author, partially stated in Amorim and Ben-Bassat \cite[\S 5.3]{AmBB}, which associates hypercohomology classes of perverse sheaves to oriented $-1$-shifted Lagrangians in oriented $-1$-shifted symplectic derived $\C$-stacks, and apply the conjecture to \eq{od4eq8}. And secondly, we need a choice of morphism $\la_\M$ in \eq{od4eq6} with~$\la_\M\ot\la_\M=\ka_\M$.

For $\mu$ to define an associative multiplication on $\bH^*(P_\bcM^\bu)$, this $\la_\M$ must satisfy an associativity condition. An argument similar to Proposition \ref{od4prop1} shows this is equivalent to \eq{od3eq27} commuting for $\xi_\M=\Xi^*(\la_\M)$. Hence $(K_\M^{1/2},\xi_\M)$ is a strong spin structure on $\M$ compatible with direct sums in the sense of Definition \ref{od3def3}. That is, $(K_\M^{1/2},\xi_\M)$ is {\it strong orientation data\/} on $\cM$, as in Definition \ref{od4def3}. Thus, strong orientation data is essential for extending the Cohomological Hall Algebra programme of \cite{KoSo3} to Calabi--Yau 3-folds. A positive answer to Question \ref{od3quest1} would allow us to construct such strong orientation data.
\end{rem}

\section{Proof of Theorem \ref{od3thm2}}
\label{od5}

We will deduce Theorem \ref{od3thm2} from Theorem \ref{od2thm1} in a somewhat formal way, following the outline of Cao--Gross--Joyce \cite[\S 3.4]{CGJ}. The remaining task is to set up a picture for spin structures over H-spaces resembling the theory for orientations in \cite{CGJ}. Our construction of algebraic spin structures will pass through spin structures on the topological realization $\oM{}^\top$ of the stack $\oM$, which we show are equivalent data. The passage from the algebraic geometry of coherent sheaves to the differential geometry of connections on vector bundles is best taken beginning with the stack $\T$ of algebraic vector bundles equipped with global generating sections, since these have both a natural Hermitian metric and holomorphic structure, so a unique compatible Chern connection. We then show that compatible spin structures on the H-spaces $\oM{}^\top$ and $\T{}^\top$ are equivalent, using that one is a homotopy-theoretic group completion of the other.\medskip

We first recall background material on H-spaces, see Hatcher~\cite[\S 3.C]{Hatc}.

\begin{dfn}
\label{od5def1}
An {\it H-space} is a triple $(X,e_X,\mu_X)$ where $X$ is a topological space, $e_X\in X$ is a base-point, and $\mu_X\colon X\t X \ra X$ is a continuous map such that $\mu_X(e_X,\cdot)\simeq\id_X$ and $\mu_X(\cdot,e_X)\simeq\id_X$. We always require the H-space multiplication to be associative and commutative up to homotopy.

An {\it H-map} $f:(X,e_X,\mu_X)\ra(Y,e_Y,\mu_Y)$ is a continuous map $f:X\ra Y$ such that $f(e_X)\simeq e_Y$ and~$f\ci\mu_X\simeq \mu_Y\ci(f\t f):X\t X\ra Y$.
\end{dfn}

\begin{dfn}
\label{od5def2}
Let $(X,e_X,\mu_X)$ be an H-space. Then the homology $H_*(X,\Z_2)$ comes equipped with the associative, graded commutative {\it Pontryagin product}
\begin{equation*}
H_p(X,\Z_2) \ot H_q(X,\Z_2) \stackrel{\bt}{\longra} H_{p+q}(X\t X,\Z_2) \xrightarrow{H_{p+q}(\mu_X)} H_{p+q}(X,\Z_2).
\end{equation*}
Write $\pi_0(X) \ra H_0(X,\Z_2)$, $\al\mapsto 1_\al$ for the natural inclusion of the basis vectors. Letting $[e_X]\in \pi_0(X)$ denote the component of base-point, $1_{[e_X]}$ is the unit in the {\it Pontryagin ring} $H_*(X,\Z_2)$. For more details, see Dold~\cite[\S VII.3]{Dold}.
\end{dfn}

As our construction of spin structures passes through several intermediate stages, we first extend the terminology of Definition \ref{od3def2} to general H-spaces.

\begin{dfn}
\label{od5def3}
{\bf (a)} Let $K \ra X$ be a complex line bundle.  A {\it square root} of $K$ is an isomorphism class $[J,\jmath]$ of a complex line bundle $J\ra X$ and an isomorphism $\jmath \colon J \ot J \ra K$, where an isomorphism from $(J,\jmath)$ to $(J',\jmath')$ is a bundle isomorphism $\io\colon J\ra J'$ with $\jmath=\jmath'\circ(\io\ot\io)$.\smallskip

\noindent
{\bf (b)} Let $(X,e_X,\mu_X)$ be an H-space. Assume $K=\De_X^*(L)$ for a complex line bundle $L\ra X\t X$, and that there are isomorphisms, as in \eqref{od3eq16}--\eqref{od3eq18},
\ea
\phi&:\pi_1^*(K)\ot\pi_2^*(K)\ot L^{\ot^2}\longra\mu_X^*(K),\label{od5eq1}\\
\chi&:\pi_{13}^*(L)\ot\pi_{23}^*(L)\longra(\mu_X\t\id_X)^*(L),\label{od5eq2}\\
\psi&:\pi_{12}^*(L)\ot\pi_{13}^*(L)\longra(\id_X\t\mu_X)^*(L).\label{od5eq3}
\ea
We suppose also that for some choice of homotopy $h: \mu_X\circ(\mu_X\t\id_X) \simeq \mu_X\circ(\id_X\t\mu_X)$ there exists an isotopy between the two ways round \eqref{od3eq19}, meaning there is a bundle isomorphism
\e
\label{od5eq4}
\ze\colon\pi_1^*(K)\!\ot\!\pi_2^*(K)\!\ot\!\pi_3^*(K)\!\ot\!\pi_{12}^*(L^{\ot^2})\!\ot\!\pi_{13}^*(L^{\ot^2})\!\ot\!\pi_{23}^*(L^{\ot^2})\!\longra\! h^*(K),
\e
restricting over $X\t X\t X\t\{0\}$ to $(\mu_X\t\id_X)^*(\phi)\circ(\pi_{12}^*(\phi)\ot \id_{\pi_3^*(K)}\ot\chi^{\ot^2})$ and over $X\t X\t X\t\{1\}$ to $(\id_X\t\mu_X)^*(\phi)\circ(\id_{\pi_1^*(K)}\ot \pi_{23}^*(\phi)\ot \psi^{\ot^2})$. Here the projections appearing in \eqref{od5eq4} all have the domain $X\t X\t X\t[0,1]$.

Given this setup, we call a square root $[J,\jmath]$ of $K$ {\it compatible} (with respect to $L,\phi,\mu_X$) if there exists an isomorphism
\begin{equation*}
\xi\colon \pi_1^*(J) \ot \pi_2^*(J) \ot L \longra \mu_X^*(J)
\end{equation*}
satisfying $\mu_X^*(\jmath)\circ(\xi\ot\xi)=\phi\circ\bigl(\pi_1^*(\jmath)\ot\pi_2^*(\jmath)\ot\id_{L^{\ot^2}}\bigr)$. We note that this terminology indeed depends only on the isomorphism class $[J,\jmath]$. Of course, the notion of compatible square root is independent of \eqref{od5eq2}--\eqref{od5eq4}, but these additional properties yield a better-behaved theory below. Write $S_K$ for the set of square roots of $K$, and $S_{K,L}^\phi \subset S_K$ for the subset of compatible square roots.
\end{dfn}

Given representatives $(J_1,\jmath_1), (J_2,\jmath_2)$ of square roots of $K$, the square of $J_1\ot_\C J_2^*$ has a given trivialization from $\jmath_1, \jmath_2$ and therefore is equipped with a real structure, so $J_1 = J_2 \ot_\R P$ for a real line bundle $P \ra X$. Up to isomorphism, $P$ depends only on the equivalence classes $[J_1,\jmath_1], [J_2,\jmath_2]$ and so on its Stiefel--Whitney class $w_1(P) \in H^1(X,\Z_2)$. Conversely, any square root may be tensored by isomorphism classes of real line bundles, elements of $H^1(X,\Z_2)$.

Similarly, isomorphism classes of complex line bundles correspond via the first Chern class to elements of $H^2(X,\Z)$. Finding a square root of $K$ is therefore equivalent to factoring $c_1(K) = 2\cdot x$ for $x \in H^2(X,\Z)$. Applying the Bockstein exact sequence $H^2(X,\Z) \xrightarrow{2\cdot} H^2(X,\Z) \ra H^2(X,\Z_2)$ then proves the following.

\begin{prop}\label{od5prop1}
Let\/ $K \ra X$ be a complex line bundle.
\begin{enumerate}
\item[{\bf (a)}]
The bundle\/ $K$ admits a square root if and only if the mod\/ $2$ reduction\/ $\mathrm{o}_K \in H^2(X,\Z_2)$ of its first Chern class\/ $c_1(K)$ vanishes.
\item[{\bf (b)}]
If\/ $K$ admits a square root, then the set\/ $S_K$ of square roots of\/ $K$ is a torsor over the group\/ $H^1(X,\Z_2)$.
\end{enumerate}
\end{prop}

\begin{rem}
\label{od5rem1}
The obstruction class $\mathrm{o}_K$ can be expressed geometrically as follows. Viewing $K$ as a principal $\C^*$-bundle, define $\de\colon K\t_X K \ra \C^*$ on the fibre product over $X$ by $\de(k_1,k_2)\cdot k_1=k_2$. Define $Q \ra K\t_X K$ as the pullback along $\de$ of the principal $\Z_2$-bundle $\C^* \ra \C^*$, $z\mapsto z^2$. Multiplication of complex numbers defines a map $m\colon Q \t_K Q \ra Q$, $m\bigl((k_1,k_2,z),(k_2,k_3,w)\bigr)=(k_1,k_3,zw)$. Up to isomorphism, square roots of $K$ correspond bijectively to pairs $(P,\eta)$ of a principal $\Z_2$-bundle $P \ra K$ and an isomorphism $\eta\colon P\ot_{\Z_2} P \ra Q$ satisfying the cocycle identity $m\bigl(\eta(p_1,p_2),\eta(p_2,p_3)\bigr)=\eta(p_1,p_3)$, see for example~\cite[Prop.~5.8]{BiUp}.
\end{rem}

For proving the analogue of Proposition \ref{od5prop1} for compatible square roots, we now introduce the {\it Hochschild--Pontryagin cohomology} $HP^{i,j}(X,k)$ of an H-space $X$, which is naturally bigraded. For compatible square roots the groups $HP^{1,1}(X,\Z_2)$ and $HP^{2,1}(X,\Z_2)$ play the role of $H^1(X,\Z_2)$ and $H^2(X,\Z_2)$.

\begin{dfn}
\label{od5def4}
Let $(X,e_X,\mu_X)$ be an H-space, $k$ a field, and write $R=H_*(X,k)$ for the Pontryagin $k$-algebra with its natural algebra homomorphism
\begin{equation*}
\ep\colon R=H_*(X,k) \xrightarrow{\enskip\pi\enskip}\bigoplus\nolimits_{\al \in \pi_0(X)} k \xrightarrow{\enskip\Si\enskip} k,
\end{equation*}
where the first map $\pi$ is the projection onto degree zero. Write $M$ for the ground field $k$ equipped with the $R$-bimodule structure given by multiplication by $\ep$. The {\it Hochschild--Pontryagin complex} of the H-space $X$ has cochain groups $C^n = \Hom_k(R^{\ot_k n}, M)$ and codifferential
\begin{equation*}
\d f(r_0, \ldots, r_n)\! =\! \ep(r_0)f(r_1, \ldots, r_n) \!+\! \sum_{i=1}^{n-1} f(\ldots, r_{i-1}r_i, \ldots) \!+\! f(r_0, \ldots, r_{n-1})\ep(r_n),
\end{equation*}
where $f\in C^n$ and $r_0, \ldots, r_n \in R$. The cohomology $HP^*(X,k)$ of this cochain complex is the {\it Hochschild--Pontryagin cohomology} of the H-space $X$. An H-map $f\colon X \ra Y$ induces an algebra homomorphism of the Pontryagin rings and therefore a map $f^*\colon HP^*(Y,k) \ra HP^*(X,k)$ in a functorial way. Regarding $M$ as a chain complex in degree $0$, $\ep$ has degree $0$ and we find that the Hochschild--Pontryagin cohomology groups $HP^i(X,k)$ are naturally bigraded $HP^i(X,k)=\bigoplus_{j\in \N} HP^{i,j}(X,k)$ by homomorphisms $f$ lowering degree by $j$.
\end{dfn}

For example, $HP^1(X,k)\! =\! \bigl\{f \!\in\! \Hom_k(R,M)\! :\! f(ab)\!=\!\ep(a)f(b) \!+\!f(a)\ep(b)\bigr\}$, and elements of $HP^2(X,k)$ are represented by $f\in \Hom_k(R\ot_k R, M)$ satisfying $\ep(a)f(b,c) - f(ab,c) + f(a,bc) - f(a,b)\ep(c) = 0$ for all $a,b,c \in R$, modulo coboundaries. Like the Pontryagin ring, the Hochschild--Pontryagin cohomology depends on the H-space structure up to homotopy, omitted from the notation.

Hochschild--Pontryagin cohomology is simply the usual Hochschild cohomology $HH^*(R,M)$ as in Weibel \cite[\S 9.1]{Weib} of the Pontryagin ring $R=H_*(X,k)$ with coefficients in $M=k$. In particular, it is a homological functor, independent of the choice of resolution of $R$, where we took the bar resolution. According to \cite[Lem.~9.1.3]{Weib}, we can therefore write Hochschild cohomology as
\e
\label{od5eq5}
HH^i(R,M) = \Ext^i_{R\ot_k R^\mathrm{op}}(R,M).
\e

\begin{prop}
\label{od5prop2}
Work in the situation of Definition\/~{\rm\ref{od5def3}(b),} in particular, with complex line bundles\/ $L \ra X\t X,$ $K=\De_X^*(L),$ and an isomorphism\/ $\phi\colon \pi_1^*(K)\ot\pi_2^*(K)\ot L^{\ot^2}\ra\mu_X^*(K)$. Assume that\/ $K$ admits a square root.
\begin{enumerate}
\item[{\bf (a)}]
There exists an obstruction class\/ $\mathrm{o}_{K,L}^\phi \in HP^{2,1}(X,\Z_2),$ defined in the proof, so that\/ $K$ admits a compatible square root if and only if\/~$\mathrm{o}_{K,L}^\phi=0$.
\item[{\bf (b)}]
If\/ $K$ admits a compatible square root, then the set of compatible square roots\/ $S_{K,L}^\phi$ of\/ $K$ is a torsor over the group\/~$HP^{1,1}(X,\Z_2)$.
\end{enumerate}
\end{prop}

\begin{proof}
{\bf (a)} Let $[J,\jmath]$ be a square root of $K$. Define a principal $\Z_2$-bundle $Q_J$ over $X\t X$ with fibres ${Q_J}|_{(x,y)} = \bigl\{ \xi_{x,y}\in \Hom_\C\big(J_x \ot J_y \ot L_{x,y}, J_{\mu_X(x,y)}\big) : \mu_X^*(\jmath)|_{(x,y)}\circ \xi^{\ot^2}_{x,y} = \phi\circ(\pi_1^*(\jmath)\ot\pi_2^*(\jmath)\ot\id_{L^{\ot^2}})|_{(x,y)} \bigr\}$. Since an isomorphism $\iota\colon (J,\jmath) \ra (J',\jmath')$ induces an isomorphism $Q_J \ra Q_{J'}$, the isomorphism class $f=w_1(Q_J) \in H^1(X\t X,\Z_2)$ depends only on $[J,\jmath]$. Moreover, $[J,\jmath]$ is compatible if and only if $Q_J$ is trivial, so when $f=0$.

We next verify that $f$ is a cocycle for a Hochschild--Pontryagin cohomology class. According to Definition \ref{od5def3}(b), there exists an isomorphism $\ze$ as in \eqref{od5eq4}, for an appropriate associativity homotopy $h\colon \mu_X\circ(\mu_X\t\id_X)\simeq \mu_X\circ(\id_X\t\mu_X)$. Using this, we can define a principal $\Z_2$-bundle $Q^\mathrm{ass}_J$ over $X\t X\t X\t[0,1]$ of those $\xi_{x,y,z,t} \in \Hom_\C\bigl( J_x\ot J_y\ot J_z\ot L_{x,y}\ot L_{x,z}\ot L_{y,z}, h_t^*(J)|_{(x,y,z)}\bigr)$ with square $h_t^*(\jmath)|_{(x,y,z)}\circ\xi_{x,y,z,t}^{\ot^2} = \ze\circ\bigl(\jmath|_x\ot\jmath|_y\ot\jmath|_z\ot\id_{L_{x,y}^{\ot^2}\ot L_{x,z}^{\ot^2}\ot L_{y,z}^{\ot^2}}\bigr)$ as fibre over $(x,y,z,t)$. There is an isomorphism $(\mu_X\t\id_X)^*(Q_J)\ot_{\Z_2} \pi_{12}^*(Q_J)\cong Q_J^\mathrm{ass}|_{X\t X\t X\t\{0\}}$, defined on elementary tensors $\xi_{(x,y),z}\ot \xi_{x,y}$ by composing $\xi_{(x,y),z}\colon J_{\mu(x,y)}\ot J_z\ot L_{\mu(x,y),z}\ra J_{\mu(\mu(x,y),z)}$, $\xi_{x,y}\colon J_x\ot J_y\ot L_{x,y}\ra J_{\mu(x,y)}$, and $\chi|_{(x,y,z)}\colon L_{x,z}\ot L_{y,z}\ra L_{\mu(x,y),z}$ from \eqref{od5eq2} to get an element of $Q_J^\mathrm{ass}|_{X\t X\t X\t\{0\}}$ with correct square. Similarly, $Q_J^\mathrm{ass}|_{X\t X\t X\t\{1\}} \cong (\id_X\t\mu_X)^*(Q_J)\ot_{\Z_2} \pi_{23}^*(Q_J)$ using $\psi$ from \eqref{od5eq3}. Hence $(\mu_X\t\id_X)^*(Q_J)\ot_{\Z_2} \pi_{12}^*(Q_J) \cong (\id_X\t\mu_X)^*(Q_J)\ot_{\Z_2} \pi_{23}^*(Q_J)$, non-canonically. Taking the first Stiefel--Whitney class of line bundles takes `$\ot_{\Z_2}$' to `$+$', so $\d f=0$ for $f=w_1(Q_J)$ in the notation of Definition~\ref{od5def4}.

The general square root of $K$ is obtained by tensoring $J$ with an arbitrary principal $\Z_2$-bundle $P \ra X$. This replaces $Q_J$ by $Q_J\ot\mu_X^*(P)\ot\pi_1^*(P^*)\ot\pi_2^*(P^*)$ and $f=w_1(Q_J)$ by a $\d$-coboundary. Therefore the image of $w_1(Q_J) \in H^1(X\t X,\Z_2) = \Hom(H_1(X,\Z_2)^{\ot2},\Z_2)$ in Hochschild--Pontryagin cohomology $HP^{2,1}(X,\Z_2)$, the {\it obstruction class} $\mathrm{o}_{K,L}^\phi$, is independent of $[J,\jmath]$.

\noindent{\bf (b)} Any two compatible square roots $[J_1,\jmath_1], [J_2,\jmath_2]$ of $K$ differ by an arbitrary isomorphism class $w_1(P)$ of a real line bundle with the additional property $\mu_X^*w_1(P)=\pi_1^*w_1(P) + \pi_2^*w_1(P)$, so by a class $w_1(P) \in HP^{1,1}(X,\Z_2) \subset \Hom_{\Z_2}(H_1(X),\Z_2)=H^1(X,\Z_2)$ in Hochschild--Pontryagin cohomology.
\end{proof}

\begin{rem}
\label{od5rem2}
Combining Propositions \ref{od5prop1} and \ref{od5prop2}, we see that $\mathrm{o}_K \in H^2(X,\Z_2)$ is the {\it primary obstruction} for the existence of a square root. Assuming the existence of $\phi\colon \pi_1^*(K)\ot\pi_2^*(K)\ot L^{\ot^2}\ra\mu_X^*(K)$, we find that the primary obstruction actually belongs to the Hochschild--Pontryagin subgroup $HP^{1,2}(X,\Z_2)$. This is because $\mu_X^*c_1(K) \equiv \pi_1^* c_1(K) + \pi_2^* c_1(K) \pmod 2$, as $c_1(L^{\ot^2})\equiv 0 \pmod 2$. When the primary obstruction vanishes, the {\it secondary obstruction} $\mathrm{o}_{K,L}^\phi \in HP^{2,1}(X,\Z_2)$ to the existence of a compatible square root is defined.
\end{rem}

The next definition comes from Caruso et al.~\cite[\S 1]{CCMT}.

\begin{dfn}
\label{od5def5}
A {\it homotopy-theoretic group completion\/} of an H-space $(X,\mu_X)$ is an H-map $f:(X,\mu_X)\ra(Y,\mu_Y)$ to a grouplike H-space $(Y,\mu_Y)$ such that:
\begin{itemize}
\setlength{\itemsep}{0pt}
\setlength{\parsep}{0pt}
\item[(i)] The map on connected components $\pi_0(f)\colon \pi_0(X) \ra \pi_0(Y)$ is the group completion of the abelian monoid $\pi_0(X)$.
\item[(ii)] The map $H_*(f)\colon H_*(X,k)\ra H_*(Y,k)$ is localization by the central multiplicative set $\pi_0(X)\subset H_*(X,k)$, for all fields $k$.
\end{itemize}
\end{dfn}

\begin{prop}
\label{od5prop3}
Let\/ $f:(X,\mu_X)\ra(Y,\mu_Y)$ be a homotopy-theoretic group completion of H-spaces. Then the induced map\/ $f^*\colon HP^{i,j}(Y,k) \ra HP^{i,j}(X,k)$ in Hochschild--Pontryagin cohomology is an isomorphism for all\/ $i,j \in \N$.
\end{prop}

\begin{proof}
Write $R=H_*(X,\Z_2)$ for the Pontryagin ring, $M$ for the $R$-bimodule $\Z_2$ via $\ep\colon R \ra \Z_2$, and $S=\pi_0(X) \subset R$. By assumption, the central localization $S^{-1}R$ can be identified with Pontryagin ring $H_*(Y,\Z_2)$ and the localized $S^{-1}R$-bimodule on $S^{-1}M$ with $\Z_2$ with the action $\ep\colon H_*(Y,\Z_2) \ra \Z_2$. We claim in greater generality that for any $k$-algebra $R$, $R$-bimodule $M$, and central $S\subset R$ acting trivially on $M$, that the localization map $\la\colon R \ra S^{-1}R$ induces an isomorphism $\la^*\colon HH^*(S^{-1}R,S^{-1}M) \ra HH^*(R,M)$.

The proof is an adaption of \cite[Lem.~3.3.8, Prop.~3.3.10]{Weib} to our present hypotheses. Observe that for any $k$-algebra $\tilde{R}$, left $\tilde{R}$-modules $A,B$, and central multiplicative set $T \subset \tilde{R}$ acting trivially on $B$, the localization $\tilde\la\colon \tilde{R} \ra T^{-1}\tilde{R}$ induces an isomorphism $\tilde{\la}^*\colon \Hom_{T^{-1}\tilde{R}}(T^{-1}A,T^{-1}B) \cong \Hom_{\tilde{R}}(A,B)$. Under the same hypotheses, this implies $\la^*\colon\Ext^i_{T^{-1}\tilde{R}}(T^{-1}A,T^{-1}B)\cong\Ext^i_{\tilde{R}}(A,B)$, since a free $\tilde{R}$-module resolution $F_* \ra A \ra 0$ (by possibly infinitely generated $\tilde{R}$-modules) induces a free $T^{-1}\tilde{R}$-module resolution $T^{-1}F_* \ra T^{-1}A \ra 0$, as module localization is an exact functor preserving (possibly infinite) direct sums. Apply the previous observation to identify the cochain complexes $\Hom_{T^{-1}\tilde{R}}(T^{-1}F_*,T^{-1}B) \cong \Hom_{\tilde{R}}(F_*,B)$. Finally, we apply this to $\tilde{R}=R\ot_k R^\mathrm{op}$ and $T=S\t S^\mathrm{op}$ to obtain the result on Hochschild cohomology:
\begin{align*}
 HH^*&(S^{-1}R, S^{-1}M) \stackrel{\mathclap{\eqref{od5eq5}}}{\;=\;} \Ext^*_{S^{-1}R \ot_k (S^{-1}R)^\mathrm{op}}(S^{-1}R, S^{-1}M)\\
 &\cong\Ext^*_{T^{-1}(R\ot_k R^\mathrm{op})}(T^{-1}R, T^{-1}M)\cong \Ext^*_{R\ot_k R^\mathrm{op}}(R,M) \stackrel{\mathclap{\eqref{od5eq5}}}{\;=\;} HH^*(R,M).
\end{align*}
Here we identify the $k$-algebras $T^{-1}(R\ot_k R^\mathrm{op}) = S^{-1}R \ot_k (S^{-1}R)^\mathrm{op}$ and correspondingly the $S^{-1}R$-bimodules $S^{-1}R$ and $S^{-1}M$ with the left $T^{-1}(R\ot_k R^\mathrm{op})$-modules $T^{-1}R$ and $T^{-1}M$.
\end{proof}

\begin{prop}
\label{od5prop4}
{\bf (a)} Let\/ $K_X \ra X, K_Y \ra Y$ be line bundles,\/ $f\colon X \ra Y$ a continuous map, and\/ $\ka \colon f^*K_Y \ra K_X$ an isomorphism. Suppose\/ $K_Y$ has a square root, so\/ $S_{K_Y}\neq \es$. Then pullback defines a map
\e
\label{od5eq6}
 S_{K_Y} \longra S_{K_X},\qquad
[J,\jmath] \longmapsto [f^*J, \ka\circ f^*\jmath],
\e
which is equivariant over\/ $f^*\colon H^1(Y;\Z_2)\to H^1(X;\Z_2).$

\noindent{\bf (b)} Let\/ $(X,e_X,\mu_X), (Y,e_Y,\mu_Y)$ be H-spaces, $L_X \ra X\t X, L_Y \ra Y\t Y$ line bundles, $K_X=\De_X^*(L_X),$ $K_Y=\De_X^*(L_Y),$ and\/ $\phi_X\colon \pi_1^*(K_X)\ot \pi_2^*(K_X) \ot L_X^{\ot^2} \ra \mu_X^*(K_X)$ and\/ $\phi_Y\colon \pi_1^*(K_Y)\ot \pi_2^*(K_Y) \ot L_Y^{\ot^2} \ra \mu_Y^*(K_Y)$ isomorphisms {\rm(}at this point, it is unnecessary to require the entire structure as in Definition\/~{\rm\ref{od5def3}(b))}.

Assume also an H-map\/ $f\colon X \ra Y,$ an isomorphism\/ $\la\colon (f\t f)^*(L_Y) \ra L_X,$ homotopy\/ $h\colon \mu_Y\circ(f\t f) \simeq f\circ \mu_X,$ and isomorphism\/ $\th\colon \bigl(\pi_1^*f^*(K_Y) \ot \pi_2^*f^*(K_Y) \ot (f\t f)^*(L_Y^{\ot^2})\bigr) \t [0,1] \ra h^*(K_Y)$ over\/ $X\t X \t [0,1]$ interpolating between\/ $(f\t f)^*(\phi_Y)$ and\/ $\mu_X^*(\ka^{-1})\circ \phi_X\circ (\pi_1^*(\ka)\ot\pi_2^*(\ka)\ot\la^{\ot^2})$ at the endpoints, where we set\/ $\ka=\De_X^*(\la)\colon f^*(K_Y) \ra K_X$. If\/ $S_{K_Y,L_Y}^{\phi_Y}\neq \es,$ then the map \eqref{od5eq6} induced by\/ $(f,\ka)$ restricts to a map
\e
\label{od5eq7}
 S_{K_Y,L_Y}^{\phi_Y} \longra S_{K_X,L_X}^{\phi_X},
\e
which is equivariant over\/ $f^*\colon HP^{1,1}(Y,\Z_2)\to HP^{1,1}(X,\Z_2).$

\noindent{\bf (c)} Continuing {\bf(b)}, suppose\/ $K_X,L_X$ and\/ $K_Y,L_Y$ both admit the entire structure of Definition\/~{\rm\ref{od5def3}(b)}. Assume\/ $\mathrm{o}_{K_Y}=0,$ so that also\/ $\mathrm{o}_{K_X}=0$ and we have well-defined secondary obstruction classes\/ $\mathrm{o}_{K_X,L_X}^{\phi_X}\in HP^{2,1}(X,\Z_2)$ and\/ $\mathrm{o}_{K_Y,L_Y}^{\phi_Y}\in HP^{2,1}(Y,\Z_2)$. Then\/ $f^*\colon HP^{2,1}(Y,\Z_2)\to HP^{2,1}(X,\Z_2)$ maps the obstruction class\/ $\mathrm{o}_{K_Y,L_Y}^{\phi_Y}$ to\/~$\mathrm{o}_{K_X,L_X}^{\phi_X}$.
\smallskip

\noindent
{\bf (d)} Continuing {\bf(b)}, suppose\/ $K_X,L_X$ and\/ $K_Y,L_Y$ both admit the entire structure of Definition\/~{\rm\ref{od5def3}(b)}. Assume\/ $f$ is a homotopy-theoretic group completion. Then\/ $S_{K_X,L_X}^{\phi_X}\neq \es \iff S_{K_Y,L_Y}^{\phi_Y}\neq \es,$ and \eqref{od5eq7} becomes a bijection.
\end{prop}

\begin{proof}
Part {\bf (a)} is obvious. For {\bf (b)}, let $[J,\jmath] \in S_{K_Y,L_Y}^{\phi_Y}$. We can use $\th$ to define a principal $\Z_2$-bundle $P$ over $X\t X\t[0,1]$ whose fibre over $(x,y,t)$ consists of those $\xi_{x,y,t}\in\Hom_\C\bigl(f^*(J_Y)|_x\ot f^*(J_Y)|_y\ot (f\t f)^*(L_Y)|_{(x,y)},h_t^*(J_Y)|_{(x,y)}\bigr)$ squaring to $\th|_{X\t X\t\{t\}}$. Since $[J,\jmath]$ is compatible, there exists an isomorphism $\pi_1^*(J_Y)\ot \pi_2^*(J_Y) \ot L_Y \ra \mu_Y^*(J_Y)$ squaring to $\phi_Y$ whose pullback along $f\t f$ determines a global section of $P|_{X\t X\t \{0\}}$. The fibre transport in $P$ of this section along the interval yields a global section of $P|_{X\t X\t \{1\}}$, explicitly, an isomorphism $\pi_1^*f^*(J_Y)\ot \pi_2^*f^*(J_Y) \ot (f\t f)^*(L_Y) \ra \mu_X^*f^*(J_Y)$ squaring to $\psi|_{X\t X\t\{1\}}$, which we precompose with $\id_{\pi_1^*f^*(J_Y)\ot \pi_2^*f^*(J_Y)}\ot\la^{-1}$ to verify that $[f^*J, \ka\circ f^*\jmath]$ is compatible.\smallskip
 
\noindent{\bf (c)} We note that the full structure of Definition {\rm\ref{od5def3}(b)} is only required to show that the defining representatives for $\mathrm{o}_{K_Y,L_Y}^{\phi_Y}, \mathrm{o}_{K_X,L_X}^{\phi_X}$ are Hochschild--Pontryagin cocycles, in particular, for the application of Proposition~{\rm\ref{od5prop3}}.
 
Fix a square root $J_Y$ of $K_Y$. Recall from the proof of Proposition \ref{od5prop2}(a) that $\mathrm{o}_{K_Y,L_Y}^{\phi_Y}$ is represented by the first Stiefel--Whitney class of the principal $\Z_2$-bundle $Q_{J_Y}\ra Y\t Y$ with fibre over $(x,y)$ those $\xi_{x,y}\in \Hom_\C\big(J_Y|_x \ot J_Y|_y \ot L_Y|_{(x,y)}, J_Y|_{\mu_Y(x,y)}\big)$ squaring to $\phi_Y|_{(x,y)}$. Set $J_X=f^*(J_Y)$, as in (a). Then $\mathrm{o}_{K_X,L_X}^{\phi_X}$ is represented by the class of $Q_{J_X}\ra X\t X$ with fibre over $(x,y)$ those $\Hom_\C\big(f^*(J_Y)|_x \ot f^*(J_Y)|_y \ot L_X|_{(x,y)}, f^*(J_Y)|_{\mu_X(x,y)}\big)$ squaring to $\phi_X|_{(x,y)}$. We can form another principal $\Z_2$-bundle $\tilde{Q}\ra X\t X\t [0,1]$ with fibre over $(x,y,t)$ the elements $f^*(J_Y)|_x\ot f^*(J_Y)|_y \ot (f\t f)^*(L_Y)|_{(x,y)} \ra h_t^*(J_Y)|_{(x,y)}$ squaring to $\th$. Then $\tilde{Q}|_{X\t X\t\{0\}} \cong (f\t f)^*(Q_{J_Y})$ and $\tilde{Q}|_{X\t X\t\{1\}} \cong Q_{J_X}$. Hence $(f\t f)^*w_1(Q_{J_Y})=w_1(Q_{J_X})$, as required.\smallskip
 
\noindent{\bf (d)} When $f$ is a homotopy-theoretic group completion, \eqref{od5eq7} is a map of torsors, unless $S_{K_Y,L_Y}^{\phi_Y}=\es$, equivariant over the homomorphism $f^*\colon HP^{1,1}(Y,\Z_2)\to HP^{1,1}(X,\Z_2),$ which is bijective by Proposition \ref{od5prop3}. It remains to prove $S_{K_X,L_X}^{\phi_X}\neq\es \implies S_{K_Y,L_Y}^{\phi_Y}\neq\es$. As in Remark \ref{od5rem2}, we have $\mathrm{o}_{K_X} \in HP^{1,2}(X,\Z_2)$ and $\mathrm{o}_{K_Y} \in HP^{1,2}(Y,\Z_2)$. By assumption, $S_{K_X,L_X}^{\phi_X}\neq\es$, so both obstructions $\mathrm{o}_{K_X}=0$ and $\mathrm{o}_{K_X,L_X}^{\phi_X}=0$ vanish. Since $\ka\colon f^*K_Y\cong K_X$, the isomorphism $f^*\colon HP^{1,2}(Y,\Z_2)\to HP^{1,2}(X,\Z_2)$ of Proposition \ref{od5prop3} maps $\mathrm{o}_{K_Y} \mapsto \mathrm{o}_{K_X}$, so we conclude $\mathrm{o}_{K_Y}=0$. Then the secondary obstruction class $\mathrm{o}_{K_Y,L_Y}^{\phi_Y} \in HP^{2,1}(Y,\Z_2)$ is defined, and by (c) gets mapped to $\mathrm{o}_{K_X,L_X}^{\phi_X}=0$ under the isomorphism $f^*\colon HP^{2,1}(Y,\Z_2)\to HP^{2,1}(X,\Z_2).$ We therefore have $\mathrm{o}_{K_Y,L_Y}^{\phi_Y}=0$ which implies that $S_{K_Y,L_Y}^{\phi_Y}\neq\es$ by Proposition~\ref{od5prop2}(a).
\end{proof}

We need one last preparation before we prove Theorem \ref{od3thm2}.

\begin{prop}
\label{od5prop5}
Algebraic spin structures on\/ $\oM$ as in Definition\/~{\rm\ref{od3def2}} correspond bijectively to square roots of\/ $K_\oM^\top \ra \oM{}^\top$ as a topological complex line bundle. This bijection identifies the subsets of compatible spin structures.
\end{prop}

\begin{proof}
For any $S \in \HSta_\C$ as in Remark \ref{od3rem1}(f), write $\cP_S$ for the Picard groupoid of algebraic principal $\Z_2$-bundles over $S$. Topological realization defines a monoidal equivalence to topological principal $\Z_2$-bundles over $S^\top$,
\ea
(\;)^\top &\colon \cP_S \stackrel{\simeq}{\longra} \cP_{S^\top},
\label{od5eq8}\\
\pi_0(\cP_S)&=\Hom_{\Ho(\HSta_\C)}(S,[*/\Z_2]) \cong \Hom_{\Ho(\Top)}(S^\top,B\Z_2)=\pi_0(\cP_{S^\top}),
\nonumber\\
\pi_1(\cP_S)&=\Hom_{\Ho(\HSta_\C)}(S,\Z_2) \cong \Hom_{\Ho(\Top)}(S^\top,\Z_2)=\pi_1(\cP_{S^\top}),
\nonumber
\ea
viewing $\Z_2$ as a higher $\C$-stack as in Remark \ref{od3rem1} and for the quotient stack $[*/\Z_2]$. For any affine $\C$-scheme $T$ the \'etale cohomology $H^*(T,\underline\Z_2)$ agrees with singular cohomology $H^*(T^\top,\Z_2)$, by Artin's comparison theorem \cite[Exp.~XI, Th\'eor\`eme 4.4]{SGA4}. Recall the Quillen adjunction from Blanc~\cite[Prop.~3.2(i)]{Blan}
\begin{equation*}
(\;)^\top : \HSta_\C \longleftrightarrows \Top : R,
\end{equation*}
where $R(K)_{\Spec A}=\Map\bigl((\Spec A)^\top, K\bigr)$. By the fact just discussed, the morphisms $(\;)^\top\colon \Z_2 \ra R(\Z_2)$ and $(\;)^\top\colon [*/\Z_2] \ra R(B\Z_2)$ are equivalences of stacks, so the maps \eqref{od5eq8} are indeed bijections by adjunction.

The key point now is that $(\;)^\top$ takes (compatible) spin structures to (compatible) square roots, and that in each case both the choices and the obstructions can be expressed entirely within the category of principal $\Z_2$-bundles.

In detail, write $S_{K_\oM}$ for the set of algebraic spin structures on $\oM$ and $S_{K_\oM^\top}$ for the set of square roots of $K_\oM^\top$. Topological realization determines a map $S_{K_\oM}\ra S_{K_\oM^\top}$, equivariant over the isomorphism $\pi_0(\cP_\oM)\ra\pi_0(\cP_{\oM{}^\top})$, which is therefore a bijection unless $S_{K_\oM}=\es$, as both sets are torsors under $\pi_0(\cP_\oM), \pi_0(\cP_{\oM{}^\top})$.

As to the case $S_{K_\oM}=\es$, Remark \ref{od5rem1} expresses the obstruction to the existence of a square root of $K_\oM$ in terms of the existence of a principal $\Z_2$-bundle $P \ra Q$ with an isomorphism $P\ot_{\Z_2} P \ra Q$ in the category $\cP_{K_\oM\t_\oM K_\oM}$ satisfying an identity in $\cP_{K_\oM\t_\oM K_\oM\t_\oM K_\oM}$, where $Q$ denotes the pullback of the principal $\Z_2$-bundle $\C^* \ra \C^*$, $z\mapsto z^2$, along $\de\colon K_\oM \t_\oM K_\oM \ra \C^*$. Applying \eqref{od5eq8} to $S=Q, K_\oM\t_\oM K_\oM, K_\oM\t_\oM K_\oM\t_\oM K_\oM$, this formulation makes it clear that $S_{K_\oM}=\es \iff S_{K_\oM^\top}=\es$, so we have a bijection. 

From an alternative point of view, the algebraic obstruction class $\mathrm{o}_{K_\oM} \equiv c_1(K_\oM) \pmod{2} $ in $H^2(\oM, \underline{\Z}_2)$ is mapped onto the topological obstruction class $\mathrm{o}_{K_{\oM{}^\top}} \equiv c_1(K_{\oM{}^\top}) \pmod{2}$ in $H^2(\oM{}^\top,\Z_2)$ by topological realization, which by Artin comparison is an isomorphism~$H^2(\oM,\underline{\Z}_2) \cong H^2(\oM{}^\top,\Z_2)$.

Similarly, in both the algebraic and topological categories, any two compatible square roots differ by an element of the group $\cG^\mathrm{alg}$ or $\cG^\top$ of isomorphism classes of algebraic or topological principal $\Z_2$-bundles $[P]$ with the additional property that they admit an isomorphism $\pi_1^*(P)\ot\pi_2^*(P) \cong \Phi^*(P)$, and $(\;)^\top\colon\cG^\mathrm{alg} \ra \cG^\top$ is an isomorphism by \eqref{od5eq8} applied to $S=\oM, \oM\t\oM$. As before, topological realization defines a map of torsors
\begin{equation*}
(\;)^\top\colon S_{K_\oM, L_\oM}^{\phi_\oM} \longra S_{K_\oM^\top, L_\oM^\top}^{\phi_\oM^\top},
\end{equation*}
which is therefore a bijection, unless~$S_{K_\oM, L_\oM}^{\phi_\oM}=\es$. 

Artin comparison gives an isomorphism $H_*(\oM,\underline{\Z}_2) \cong H_*(\oM{}^\top,\Z_2)$ of the Pontryagin rings and therefore of the corresponding Hochschild--Pontryagin cohomologies, and $\cG^\mathrm{alg}=HP^{1,1}(\oM,\underline{\Z}_2)$, $\cG^\top=HP^{1,1}(\oM{}^\top,\Z_2)$. The argument for obstructions to compatible spin structures is similar to the previous case. For the secondary obstruction classes, $\mathrm{o}_{K_\oM,L_\oM}^{\phi_\oM} \in H^{2,1}(\oM,\underline{\Z}_2)$ gets mapped to $\mathrm{o}_{K_{\oM{}^\top},L_{\oM{}^\top}}^{\phi_{\oM{}^\top}} \in H^{2,1}(\oM{}^\top,\Z_2)$ under the topological realization isomorphism, as is clear from the construction in Proposition~\ref{od5prop2}(a).
\end{proof}

We now prove Theorem \ref{od3thm2}, following the general strategy of Cao--Gross--Joyce \cite[\S 3.4]{CGJ} whose basic setup we now briefly recall. Let $X$ be a spin smooth projective $3$-fold. Use the notation $\M,\oM,K_\M,\ldots$ of Definitions \ref{od3def1} and \ref{od3def2}. As in \cite[\S 3.4]{CGJ} we have the $\C$-stack $\T$ of algebraic vector bundles $F\ra X$ generated by global sections $(s_1, s_2, \ldots)$ with $s_i=0$  for $i\gg 0$. Taking the direct sum of vector bundles defines a morphism $\Xi\colon \T\t \T \ra \T$, whose topological realization $\Xi^\top$ determines an H-space structure on $\T{}^\top$, see~\cite[Lem.~3.20]{CGJ}. 

As in \cite[Def.~3.18]{CGJ} there is a canonical morphism $\De\colon \T \ra \oM$ mapping the $\C$-point $[F,(s_1,s_2,\ldots)]$ to $[F]$, whose topological realization $\De^\top\colon \T{}^\top \ra \oM{}^\top$ is a homotopy-theoretic group completion by \cite[Prop.~3.22]{CGJ}. Here $\oM{}^\top$ gets the H-space structure $\bar\Phi^\top$ of Remark \ref{od3rem1}(h). As in \cite[Def.~3.23]{CGJ}, the comparison between algebraic and differential geometry is given by a morphism
\begin{equation*}
\La\colon \T{}^\top \longra \coprod\nolimits_{\;
\begin{subarray}{l}\text{iso. classes $[P]$ of principal}\\ \text{$\U(n)$-bundles $P \ra X$, $n\ge 0$}\end{subarray}} \B_P^\cla
\end{equation*}
mapping the $\C$-point $[F,(s_1,s_2,\ldots)]$ to the unique Chern connection compatible both with the holomorphic structure on the vector bundle $F\ra X$ and the Hermitian metric determined by the splitting $\C^N = F \op \Ker(s_1,s_2,\ldots, s_N)$ for some $N\gg 0$ with $s_i=0$ for all $i>N$. According to \cite[(3.39)]{CGJ}, $\La$ is a morphism of H-spaces, where we use $\Phi^\top$ from Definition \ref{od2def2} to give the disjoint union $\B^\cla \coloneqq \coprod\nolimits_{[P]} \B_P^\cla$ an H-space structure.

Recall from Definition \ref{od3def1} and Remark \ref{od3rem2}(a) the canonical line bundles $K_\oM \ra \oM$, $L_\oM \ra \oM\t\oM$ and $\phi_\oM\colon \pi_1^*(K_\oM)\ot\pi_2^*(K_\oM)\ot L_\oM^{\ot^2}\longra\Phi^*(K_\oM)$. Explicitly, at $\C$-points $[F], [F_1], [F_2]$ of $\M \subset \oM$ we have
\e
\label{od5eq9}
\begin{aligned}
K_\oM\vert_{[F]}&\cong \bigot\nolimits_{i=0}^m\bigl(\La^{\rm top}_\C\Ext^i(F,F\ot J)\bigr)^{(-1)^i},\\
L_\oM\vert_{([F_1],[F_2])}&\cong \bigot\nolimits_{i=0}^m\bigl(\La^{\rm top}_\C\Ext^i(F_1,F_2\ot J)\bigr)^{(-1)^i}.
\end{aligned}
\e
These bundles pull back to $\T$ as $K_\T = \De^*K_\oM$, $L_\T = \De^*L_\oM$ and also we have $\phi_\T=\De^*\phi_\oM\colon \pi_1^*(K_\T)\ot\pi_2^*(K_\T)\ot L_\T \ra \Xi^*K_\T$, using $\bar\Phi\circ(\De\t\De)=\De\circ\Xi$.

On the differential geometric side, the disjoint union of the topological line bundles $K^{E_\bu}_P \ra \B_P$ and $L^{E_\bu}_{P_1,P_2} \ra \B_{P_1} \t \B_{P_2}$ from Definitions \ref{od2def4} and \ref{od2def5} determine line bundles $K^{E_\bu} \ra \B$ and $L^{E_\bu} \ra \B \t \B$. From \eqref{od2eq11} we recall the isomorphism $\phi^{E_\bu}\colon \pi_1^*(K^{E_\bu})\ot \pi_2^*(K^{E_\bu}) \ot \bigl(L^{E_\bu}\bigr)^{\ot^2} \ra \Phi^*(K^{E_\bu})$. Explicitly,
\e
\label{od5eq10}
\begin{aligned}
K^{E_\bu}\big\vert_{[\nabla_P]} &= \det_\C\bigl( \bar\partial_{\nabla_{\Ad P}} + \bar\partial_{\nabla_{\Ad P}}^*\bigr),\\
L^{E_\bu}\big\vert_{[\nabla_{P_1},\nabla_{P_2}]} &= \det_\C\bigl( \bar\partial_{\nabla_{P_1^*\ot P_2}} + \bar\partial_{\nabla_{P_1^*\ot P_2}}^*\bigr),
\end{aligned}
\e
acting on $(0,p)$-forms, $p$ odd, with values in $\End_\C(P)$ and $\Hom_\C(P_1,P_2)$, respectively. As in \cite[(3.57)]{CGJ}, for coefficients in a holomorphic vector bundle the Dolbeault resolution \eqref{od5eq10} computes the Ext-groups \eqref{od5eq9}. We obtain in this way isomorphisms $\La^*(K^{E_\bu}) \cong K_\T^\top$ and $\La^*(L^{E_\bu}) \cong L_\T^\top$, compatible with taking direct sums, so Proposition \ref{od5prop4}(b) yields a map
\e
\label{od5eq11}
S_{K^{E_\bu}, L^{E_\bu}}^{\phi^{E_\bu}}
\longra
S_{K_\T^\top, L_\T^\top}^{\phi_\T^\top}.
\e
As $\De^\top$ is a homotopy-theoretic group completion, and the bundles on $\T$ are just the pullbacks of the bundles on $\oM$, Proposition \ref{od5prop4}(d) gives a bijection
\e
\label{od5eq12}
S_{K_\oM^\top, L_\oM^\top}^{\phi_\oM^\top}\longra S_{K_\T^\top, L_\T^\top}^{\phi_\T^\top}.
\e
By Theorem \ref{od2thm1} we have a canonical spin structure in $S_{K^{E_\bu}, L^{E_\bu}}^{\phi^{E_\bu}}$, whose image under \eqref{od5eq11} has a unique pre-image under \eqref{od5eq12}. The resulting compatible topological spin structure for $K_\oM^\top$ corresponds by Proposition \ref{od5prop5} uniquely to a compatible algebraic spin structure.

\section{Proof of Theorem \ref{od3thm3}}
\label{od6}

We recall the definition of group cohomology, following Brown \cite[p.~59]{Brow}:

\begin{dfn}
\label{od6def1}
Let $G$ be a group, and $R$ a commutative ring. Define a complex $\bigl(C^*(G,R),\d\bigr)$ of $R$-modules by $C^n(G,R)=\Map(G^n,R)$, the set of all maps of sets $f:G^n\ra R$, with $C^0(G,R)=R$, and $\d:C^n(G,R)\ra C^{n+1}(G,R)$ by
\ea
\d f(g_1,\ldots,g_{n+1})=f(g_2,&\ldots,g_{n+1})+\sum_{i=1}^nf(g_1,\ldots,g_{i-1},g_ig_{i+1},g_{i+2},\ldots,g_{n+1})
\nonumber\\
&+(-1)^{n+1}f(g_1,\ldots,g_n),
\label{od6eq1}
\ea
for all $f:G^n\ra R$ and $g_1,\ldots,g_{n+1}\in G$. Then $\d\ci\d=0$. Define an $R$-module
\begin{equation*}
H^n(G,R)=\frac{\Ker\bigl(\d:C^n(G,R)\ra C^{n+1}(G,R)\bigr)}{\Im\bigl(\d:C^{n-1}(G,R)\ra C^n(G,R)\bigr)},\qquad n=0,1,\ldots.
\end{equation*}
(More generally one can define $H^*(G,M)$ for a $G$-module $M$ over $R$. We have taken the $G$ action on $R$ to be trivial in the definition of $\d$ in~\eq{od6eq1}.)
\end{dfn}

Work in the situation of Theorem \ref{od3thm3}, with $K_0^\semi(X)$ as in Definition \ref{od3def4}. Then $[K_\oM^{1/2}]$ is a spin structure on $\oM$ compatible with direct sums. Choose a representative $K_\oM^{1/2}$ for $[K_\oM^{1/2}]$ with isomorphism $\jmath:(K_\oM^{1/2})\ot(K_\oM^{1/2})\ra K_\oM$. By Definition \ref{od3def2} there exists an isomorphism $\check\xi_\oM$ in \eq{od3eq23} with~$\check\xi_\oM\ot\check\xi_\oM=\phi_\oM$.

Suppose that $(\dot K_\oM^{1/2},\dot \xi_\oM)$ is a strong spin structure compatible with direct sums extending $[K_\oM^{1/2}]$, as in Definition \ref{od3def3}, with isomorphism $j:(\dot K_\oM^{1/2})\ot(\dot K_\oM^{1/2})\ra K_\oM$. Then $(K_\oM^{1/2},\jmath),(\dot K_\oM^{1/2},j)$ lie in the same isomorphism class of square roots, so there exists an isomorphism $\io:K_\oM^{1/2}\ra\dot K_\oM^{1/2}$ with $\jmath=j\ci(\io\ot\io)$. Define a morphism $\xi_\oM$ as in \eq{od3eq23} by $\xi_\oM=\dot\xi_\oM\ci(\pi_1^*(\io)\ot\pi_2^*(\io)\ot\id_{L_\oM})$. Then $(K_\oM^{1/2},\xi_\oM)\cong(\dot K_\oM^{1/2},\dot\xi_\oM)$, so $(K_\oM^{1/2},\xi_\oM)$ is a strong spin structure compatible with direct sums extending~$[K_\oM^{1/2}]$. 

This shows that we may take $\dot K_\oM^{1/2}$ to be the fixed square root $K_\oM^{1/2}$. Also, $\xi_\oM$ and $\check\xi_\oM$ are both isomorphisms in \eq{od3eq23} satisfying $\xi_\oM\ot\xi_\oM=\check\xi_\oM\ot\check\xi_\oM=\phi_\oM$. Hence $\xi_\oM=\pm \check\xi_\oM$ locally on $\oM\t\oM$. Therefore $\xi_\oM=E\cdot\check\xi_\oM$ for some locally constant function $E:\oM\t\oM\ra\{\pm 1\}$. So $E$ takes a constant value $\pm 1$ on each connected component of $\oM\t\oM$. Let us write $\Z_2=\Z/2\Z=\{\ul 0,\ul 1\}$, and define a bijection $\Z_2\ra\{\pm 1\}$, written $z\mapsto (-1)^z$, by $(-1)^{\ul 0}=1$ and~$(-1)^{\ul 1}=-1$. 

The connected components of $\oM$ are $\al\in K_0^\semi(X)$ as in Definition \ref{od3def4}, so the connected components of $\oM$ are $\al\t\be$ for $\al,\be\in K_0^\semi(X)$. Therefore $E:\oM\t\oM\ra\{\pm 1\}$ is uniquely determined by a map $\ep:K_0^\semi(X)^2\ra\Z_2$, such that $E\vert_{\al\t\be}=(-1)^{\ep(\al,\be)}$ for all $\al,\be\in K_0^\semi(X)$. We have proved:

\begin{lem}
\label{od6lem1}
Fix\/ $K_\oM^{1/2},\jmath,\check\xi_\oM$ as above. Then any strong spin structure on\/ $\oM$ compatible with direct sums extending\/ $[K_\oM^{1/2}]$ is isomorphic to one of the form\/ $(K_\oM^{1/2},\xi_\oM),$ where\/ $\xi_\oM$ in \eq{od3eq23} is defined by\/ $\xi_\oM\vert_{\al\t\be}=(-1)^{\ep(\al,\be)}\cdot\check\xi_\oM \vert_{\al\t\be}$ for all\/ $\al,\be$ in\/ $K_0^\semi(X),$ for\/ $\ep:K_0^\semi(X)^2\ra\Z_2,$ that is,\/~$\ep\in C^2(K_0^\semi(X),\Z_2)$.
\end{lem}

Next we calculate the condition on $\ep\in C^2(K_0^\semi(X),\Z_2)$ for $(K_\oM^{1/2},\xi_\oM)$ to be a strong spin structure on $\oM$ compatible with direct sums. Consider equation \eq{od3eq27} on $\oM\t\oM\t\oM$ with $\check\xi_\oM$ in place of $\xi_\oM$. Although \eq{od3eq27} for $\check\xi_\oM$ need not commute, \eq{od3eq19} does commute, and \eq{od3eq19} is the `square' of \eq{od3eq27} as $\check\xi_\oM\ot\check\xi_\oM=\phi_\oM$. Hence \eq{od3eq27} for $\check\xi_\oM$ commutes up to sign. 

That is, there is a unique locally constant function $Z:\oM\t\oM\t\oM\ra\{\pm 1\}$ such that the two routes around \eq{od3eq27} for $\check\xi_\oM$ differ by multiplication by $Z$. Define a unique function $\ze:K_0^\semi(X)^3\ra\Z_2$ such that $Z\vert_{\al\t\be\t\ga}=(-1)^{\ze(\al,\be,\ga)}$ for all $\al,\be,\ga\in K_0^\semi(X)$.

Now compare \eq{od3eq27} for $\xi_\oM$ and $\check\xi_\oM$, and restrict to the connected component $\al\t\be\t\ga$ in $\oM\t\oM\t\oM$. Then:
\begin{itemize}
\setlength{\itemsep}{0pt}
\setlength{\parsep}{0pt}
\item[(i)] \eq{od3eq27} for $\check\xi_\oM$ fails to commute by a factor $(-1)^{\ze(\al,\be,\ga)}$ on $\al\t\be\t\ga$.
\item[(ii)] In the top row, $\pi_{12}^*(\xi_\oM)$ and $\pi_{12}^*(\check\xi_\oM)$ differ by a factor $(-1)^{\ep(\al,\be)}$.
\item[(iii)] In the left column, $\pi_{23}^*(\xi_\oM)$ and $\pi_{23}^*(\check\xi_\oM)$ differ by a factor $(-1)^{\ep(\be,\ga)}$.
\item[(iv)] In the right column, $(\Phi\t\id_\oM)^*(\xi_\oM)$ and $(\Phi\t\id_\oM)^*(\check\xi_\oM)$ differ by a factor $(-1)^{\ep(\al+\be,\ga)}$, as $\Phi$ maps $\al\t\be\ra\al+\be$.
\item[(v)] In the bottom row, $(\id_\oM\t\Phi)^*(\xi_\oM),(\id_\oM\t\Phi)^*(\check\xi_\oM)$ differ by~$(-1)^{\ep(\al,\be+\ga)}$.
\end{itemize}
Combining (i)--(v), we see \eq{od3eq27} for $\xi_\oM$ commutes on $\al\t\be\t\ga$ if and only if
\begin{equation*}
\d\ep(\al,\be,\ga)=\ep(\be,\ga)-\ep(\al+\be,\ga)+\ep(\al,\be+\ga)-\ep(\al,\be) =\ze(\al,\be,\ga)
\end{equation*}
in $\Z_2$, using \eq{od6eq1}, where the signs do not matter as $1=-1$ in $\Z_2$. This shows:

\begin{lem}
\label{od6lem2}
In Lemma\/ {\rm\ref{od6lem1},}\/ $(K_\oM^{1/2},\xi_\oM)$ is a strong spin structure on\/ $\oM$ compatible with direct sums if and only if\/ $\ep$ in\/ $C^2(K_0^\semi(X),\Z_2)$ and\/ $\ze$ in $C^3(K_0^\semi(X),\Z_2)$ satisfy\/ $\d\ep=\ze,$ for\/ $\d$ as in \eq{od6eq1}.
\end{lem}

Next we prove:

\begin{lem}
\label{od6lem3}
In the situation above,\/ $\d\ze=0$ in\/ $C^4(K_0^\semi(X),\Z_2)$.
\end{lem}

\begin{proof} Let $\al,\be,\ga,\de\in K_0^\semi(X)$. We must show that
\e
\ze(\be,\ga,\de)-\ze(\al+\be,\ga,\de)+\ze(\al,\be+\ga,\de)-\ze(\al,\be,\ga+\de)+\ze(\al,\be,\ga)=0.
\label{od6eq2}
\e
We will do this using a diagram of morphisms of line bundles on the connected component $\al\t\be\t\ga\t\de$ of $\oM\t\oM\t\oM\t\oM$, a diagram made up of five copies of pullbacks of \eq{od3eq27}. Unfortunately this diagram is too large to fit on the page if we include all the information, so we represent it schematically as follows:
\e
\begin{gathered}
\xymatrix@!0@C=48pt@R=35pt{
*+[r]{J_\al\ot J_\be\ot J_\ga\ot J_\de} \ar[ddd]^(0.65){\check\xi_{\al,\be}\ot\id_{J_\ga}\ot\id_{J_\de}} \ar[rrrrrr]^(0.5){\id_{J_\al}\ot\id_{J_\be}\ot\check\xi_{\ga,\de}} \ar[drr]^(0.65){\id_{J_\al}\ot\check\xi_{\be,\ga}\ot\id_{J_\de}} &&& \ar@<-3.6ex>@{}[d]^(0.4){\ze(\be,\ga,\de)} &&& *+[l]{J_\al\ot J_\be\ot J_{\ga+\de}} \ar[ddd]_(0.65){\check\xi_{\al,\be}\ot\id_{J_{\ga+\de}}} \ar[dll]_(0.65){\id_{J_\al}\ot\check\xi_{\be,\ga+\de}}
\\
& \ar@<-8.5ex>@{}[d]^(.1){\ze(\al,\be,\ga)} & \!\!\!\!\!\!\! J_\al\!\ot\! J_{\be+\ga}\!\ot\! J_\de \ar@<-.8ex>[d]_{\check\xi_{\al,\be+\ga}\ot\id_{J_\de}} \ar[rr]^{\id_{J_\al}\ot\check\xi_{\be+\ga,\de}} & \ar@<-4.5ex>@{}[d]^{\ze(\al,\be+\ga,\de)} & J_\al\!\ot\! J_{\be+\ga+\de}\!\!\!\!\!\!\!\! \ar@<.8ex>[d]^{\check\xi_{\al,\be+\ga+\de}} & \ar@<-.5ex>@{}[d]^(0.1){\ze(\al,\be,\ga+\de)} 
\\
&& \!\!\!\!\!\!\!J_{\al+\be+\ga}\!\ot\! J_\de  \ar[rr]_{\check\xi_{\al+\be+\ga,\de}} & \ar@<-4.5ex>@{}[d]^(0.65){\ze(\al+\be,\ga,\de)} & J_{\al+\be+\ga+\de}\!\!\!\!\!\!\! &
\\
*+[r]{J_{\al+\be}\ot J_\ga\ot J_\de} \ar[urr]_(0.7){\check\xi_{\al+\be,\ga}\ot\id_{J_\de}} \ar[rrrrrr]^(0.7){\id_{J_{\al+\be}}\ot\check\xi_{\ga,\de}} &&&&&& *+[l]{J_{\al+\be}\ot J_{\ga+\de}.\!} \ar[ull]_(0.4){\check\xi_{\al+\be,\ga+\de}} }\!\!\!\!\!
\end{gathered}
\label{od6eq3}
\e

Here is what \eq{od6eq3} means. It is a diagram of line bundles on $\al\t\be\t\ga\t\de$, and isomorphisms between them. For the objects, by $J_\al$ we mean the pullback $\pi_1^*(K_\oM^{1/2}\vert_\al)$ to $\al\t\be\t\ga\t\de$ of the line bundle $K_\oM^{1/2}\vert_\al$ on $\al\subset\oM$. By $J_{\al+\be}$ we mean the pullback of $K_\oM^{1/2}\vert_{\al+\be}$ on $\al+\be\subset\oM$ by the morphism $\Phi\ci(\pi_1,\pi_2):\al\t\be\t\ga\t\de\ra\al+\be$. Other line bundles $J_{\cdots}$ are similar. There should also be tensor products with many copies of line bundles like $\pi_{\al,\be}^*(L_\oM)$, as in \eq{od3eq27}, but we have omitted these for clarity.  
 
For the morphisms, by $\check\xi_{\al,\be}:J_\al\ot J_\be\ra J_{\al+\be}$ we mean the pullback to $\al\t\be\t\ga\t\de$ of the isomorphism from \eq{od3eq23}
\begin{equation*}
\check\xi_\oM\vert_{\al\t\be}:\pi_1^*(K_\oM^{1/2}\vert_\al)\ot\pi_2^*(K_\oM^{1/2}\vert_\be)\ot L_\oM\vert_{\al\t\be}\longra\Phi_{\al,\be}^*(K_\oM^{1/2}\vert_{\al+\be}),
\end{equation*}
where the line bundle $L_\oM$ is omitted. The other morphisms $\check\xi_{\cdots}$ are similar. We have also omitted pullbacks of $\chi_\M,\psi_\M$ in \eq{od3eq17}--\eq{od3eq18} which appear in~\eq{od3eq27}.

We have also written $\ze(\al,\be,\ga),\ldots,\ze(\al,\be,\ga+\de)$ in the five small quadrilaterals in \eq{od6eq3}. Each is a pullback of \eq{od3eq27} to $\al\t\be\t\ga\t\de$, and the insertion $\ze(\al,\be,\ga),\ldots$ means the quadrilateral commutes up to the sign $(-1)^{\ze(\al,\be,\ga)},\ldots.$

For comparison, the analogue of \eq{od3eq27} written in the notation of \eq{od6eq3} is 
\begin{equation*}
\xymatrix@C=190pt@R=15pt{
*+[r]{J_\al\ot J_\be\ot J_\ga} \ar[d]^{\id_{J_\al}\ot\check\xi_{\be,\ga}} \ar[r]_(0.7){\check\xi_{\al,\be}\ot\id_{J_\ga}} \ar@<-2ex>@{}[dr]^(0.5){\ze(\al,\be,\ga)} & *+[l]{J_{\al+\be}\ot J_\ga} \ar[d]_{\check\xi_{\al+\be,\ga}} \\
*+[r]{J_\al\ot J_{\be+\ga}} \ar[r]^(0.7){\check\xi_{\al,\be+\ga}}  & *+[l]{J_{\al+\be+\ga}.\!} }
\end{equation*}

The outer rectangle of \eq{od6eq3} obviously commutes, as it is applying $\check\xi_{\al,\be}$ to $J_\al\ot J_\be$ and $\check\xi_{\ga,\de}$ to $J_\ga\ot J_\de$ in the two possible orders.  It follows that the product of the five signs $(-1)^{\ze(\al,\be,\ga)},\ldots,(-1)^{\ze(\al,\be,\ga+\de)}$ measuring the failure of the small quadrilaterals to commute must be 1. But this is equivalent to \eq{od6eq2} in $\Z_2$. Since this holds for all $\al,\be,\ga,\de\in K_0^\semi(X)$, the lemma follows.
\end{proof}

We can now deduce Theorem \ref{od3thm3}(a). By Lemma \ref{od6lem3}, there is a class $\Om\bigl([K_\oM^{1/2}]\bigr):=[\ze]$ in $H^3\bigl(K_0^\semi(X),\Z_2\bigr)$. Then $\Om\bigl([K_\oM^{1/2}]\bigr)=0$ if and only if there exists $\ep\in C^2(K_0^\semi(X),\Z_2)$ with $\d\ep=\ze$. But Lemmas \ref{od6lem1} and \ref{od6lem2} show that this is the necessary and sufficient condition for there to exist a strong spin structure on $\oM$ compatible with direct sums extending~$[K_\oM^{1/2}]$.

For part (b), suppose $\Om\bigl([K_\oM^{1/2}]\bigr)=0$. Then Lemmas \ref{od6lem1} and \ref{od6lem2} imply that the set of isomorphism classes of strong spin structures on $\oM$ compatible with direct sums extending $[K_\oM^{1/2}]$, is in 1-1 correspondence with the set of $\ep\in C^2(K_0^\semi(X),\Z_2)$ with $\d\ep=\ze$ modulo the equivalence relation $\ep\sim\ep'$ if the corresponding pairs admit an isomorphism $\ka:(K_\oM^{1/2},\xi_\oM)\ra(K_\oM^{1/2},\xi'_\oM)$ in the sense of Definition \ref{od3def3}. Then $\ka:K_\oM^{1/2}\ra K_\oM^{1/2}$ is an isomorphism on $\oM$ with $\ka\ot\ka=\id$ by the first diagram in \eq{od3eq28}, so $\ka=\pm\id$ locally on $\oM$. 

Hence there exists a unique $\de:K_0^\semi(X)\ra\Z_2$ with $\ka\vert_\al=(-1)^{\de(\al)}\cdot\id_{K_\oM^{1/2}}\vert_\al$ for all $\al\in K_0^\semi(X)$. From the second diagram in \eq{od3eq28} restricted to~$\al\t\be$
\begin{equation*}
\xymatrix@C=208pt@R=20pt{
*+[r]{\pi_1^*(K_\oM^{1/2}\vert_\al)\ot\pi_2^*(K_\oM^{1/2}\vert_\be)\ot L_\oM\vert_{\al\t\be}} \ar[d]^{\pi_1^*(\ka\vert_\al)\ot\pi_2^*(\ka\vert_\be)\ot\id=(-1)^{\de(\al)}(-1)^{\de(\be)}\id}\ar[r]_(0.65){\xi_\oM\vert_{\al\t\be}=(-1)^{\ep(\al,\be)}\cdot\check\xi_\oM\vert_{\al\t\be}} & *+[l]{\Phi^*(K_\oM^{1/2})\vert_{\al\t\be}} \ar[d]_{\Phi^*(\ka\vert_{\al+\be})=(-1)^{\de(\al+\be)}\id} \\
*+[r]{\pi_1^*(K_\oM^{1/2}\vert_\al)\ot\pi_2^*(K_\oM^{1/2}\vert_\be)\ot L_\oM\vert_{\al\t\be}} \ar[r]^(0.65){\xi'_\oM\vert_{\al\t\be}=(-1)^{\ep'(\al,\be)}\cdot\check\xi_\oM\vert_{\al\t\be}} & *+[l]{\Phi^*(K_\oM^{1/2})\vert_{\al\t\be}} }
\end{equation*}
we see that $\ep'(\al,\be)=\ep(\al,\be)+\de(\be)-\de(\al+\be)+\de(\al)$. Thus $\ep'=\ep+\d\de$ for some $\de\in C^1(K_0^\semi(X),\Z_2)$. So the set of isomorphism classes of such strong spin structures is in bijection with 
\begin{equation*}
\frac{\bigl\{\ep\in C^2(K_0^\semi(X),\Z_2):\d\ep=\ze\bigr\}}{\ep+\d\de\sim\ep,\;\> \de\in C^1(K_0^\semi(X),\Z_2)}.
\end{equation*}
This is a torsor over $H^2\bigl(K_0^\semi(X),\Z_2\bigr)$, proving Theorem~\ref{od3thm3}(b).

Finally, taking $\xi'_\oM=\xi_\oM$ and $\ep'=\ep$ above, the group of automorphisms $\ka$ of $(K_\oM^{1/2},\xi_\oM)$ is identified with the subgroup of $\de\in C^1(K_0^\semi(X),\Z_2)$ satisfying $\ep+\d\de=\ep$, that is, $\d\de=0$. Since $\d\bigl(C^0(K_0^\semi(X),\Z_2)\bigr)=0$, this subgroup is $H^1(K_0^\semi(X),\Z_2)$, proving Theorem~\ref{od3thm3}(c).

\medskip

\noindent{\small\sc The Mathematical Institute, Radcliffe
Observatory Quarter, Woodstock Road, Oxford, OX2 6GG, U.K.

\noindent E-mails: {\tt joyce@maths.ox.ac.uk, 
upmeier@maths.ox.ac.uk.}}

\end{document}